\documentclass[12pt, a4paper]{amsart}
\usepackage{amscd, amssymb, pb-diagram}
\usepackage{mathtools}
\usepackage{enumitem}
\usepackage{amsmath}
\usepackage[margin=2.9cm]{geometry}
\usepackage{amscd, amssymb, pb-diagram}
\usepackage{enumitem}
\usepackage{tikz-cd}
\usepackage{mathtools}
\usepackage{amsmath}
\usepackage{amsthm} 
\usepackage[nospace,noadjust]{cite}
\usepackage{lmodern}
\usepackage{graphicx}
\graphicspath{{images/}{../images/}}
\usepackage[toc,page]{appendix}
\usepackage[normalem]{ulem}
\usepackage{emptypage}
\usepackage{hyperref}
\usepackage{subfiles}
\usepackage{datetime}

\usepackage{tensor}
\usepackage{wrapfig}
\usepackage{tikz}

\allowdisplaybreaks

\def\sup{\operatorname{sup}}
\def\max{\operatorname{max}}
\def\min{\operatorname{min}}

\def\C{\mathbb{C}}

\def\R{\mathbb{R}}
\def\N{\mathbb{N}}
\def\Z{\mathbb{Z}}

\def\AA{\mathcal{A}}

\def\EE{\mathcal{E}}

\def\LL{\mathcal{L}}

\def\OO{\mathcal{O}}
\def\PP{\mathcal{P}}

\def\UU{\mathcal{U}}
\def\VV{\mathcal{V}}

\def\ZZ{\mathcal{Z}}

\def\hh{\mathfrak{H}}
\def\kk{\mathfrak{K}}

\newcommand{\Rpunc}{R_\textnormal{punc}}
\newcommand{\Opunc}{\Omega_\textnormal{punc}}
\newcommand{\Apunc}{A_\textnormal{punc}}
\newcommand{\Implies}[2]{$\text{\ref{#1}}\implies\text{\ref{#2}}$}

\newtheorem{thm}{Theorem}[section]
\newtheorem{cor}[thm]{Corollary}

\newtheorem{lemma}[thm]{Lemma}
\newtheorem{prop}[thm]{Proposition}
\newtheorem{thm1}{Theorem}

\theoremstyle{definition}
\newtheorem{definition}[thm]{Definition}
\newtheorem{notation}[thm]{Notation}
\theoremstyle{remark}
\newtheorem{remark}[thm]{Remark}

\newtheorem*{Acknowledgements}{Acknowledgements}
\numberwithin{equation}{section}

\tikzstyle{vertex}=[circle]
\tikzstyle{goto}=[->,shorten >=1pt,>=stealth,semithick]

\begin{document}

\title[Classification of tiling $C^*$-algebras]{Classification of tiling $C^*$-algebras}
\author[Luke J. Ito]{Luke J. Ito}
\address{Luke J. Ito, Michael F. Whittaker and Joachim Zacharias, School of Mathematics and Statistics, University of Glasgow, University Place, Glasgow Q12 8QQ, United Kingdom}
\email{l.ito.1@research.gla.ac.uk}

\author[Michael F. Whittaker]{Michael F. Whittaker}
\email{Mike.Whittaker@glasgow.ac.uk}

\author[Joachim Zacharias]{Joachim Zacharias}
\email{Joachim.Zacharias@glasgow.ac.uk}

 \thanks{This research was partially supported by EPSRC grants EP/R013691/1, EP/R025061/1, EP/N509668/1 and EP/M506539/1.}

\begin{abstract}
We prove that Kellendonk's $C^*$-algebra of an aperiodic and repetitive tiling with finite local complexity is classifiable by the Elliott invariant. Our result follows from showing that tiling $C^*$-algebras are $\ZZ$-stable, and hence have finite nuclear dimension. To prove $\ZZ$-stability, we extend Matui's notion of almost finiteness to the setting of \'etale groupoid actions following the footsteps of Kerr. To use some of Kerr's techniques we have developed a version of the Ornstein-Weiss quasitiling theorem for general \'etale groupoids.
\end{abstract}

\maketitle

\section{Introduction}\label{sec:intro}

Due to the recent completion of the classification program for simple $C^*$-algebras with finite nuclear dimension \cite{EGLN,GLN,TWW}, the time is ripe to look for naturally occurring examples to which this program applies. In this paper, we consider $C^*$-algebras associated to repetitive aperiodic tilings with finite local complexity (FLC). A tiling is a covering of $\R^d$ by a collection of labelled compact subsets, called tiles, that only meet on their boundary. Aperiodic tilings give rise to two $C^*$-algebras: a crossed product $C(\Omega) \rtimes \R^d$ and a quantised algebra that arises from an abstract transversal to the translation action on the tiling space $\Omega$. The latter algebra is due to Kellendonk \cite{Kel}, and was developed to give a new picture (up to stable isomorphism) of algebras studied by Bellissard in order to understand the spectrum of Sch\"odinger operators with aperiodic potential \cite{Bel}.

Tiling algebras are particularly well-suited to the classification program since their Elliott invariant is well understood \cite{AP, Rob, Starling, KP, Kel, FHK, GonRS}. Our strategy is in line with the most recent techniques developed for the classification of $C^*$-algebras coming from minimal actions of general amenable groups acting on compact metric spaces, studied by Kerr and others \cite{Kerr, CJKMSTD}. The main result of this paper is the following.

\begin{thm1}[c.f. Theorem \ref{tilingZstable}]
The class of $C^*$-algebras consisting of those which are associated to any aperiodic and repetitive tiling with FLC is classified by the Elliott invariant.
\end{thm1}

The classification of tiling algebras is highly relevant to mathematical physics. Bellissard showed that gaps in the spectrum of Schr\"odinger operators with aperiodic potential are topological in nature in the sense that they are unchanged under perturbation of the Hamiltonian \cite{Bel}. In particular, he proved that these gaps are labelled by the range of the trace of K-theory elements of a crossed product tiling $C^*$-algebra that is strongly Morita equivalent to Kellendonk's algebra. Standard results from the classification program for simple $C^*$-algebras then imply that the isomorphism class of the $C^*$-algebra Bellissard associated to a Schr\"odinger operator is completely determined by the operator itself.

Before discussing the specific techniques used in this paper, we wish to put our methods and results into the context of the classification program for simple nuclear $C^*$-algebras. The classification program may be regarded as a $C^*$-analogue of the Connes-Haagerup classification of injective factors. It has its origin in Elliott's classification of AF-algebras by K-theory from the 1970s \cite{Ell} but was put forward as a conjecture only in the late 1980s \cite{Ell2}. Whilst initially regarded as  very speculative, an outburst of subsequent intensive research produced a great deal of positive evidence. Around 1995, Kirchberg and Phillips proved the Elliott conjecture for simple nuclear purely infinite $C^*$-algebras (now known as Kirchberg algebras) satisfying the UCT \cite{Phil97}. A key role in this classification was played by Kirchberg's absorption result \cite{KirchbergPhillips, Kirchberg}, which says that $A$ is a Kirchberg algebra if and only if $A \otimes \OO_{\infty} \cong A$.  In 2003, R{\o}rdam \cite{Rordam} proved that there are infinite simple nuclear $C^*$-algebras that fail the Elliott conjecture. Subsequently, Toms discovered examples with no infinite projections that the Elliott invariant cannot classify \cite{Toms1, Toms2}.

Unlike in the von Neumann algebra case, where there are only very few injective factors with traces - one finite and one semifinite such factor - the classification of simple nuclear $C^*$-algebras with traces poses a major challenge. Inspired by Kirchberg's absorption result, Jiang and Su constructed a simple unital nuclear $C^*$-algebra $\ZZ$ with unique trace which is KK-equivalent to the complex numbers \cite{JS}. They proved that $A \otimes \ZZ \cong A$ for simple AF-algebras and this was subsequently extended to large classes of other examples. Since $A$ and $A \otimes \ZZ$ have the same Elliott invariants (under mild assumptions), $\ZZ$-stability appears as a minimal classifiability requirement. It was becoming more and more clear that a new regularity hypothesis was necessary to continue Elliott's classification program.

The necessary regularity condition came in 2010 when Winter and the third named author introduced nuclear dimension in \cite{WZ} (thanks to Winter's handwriting, it was also known as unclear dimension). This is a non-commutative dimension theory which builds on previous definitions by Kirchberg and Winter but which is finite for purely infinite simple $C^*$-algebras. It turned out that  simple $C^*$-algebras without traces have finite nuclear dimension if and only if they are purely infinite. Hence, finiteness of nuclear dimension captures the correct classifiability condition for infinite $C^*$-algebras. Moreover, Winter proved that finite nuclear dimension for simple $C^*$-algebras implies $\ZZ$-stablity, as well as far reaching regularity properties of their Cuntz semigroup \cite{Win}. Shortly afterwards, Toms and Winter conjectured that for simple nuclear $C^*$-algebras, finiteness of nuclear dimension, $\ZZ$-stablity, and strict comparison, a regularity condition for the Cuntz semigroup of the algebra, should all be equivalent. 

After many partial results, it is now known by the fundamental work of  \cite{CETWW} that  finiteness of nuclear dimension and $\ZZ$-stablity are indeed equivalent. To establish classifiability of concrete examples one is therefore free to either to find an upper bound for the nuclear dimension or to show $\ZZ$-stablity. Whilst it initially appeared difficult to prove $\ZZ$-stablity directly, there are now relatively easy methods to do so thanks to Hirshberg and Orovitz's application \cite{HO} of the breakthroughs of Matui and Sato \cite{MS2}.

Kellendonk's tiling algebras are defined using a principal \'etale groupoid with compact and totally disconnected unit space.  The notion of almost finiteness for second countable \'etale groupoids with compact and totally disconnected unit spaces was introduced by Matui in \cite{Matui}. The definition was designed as a weakening of the AF (approximately finite) property of groupoids, motivated by the fact that AF \'etale groupoids with totally disconnected unit spaces were already known to be completely classified up to isomorphism \cite{Ren}. In \cite{Suzuki}, Suzuki extended the definition to general \'etale groupoids.

Kerr introduced the notion of almost finiteness for free actions of groups on compact metric spaces in \cite{Kerr}. This complements Matui's original definition in the sense that the transformation groupoids arising from such actions are almost finite. The motivation for the definition in the group case is to provide a dynamical substitute for the properties of $\ZZ$-stability from the topological setting and hyperfiniteness from the measure-theoretic setting. It is shown in \cite{Kerr} that the definition accomplishes this goal; almost finite, free, minimal actions of infinite groups on compact metrisable spaces of finite covering dimension have $\ZZ$-stable crossed product $C^*$-algebras. Therefore, almost finiteness should be viewed as a dynamical analogue of $\ZZ$-stability, at least for amenable groups acting on compact metric spaces.

It then becomes natural to ask to what extent the methods of \cite{Kerr} apply to actions of groupoids. The present work explores this question, generalising the foundations behind the classification result in \cite{Kerr} to the \'etale groupoid case and applying the new machinery to the groupoid arising from an aperiodic, repetitive tiling satisfying FLC. As part of this development, we extend the Ornstein-Weiss quasitiling theorem to \'etale groupoids, which could play into a more general $\ZZ$-stability result for \'etale groupoids.

Classification of tiling algebras was conjectured in \cite{Phillips}, and several subclasses of tiling algebras were subsequently shown to be classifiable. Our approach extends and unifies these previous results, which we now outline. The most general classification result prior to ours classifies all \textit{rational} tilings, those whose tiles are polyhedra with all vertices lying on rational coordinates. For such tilings, \cite[Lemma~5]{SW} implies that Kellendonk's algebra is isomorphic to $C(\Omega_{\textnormal{sq}})\rtimes\Z^d$, where $\Omega_{\textnormal{sq}}$ is the hull of a tiling where all prototiles are cubes. Such $C^*$-algebras were classified by Winter in \cite[Corollary~3.2]{Win1} using the nuclear dimension estimates of Szab\'o \cite[Theorem~5.3]{Sza}. Subsequently, Deeley and Strung \cite{DS} proved that the stable and unstable algebras of Smale spaces containing projections are classifiable using the dynamic asymptotic dimension of Guentner, Willett and Yu \cite{GWY}. Such Smale spaces include all substitution tilings satisfying our standard hypotheses. We note that the work in this paper was originally inspired by Deeley and Strung's results. We had hoped that Deeley and Strung's techniques would extend to all aperiodic tilings, however, we were not able to prove that certain subgroupoids along so called ``fault lines" in tilings were relatively compact. Finally, \cite[Example 9.6]{HSWW} uses a Rohklin dimension argument to prove that all one-dimensional tilings with our standard hypotheses are classifiable.

The paper is organised as follows. In Section~\ref{sec:afgroupoidactions}, we present a few preliminary facts about groupoids, and establish notation. We then explore groupoid actions, and introduce a new, but associated, notion of almost finiteness. This generalises Kerr's almost finiteness from the group case, and links back to Matui's original definition for groupoids \cite{Matui}. In section~\ref{sec:groupoidOW}, we prove an analogue of the Ornstein-Weiss quasitiling theorem for \'etale groupoids. In Section~\ref{sec:gpoidcrossedproducts}, we review the theory of crossed products arising from groupoid actions. Sections~\ref{tilings} and \ref{efun} contain a brief introduction to tilings, and the construction of tiling groupoids and their $C^*$-algebras. We show that these tiling groupoids are almost finite (in the sense of \cite{Matui}) in Section~\ref{sec:tilinggroupoidaf}. Following the results of Section~\ref{sec:afgroupoidactions}, we are then able to show that tiling groupoids arise from almost finite groupoid actions, and we spend some time constructing the associated structure concretely. Section~\ref{sec:tilingZstable} contains our main result on the $\ZZ$-stability of tiling $C^*$-algebras. We give a direct proof that tiling $C^*$-algebras are quasidiagonal in Section~\ref{sec:quasidiagonal}, along with some examples of tilings whose $C^*$-algebras have unique trace, and so benefit from the more direct route to classification that this property provides. 

We note that the results in this paper made up the bulk of the first author's PhD thesis \cite{Ito}, which also contains extensive background information on various aspects of this article.

\begin{Acknowledgements}
We are grateful to Charles Starling, Gabor Szab\'o and Stuart White for their helpful comments and mathematical insights.
\end{Acknowledgements}

\section{Almost finiteness for groupoid actions}\label{sec:afgroupoidactions}

For an introduction to \'etale groupoids and their $C^*$-algebras, see \cite{Sims}. 
We typically denote groupoids by $G$, with unit space $G^{(0)}$ and $r,s:G\rightarrow G^{(0)}$ the range and source maps of $G$. A subset $V$ of a groupoid $G$ is said to be a \textit{$G$-set} if both the range and source maps are injective on $V$. 
For $A\subset G$ and $g\in G$, let
\[Ag\coloneqq\{ag\mid s(a)=r(g)\},\]
so that if $r(g)\notin s(A)$, we obtain $Ag=\emptyset$. Similarly, for subsets $A,B\subset G$, we denote
\[AB\coloneqq\{ab\mid a\in A, b\in B, s(a)=r(b)\}.\]
This will be particularly useful when we consider $A\subset G$ and $u\in G^{(0)}$, and obtain the simplified notation
\[Au:=A\cap G_u:=\{a\in A\mid s(a)=u\}\]
and
\[uA:=A\cap G^u:=\{a\in A\mid r(a)=u\}.\]
A topological groupoid is \textit{\'etale} if the range and source maps are local homeomorphisms, and is \textit{ample} if it is \'etale and $G^{(0)}$ is zero dimensional. A \textit{bisection} is a subset $B$ of a topological groupoid such that there exists an open set $U$ containing $B$ such that $r:U\rightarrow r(U)$ and $s:U\rightarrow s(U)$ are homeomorphisms. 
Since the topology on any \'etale groupoid has a base of open bisections, when $G$ is \'etale and $A\subset G$ is compact, there exists $N\in\N$ such that $|Au|<N$ for all $u\in G^{(0)}$. We will use this fact frequently.

We introduce the notion of almost finiteness for actions of groupoids, generalising the definition for group actions \cite[Definition~8.2]{Kerr}. We first make precise what it means for a groupoid to act on a set. To the best of our knowledge, the following notion first appears in \cite{Ehresmann}. The exact formulation presented here appears as \cite[Definition~1.55]{Goehle}.
\begin{definition}\label{gpoidaction}
Let $G$ be a groupoid and $X$ a set. We say that $G$ \textit{acts (on the left) of} $X$ if there is a surjection $r_X:X\rightarrow G^{(0)}$ and a map $(g,x)\mapsto g\cdot x$ from $G\ast X\coloneqq\{(g,x)\in G\times X\mid s(g)=r_X(x)\}$ to $X$ such that
\begin{enumerate}[label=(\roman*)]
\item if $(h,x)\in G\ast X$ and $(g,h)\in G^{(2)}$, then $(g, h\cdot x)\in G\ast X$ and
\[g\cdot(h\cdot x)=gh\cdot x;\textnormal{ and}\]
\item $r_X(x)\cdot x=x$ for all $x\in X$. 
\end{enumerate}
\end{definition}
We will usually drop the dot from the notation and refer to the image of $x\in X$ under $g\in G$ by $gx$.  We will preserve the dot in the case that $X=G^{(0)}$, to distinguish between groupoid multiplication $gx$, and the action $g\cdot x$. For $W\subset X$ and $S\subset G$, we denote
\[S\cdot W\coloneqq\{g\cdot x\mid g\in S,\ x\in W\textnormal{ and }s(g)=r_X(x)\}.\]
Again, we may drop the dot to denote this by $SW$. 

We make the following simple observation.
\begin{lemma}\label{range gx}
Let $G\curvearrowright X$ be any groupoid action. If $(g,x)\in G\ast X$, then $r_X(gx)=r(g)$.
\end{lemma}
\begin{proof}
Since $(g,x)\in G\ast X$ and $(g^{-1},g)\in G^{(2)}$, we have $(g^{-1},gx)\in G\ast X$ by condition (i) in Definition~\ref{gpoidaction}, so that $r(g)=s(g^{-1})=r_X(gx)$. 
\end{proof}
\begin{definition}
Let $G$ be a topological groupoid which acts on a topological space $X$.
\begin{enumerate}[label=(\roman*)]
\item (\cite[Definition~1.60]{Goehle}) We say that the action is \textit{continuous} if the maps $r_X:X\rightarrow G^{(0)}$ and $(g,x)\mapsto gx$ from $G\ast X\rightarrow X$ are continuous. 
\item We say that the action is \textit{free} if, for every $x\in X$, $gx=x$ implies $g=r_X(x)$.
\item We say that the action is \textit{minimal} if, for every $x\in X$, the subset $\{gx\mid g\in G\}$ is dense in $X$. 
\end{enumerate}
\end{definition}
We remark that when $G$ acts continuously on a compact space $X$, then $G^{(0)}$ arises as the continuous image of the compact space $X$ under the surjection $r_X:X\rightarrow G^{(0)}$. It is therefore necessary that $G^{(0)}$ be compact in this situation. 
\begin{lemma}\label{etalestrongcontinuity}
Let $G$ be a locally compact Hausdorff \'etale groupoid acting continuously on a locally compact Hausdorff space $X$. Then the range map $r_X:X\rightarrow G^{(0)}$ is open.
\end{lemma}
\begin{proof}
By \cite[Proposition~1.72]{Goehle}, $G\ltimes X$ is \'etale. Suppose that $U\subset X$ is open. We identify $U\subset X$ with $V=\{(r_X(u),u)\mid u\in U\}\subset (G\ltimes X)^{(0)}$. Since $U$ is open in $X$ and $X\cong(G\ltimes X)^{(0)}$ is open in $G\ltimes X$, it follows that $V\cong U$ is open in $G\ltimes X$. Observe that $r_X(U)=\pi_G(V)$, where $\pi_G$ is projection to the $G$-coordinate on $G\times X$. Since $\pi_G$ is an open map, this shows that $r_X(U)$ is open in $G$.
\end{proof}
The following notion of dynamical comparison is a generalisation of \cite[Definition~3.1]{Kerr} to the groupoid setting. 

\begin{definition}
Let $G\curvearrowright X$ be a groupoid action, and let $A,B\subset X$. We write $A\prec B$ if, for every closed $C\subset A$, there exist a finite collection $\UU$ of open subsets of $X$ which cover $C$, and a subset $S_U\subset G$ for each $U\in\UU$ such that $r_X(U)\subset s(S_U)$ so that the collection $\{tU\mid U\in\UU, t\in S_U\}$ consists of pairwise disjoint subsets of $B$.  
\end{definition}

We now present a construction of a groupoid from a groupoid action. Such groupoids will be vitally important later in the paper.
\begin{definition}
Let $G$ be a groupoid acting on a set $X$. The associated \textit{transformation groupoid} $G\ltimes X$ is the set $G\ast X$ with the following structure. The set of composable pairs is
\[(G\ltimes X)^{(2)}=\{((g,x),(h,y))\in(G\ltimes X)\times(G\ltimes X)\mid h\cdot y=x\}.\]
The product of such a pair is given by $(g,x)(h,y)=(gh, y)$. The inverse operation is $(g,x)^{-1}=(g^{-1},gx)$. 

In the case that $G$ is a topological groupoid acting on a topological space $X$, $G\ltimes X=G\ast X$ inherits the relative topology from the product topology on $G\times X$.
\end{definition}
We can compute the missing structure maps as follows. The range and source of $(g,x)\in G\ltimes X$ are given by
\[s(g,x)=(g,x)^{-1}(g,x)=(g^{-1},gx)(g,x)=(g^{-1}g,x)=(s(g),x)=(r_X(x),x)\]
and
\[r(g,x)=(g,x)(g,x)^{-1}=(gg^{-1},gx)=(r(g),gx)=(r_X(gx),gx).\]
Hence, we see that $(G\ltimes X)^{(0)}$ can be naturally identified with $X$ using the map $(r_X(x),x)\mapsto x$. Under this identification, $s(g,x)=x$ and $r(g,x)=gx$. From all this, it is easy to see that when $G$ is a group, we recover the usual notion of a transformation groupoid.

The following observations about the topology of groupoid actions will be useful.
\begin{lemma}\label{actiontopology}
Let $G$ be a locally compact Hausdorff groupoid which acts continuously on a locally compact Hausdorff space $X$. Then
\begin{enumerate}[label=(\roman*)]
\item if $W\subset X$ and $S\subset G$ are compact, then $S\cdot W$ is compact in $X$;
\item \cite[Proposition 1.72]{Goehle} if $G$ is \'etale, then $G\ltimes X$ is \'etale; and
\item if $G$ is \'etale and $W\subset X$ and $S\subset G$ are open, then $S\cdot W$ is open in $X$.
\end{enumerate}
\end{lemma}
\begin{proof}
For (i), observe that $S\cdot W$ is the image of the compact subset $(S\times W)\cap (G\ast X)\subset G\ast X$ under the continuous action map $(g,x)\mapsto gx$, and is therefore compact in $X$.

For (ii), it is shown in \cite{Ren} that $G$ is \'etale if and only if it admits a Haar system and the range and source maps are open. In this case, by \cite[Proposition 1.72]{Goehle}, the range and source of $G\ltimes X$ are also open and $G\ltimes X$ admits a Haar system, so that $G\ltimes X$ is \'etale.

For (iii), as above, the range map of $G\ltimes X$ is open. Then $S\cdot W=r((S\times W)\cap(G\ltimes X))$ is open in $(G\ltimes X)^{(0)}\cong X$.\qedhere
\end{proof}

The notion of approximate invariance is of central importance to notions of almost finiteness, and there are many formulations for approximate invariance of subsets of groups (see \cite[Chapter~4]{Kerr-Li}, in particular Definition~4.32, and the paragraph immediately following Definition~4.34). All of these formulations are equivalent in the sense that if a sequence of subsets becomes arbitrarily approximately invariant in one formulation, then it does so in all of them. We extend this situation to groupoids.
\begin{lemma}\label{equivgpoidapproxinvce}
Let $G$ be an \'etale groupoid. Let $C,A \subset G$ be compact and nonempty, and let $\epsilon>0$. Consider the following conditions.
\begin{enumerate}[label=(\roman*), ref=(\roman*)]
\item (\cite[Definition~3.1]{Suzuki}; c.f. \cite[Definition~6.2]{Matui}) For every $u\in G^{(0)}$, \[\frac{|CAu\setminus Au|}{|Au|}<\epsilon.\]					\label{gpoidinvariant1}
\item For every $u\in G^{(0)}$, 
\[|\{a\in Au\mid Ca\subset Au\}|>(1-\epsilon)|Au|.\]										\label{gpoidinvariant2}
\end{enumerate}
Then (i) and (ii) are equivalent in the sense that if $\{A_n\}_{n\in\N}$ is a sequence of subsets such that for any compact $C\subset G$ and $\epsilon>0$, there exists $N\in\N$ such that $A_n$ satisfies either (i) or (ii) for $C$ and $\epsilon$ whenever $n\geq N$, then there exists $M\in\N$ such that $A_n$ satisfies the other condition for $C$ and $\epsilon$ whenever $n\geq M$. If $A\subset G$ satisfies either condition, we will say that $A$ is \textnormal{$(C,\epsilon)$-invariant}. 
\end{lemma}
\begin{remark}
It is important to note that $(C,\epsilon)$-invariance of a \textit{single} set $A$ in the sense of (i) is \textit{not} equivalent to $(C,\epsilon)$-invariance of $A$ in the sense of (ii). Indeed, it is easy to construct examples to show that neither implication holds. Therefore, our use of this terminology is rather imprecise. This is justified by the fact that for our purposes we always only care about \textit{arbitrary} approximate invariance. Still, we endeavour to make clear which defiition we are using at any time.
\end{remark}
\begin{proof}[Proof of Lemma~\ref{equivgpoidapproxinvce}]
First, notice that if $A,C\subset G$ are compact and $u\in G^{(0)}$, then
\[|\{a\in Au\mid Ca\subset Au\}|=|Au|-|\{a\in Au\mid Ca\not\subset Au\}|,\]
so that condition~\ref{gpoidinvariant2} is equivalent to
\[|\{a\in Au\mid Ca\not\subset Au\}|<\epsilon|Au|.\]
Therefore, it will suffice to find constants $c_1,c_2>0$ which depend only on $C$ such that
\[\frac{1}{c_1}|CAu\setminus Au|\leq|\{a\in Au\mid Ca\not\subset Au\}|\leq c_2|CAu\setminus Au|.\]

\Implies{gpoidinvariant1}{gpoidinvariant2}: Fix $A, C\subset G$ compact and $\epsilon>0$, and suppose that $A$ is $(C,\epsilon)$-invariant in the sense of condition~\ref{gpoidinvariant1}, so that 
\[|CAu\setminus Au|<\epsilon|Au|.\]
Set $c_2=\sup_{v\in G^{(0)}}|vC|$, which is finite as $C$ is compact. Notice that it is possible for distinct $a_1,a_2\in\{a\in Au\mid Ca\not\subset Au\}$ to admit $k_1,k_2\in C$ such that $k_1a_1=k_2a_2\in CAu\setminus Au$, and that this requires that $k_1\neq k_2$ are distinct, but share a range. Therefore this situation can occur for at most $c_2$ distinct elements $a_1,\ldots,a_{c_2}$. In other words, any map from $\{a\in Au\mid Ca\not\subset Au\}$ to $CAu\setminus Au$ defined by $a\mapsto k_aa$, where $k_a\in C$ is chosen such that $k_aa\notin Au$ is at most $c_2$-to-one. It follows that
\[|\{a\in Au\mid Ca\not\subset Au\}|\leq c_2|CAu\setminus Au|<c_2\epsilon|Au|,\]
so that $A$ is $(C, c_2\epsilon)$-invariant in the sense of condition~\ref{gpoidinvariant2}. 

\Implies{gpoidinvariant2}{gpoidinvariant1}: This time, suppose that $A$ is $(C,\epsilon)$-invariant in the sense of \ref{gpoidinvariant2}, so that
\[|\{a\in Au\mid Ca\not\subset Au\}|<\epsilon|Au|,\]
and let $c_1=\sup_{w\in G^{(0)}}|Cw|>0$, which is again finite since $C$ is compact. Observe that
\begin{align*}
|CAu\setminus Au|&=\left|\bigcup_{w\in r(Au)}\bigcup_{c\in Cw} \{ca\mid a\in wAu, ca\notin Au\}\right|\\
&\leq\sum_{w\in r(Au)}\sum_{c\in Cw} |\{cau\mid a\in wAu, ca\notin Au\}|.
\end{align*}
Now, for each fixed $w\in r(Au)$, the second sum contributes $|Cw|\leq c_1$ terms. Each of these terms corresponds to some $c\in Cw$, and counts the $a\in wAu$ for which $ca\notin Au$. Therefore, if $a\in wAu$ is such that $Ca\subset Au$, it will never be counted in this sum, whereas if $a\in wAu$ has $Ca\not\subset Au$, it can be counted at most $|Cw|\leq c_1$ times. Thus, we obtain
\begin{align*}
|CAu\setminus Au|&\leq\sum_{w\in r(Au)}\sum_{c\in Cw} |\{cau\mid a\in wAu, ca\notin Au\}|\\
&\leq\sum_{w\in r(Au)}|Cw||\{a\in wAu\mid Ca\not\subset Au\}|\\
&\leq c_1\sum_{w\in r(Au)}|\{a\in wAu\mid Ca\not\subset Au\}|\\
&= c_1\left|\bigsqcup_{w\in r(Au)}\{a\in wAu\mid Ca\not\subset Au\}\right|\\
&= c_1|\{a\in Au\mid Ca\not\subset Au\}|\\
&< c_1\epsilon|Au|.
\end{align*}
showing that $A$ is $(C, c_1\epsilon)$-invariant in the sense of condition~\ref{gpoidinvariant1}.
\end{proof}

The following definition is easily seen to generalise the concepts in \cite{Kerr}.
\begin{definition}\label{defgroupoidactionaf}
Suppose that $G$ is a locally compact Hausdorff \'etale groupoid, and that $X$ is a compact metric space. Let $\alpha:G\curvearrowright X$ be a continuous action. Let $C\subset G$ be compact, and let $\epsilon>0$. 
\begin{itemize}
\item A \textit{tower} of $\alpha$ is a pair $(W,S)$ consisting of a nonempty subset $W\subset X$ and a nonempty compact open subset $S\subset G$ such that $S=\bigsqcup_{j=1}^N S_j$ decomposes into compact open $S_j$ with $s(S_j)=r_X(W)$ for each $j$, so that the range and source maps are injective on each $S_j$, and such that the sets $S_jW$ are pairwise disjoint for $j\in\{1,\ldots,N\}$. We refer to the set $W$ as the \textit{base}, $S$ as the \textit{shape}, and the sets $S_jW$ as the \textit{levels} of the tower.

\item Consider a finite collection $\{(W_1,S_1),\ldots,(W_n,S_n)\}$ of towers such that for each $i\in\{1,\ldots,n\}$, the shape of the $i$-th tower has decomposition $S_i=\bigsqcup_{j=1}^{N_i} S_{i,j}$. Such a sequence is called a \textit{castle} if the collection of all tower levels $\{S_{i,j}W_i\mid i\in\{1,\ldots,n\}, j\in\{1,\ldots,N_i\}\}$ is pairwise disjoint.
The sets $W_i$ are called the \textit{bases}, $S_i$ the \textit{shapes}, and $S_{i,j}W_i$ the \textit{levels} of the castle $\{(W_i,S_i)\}_{i=1}^n$. 
A castle is called a \textit{tower decomposition of $\alpha$} if the levels of the castle partition $X$. 

\item A castle $\{(W_i,S_i)\}_{i=1}^n$ is said to be \textit{$(C,\epsilon)$-invariant} if all of the shapes $S_1,\ldots, S_n$ are $(C,\epsilon)$-invariant subsets of $G$, in the sense defined by condition~\ref{gpoidinvariant2} from Lemma~\ref{equivgpoidapproxinvce}. 

\item A castle $\{(W_i,S_i)\}_{i=1}^n$ is said to be \textit{open} (respectively \textit{closed}, \textit{clopen}) if each $W_i$ is open (respectively closed, clopen) in $X$. 

\item A groupoid action is said to be \textit{almost finite} if, for every $m\in\N$, every compact $C\subset G$, and every $\epsilon>0$, there exist
\begin{enumerate}[label=(\roman*)]
\item an open castle $\{(W_i,S_i)\}_{i=1}^n$ whose shapes are $(C,\epsilon)$-invariant, and whose levels have diameter less than $\epsilon$; and

\item sets $S_i'\subset S_i$ for each $i\in\{1,\ldots,n\}$ such that, for each $u\in G^{(0)}$, $|S_i'u|<|S_iu|/m$, so that $(W_i,S_i')$ is a \textit{subtower} of $(W_i,S_i)$, in the sense that if $S_i=\bigsqcup_{j=1}^{N_i}S_{i,j}$, then there exist $L_i\in\N$ and a subcollection $\{j_{i,1},\ldots,j_{i,L_i}\}\subset\{1,\ldots,N_i\}$ such that $S_i'=\bigsqcup_{l=1}^{L_i}S_{i,j_{i,l}}$, and such that
\[X\setminus\bigsqcup_{i=1}^nS_iW_i\prec\bigsqcup_{i=1}^nS_i'W_i.\]
\end{enumerate}
\end{itemize}
\end{definition}
In a tower $(W,S)$, where $S=\bigsqcup_{j=1}^N S_j$, each set $S_j$ should be thought of as corresponding to a single group element from \cite[Definition~4.1]{Kerr}. The injectivity condition of the source and range together with the condition that $s(S_j)=r_X(W)$ ensure that each element of $W$ has exactly one image under the action of $S_j$. Aside from this modification to ensure compatibility with the fibred structure of the groupoid, the definition is identical in spirit to \cite[Definition~8.2]{Kerr}. 

We show that the diameter condition appearing in the statement is no obstacle in our case of interest. In the group setting, closedness of the towers comes with no loss of generality by \cite[Lemma~10.1]{Kerr}. We conjecture that this should also be true in the groupoid setting, but we do not have a proof, so it appears as an assumption in the following result. 
Another lingering question to consider is whether the castles witnessing almost finiteness can always be chosen to partition $X$ when $X$ is totally disconnected (c.f. \cite[Theorem~10.2]{Kerr}).

\begin{lemma}[c.f. {\cite[Theorem~10.2]{Kerr}}]\label{amplediameter}
Let $G$ be a locally compact Hausdorff \'etale groupoid, $X$ a totally disconnected 
compact metric space, and $\alpha:G\curvearrowright X$ a free and continuous action which admits arbitrarily invariant \textbf{clopen} castles, as in the definition of almost finiteness, but which do not necessarily satisfy the diameter condition. Then $\alpha$ is almost finite. That is, we can choose the invariant clopen castles to satisfy the diameter condition, with no loss of generality. 
\end{lemma}

\begin{proof}
Let $C\subset G$ be compact, let $\epsilon>0$, and let $(W,S)$ be a $(C,\epsilon)$-invariant clopen tower as in the statement of the lemma, so that $S=\bigsqcup_{i=1}^m S_i$, where $S_i$ is compact and open in $G$. We first show that we can refine $(W,S)$ into a clopen castle whose levels have diameter smaller than $\epsilon$. The bases of the towers in the castle will partition $W,$ and the shape of each new tower will simply be the subset of $S$ consisting of arrows which act on its base. Since $S$ is $(C,\epsilon)$-invariant, it will follow that the shape of each new tower is also $(C,\epsilon)$-invariant. 

Consider the first tower level $S_1\cdot W.$ Since $X$ is totally disconnected, it has a basis of clopen sets. Therefore, since $S_1\cdot W$ is compact and open, 
there exists a finite collection of clopen subsets $\{U_{n_1}\}_{n_1=1,\ldots,N}$ of $X$ with diameter smaller than $\epsilon$ such that $S_1\cdot W=\bigcup_{n_1=1}^N U_{n_1}$. 
By replacing $U_{n_1}$ by $U_{n_1}\setminus\bigcup_{j=1}^{n_1-1}U_j$ for each $n_1=1,2,\ldots,N$ in turn, we may assume without loss of generality that the collection $\{U_{n_1}\}_{n_1=1,\ldots,N}$ is pairwise disjoint.

For each $n_1\in\{1,\ldots,N\}$, let $V_{n_1}=S_1^{-1}\cdot U_{n_1}$, which is a compact open subset of $X$ by Lemma~\ref{actiontopology}. Observe that $S_1\cdot V_{n_1}=U_{n_1}$ since $r_X(U_{n_1})\subset r(S_1)=s(S_1^{-1})$. 
Furthermore, since the collection $\{U_{n_1}\}_{n_1=1,\ldots,N}$ is pairwise disjoint and the source map is injective on $S_1$, the collection $\{V_{n_1}\}_{n_1=1,\ldots,N}$ is pairwise disjoint, and we have $W=S_1^{-1}\cdot\bigsqcup_{n_1=1}^N U_{n_1} = \bigsqcup_{n_1=1}^N S_1^{-1}\cdot U_{n_1} = \bigsqcup_{n_1=1}^N V_{n_1}$ so that the collection $\{V_{n_1}\}_{n_1=1,\ldots,N}$ partitions $W.$ 
Associate to $V_{n_1}$ the subset $H_{n_1}$ of $S_1$ which is able to act upon it. That is, put $H_{n_1}=S_1\cap s^{-1}(r_X(V_{n_1}))$, so that $H_{n_1}\cdot V_{n_1}=S_1\cdot V_{n_1}=U_{n_1}$. Observe that $H_{n_1}$ is compact and open because $r_X$ is continuous and open by Lemma~\ref{etalestrongcontinuity} and $s$ is a local homeomorphism.

In this manner, we have created $N$ ``single-level'' towers $(V_{n_1},H_{n_1})$ for $n_1\in\{1,\ldots,N\}$ whose levels have diameter smaller than $\epsilon$ and partition $S_1\cdot W,$ and such that for each $n_1\in\{1,\ldots,N\}$ and each $x\in V_{n_1}$, we have $H_{n_1}r_X(x)=S_1r_X(x)$. 

Now, for each $n_1=1,\ldots,N$ in turn, consider the subset $S_2\cdot V_{n_1}$ of $S_2\cdot W.$ Arguing similarly as above, construct a cover of $S_2\cdot V_{n_1}$ by pairwise disjoint clopen subsets $\{U_{n_1,n_2}\}_{n_2=1,\ldots,N_{n_1}}$ of $X$ with diameter smaller than $\epsilon$. For each $n_2\in\{1,\ldots,N_{n_1}\}$ put $V_{n_1,n_2}=S_2^{-1}\cdot U_{n_1,n_2}$ and $H_{(n_1,n_2)}=(H_{n_1}\cup S_2)\cap s^{-1}(r_X(V_{n_1,n_2}))$ so that $V_{n_1,n_2}$ and $H_{(n_1,n_2)}$ are both compact and open, and so that $\{V_{n_1,n_2}\}_{n_2=1,\ldots,N_{n_1}}$ partitions $V_{n_1}$. Define $H_{(n_1,n_2),1}=S_1\cap H_{(n_1,n_2)}$ and $H_{(n_1,n_2),2}=S_2\cap H_{(n_1,n_2)}$ so that $H_{(n_1,n_2)}=\bigsqcup_{i=1}^2 H_{(n_1,n_2),i}$. Observe that $H_{(n_1,n_2),1}\cdot V_{n_1,n_2}\subset U_{n_1}$ and $H_{(n_1,n_2),2}\cdot V_{n_1,n_2}\subset U_{n_1,n_2}$ both have diameter smaller than $\epsilon$. In this way, for each $n_1\in\{1,\ldots,N\}$ we construct $N_{n_1}$ ``two-level'' towers $(V_{n_1,n_2},H_{(n_1,n_2)})$ for $n_2\in\{1,\ldots,N_{n_1}\}$. The collection of all of the levels of these towers is pairwise disjoint and partitions $(S_1\cup S_2)\cdot W,$ each level has diameter smaller than $\epsilon$, and for each $x\in V_{n_1,n_2}$ we have $H_{(n_1,n_2)}r_X(x)=(S_1\cup S_2)r_X(x)$. 

Continue to iterate this procedure, at the $k$-th step beginning with a collection of towers $(V_{n_1,n_2,\ldots,n_{k-1}}, H_{(n_1,n_2,\ldots,n_{k-1})})$ for $n_1\in\{1,\ldots,N\}$, $n_2\in\{1,\ldots,N_{n_1}\}$, $\ldots$ , $n_{k-1}\in\{1,\ldots,N_{n_1,\ldots,n_{k-2}}\}$, and covering $S_k\cdot V_{n_1,n_2,\ldots,n_{k-1}}$ by finitely many pairwise disjoint clopen subsets $\{U_{n_1,\ldots,n_k}\}_{n_k=1,\ldots,N_{n_1,\ldots,n_{k-1}}}$ of $X$ with diameter smaller than $\epsilon$. 
Set $V_{n_1,\ldots,n_k}=S_k^{-1}\cdot U_{n_1,\ldots,n_k}$ so that the collection $\{V_{n_1,\ldots,n_k}\}_{n_k=1,\ldots,N_{n_1,\ldots,n_{k-1}}}$ partitions $V_{n_1,\ldots,n_{k-1}}$. 
Also set $H_{(n_1,\ldots,n_k)}=(H_{(n_1,\ldots,n_{k-1})}\cup S_k)\cap s^{-1}(r_X(V_{n_1,\ldots,n_k}))$ (with decomposition $H_{(n_1,\ldots,n_k),i}=S_i\cap H_{(n_1,\ldots,n_k)}$ for each $i\in\{1,\ldots,k\}$) to obtain a collection of ``$k$-level'' towers $\{(V_{n_1,\ldots,n_k}, H_{(n_1,\ldots,n_k)})\}$ indexed by $n_1$ up to $n_k$ such that the collection of all tower levels is pairwise disjoint and partitions $(S_1\cup\cdots\cup S_k)\cdot W$, each level of each tower has diameter smaller than $\epsilon$, and so that for each $x\in V_{n_1,\ldots,n_k}$ we have $H_{(n_1,\ldots,n_k)}r_X(x)=(S_1\cup\cdots\cup S_k)r_X(x)$. 

The procedure terminates after $m$ steps (where $m$ is the number of levels in the tower $(W,S)$) to produce a clopen castle $\{(V_{n_1,\ldots,n_m},H_{(n_1,\ldots,n_m)})\}$, where $n_1\in\{1,\ldots,N\}$, $n_2\in\{1,\ldots,N_{n_1}\}$, $\ldots$ , $n_m\in\{1,\ldots, N_{n_1,\ldots,n_{m-1}}\}$, such that every level of every tower has diameter smaller than $\epsilon$. The levels of this castle will partition $(S_1\cup\cdots\cup S_m)\cdot W=S\cdot W.$ 
For each $x\in V_{n_1,\ldots,n_m}=r_X^{-1}(s(H_{(n_1,\ldots,n_m)}))$ we have $H_{(n_1,\ldots,n_m)}r_X(x)=(S_1\cup\cdots\cup S_m)r_X(x)=Sr_X(x)$, so that for each $u\in s(H_{(n_1,\ldots,n_m)})$ we can use $(C,\epsilon)$-invariance of $S$ to obtain 
\[
|CH_{(n_1,\ldots,n_m)}u\setminus H_{(n_1,\ldots,n_m)}u|=|CSu\setminus Su|<\epsilon|Su|=\epsilon|H_{(n_1,\ldots,n_m)}u|,
\]
which shows that $H_{(n_1,\ldots,n_m)}$ is $(C,\epsilon)$-invariant.

It remains to show that given a clopen castle $\{(W_j,S_j)\}_{j=1,\ldots,J}$ witnessing condition (ii) of almost finiteness, the clopen castle $\{(V^{(j)}_{n_1,\ldots,n_{m_j}},H^{(j)}_{(n_1,\ldots,n_{m_j})})\}_{j,n_1,\ldots,n_{m_j}}$ obtained by applying the procedure above on each tower in the castle satisfies condition (ii) of almost finiteness. For each $j\in\{1,\ldots,J\}$ find sets $S_j'\subset S_j$ such that
\[X\setminus\bigsqcup_{j=1}^JS_jW_j\prec\bigsqcup_{j=1}^JS_j'W_j.\]
Observe that for each $j$ we have
\[\bigsqcup_{n_1,\ldots,n_{m_j}}(S_j'\cap H^{(j)}_{(n_1,\ldots,n_{m_j})})V^{(j)}_{n_1,\ldots,n_{m_j}}=S_j'W_j\]
and that the sets appearing in the union on the left-hand side are levels of the tower $(V^{(j)}_{n_1,\ldots,n_{m_j}}, H^{(j)}_{(n_1,\ldots,n_{m_j})})$. Observe also that $|(S_j'\cap H^{(j)}_{(n_1,\ldots,n_{m_j})})u|\leq|S_j'u|$ for each $u\in G^{(0)}$. It follows that
\begin{align*}
X\setminus\bigsqcup_{j=1}^J\bigsqcup_{n_1,\ldots,n_{m_j}}H^{(j)}_{(n_1,\ldots,n_{m_j})}V^{(j)}_{n_1,\ldots,n_{m_j}}&=X\setminus\bigsqcup_{j=1}^J S_jW_j\\
&\prec\bigsqcup_{j=1}^JS_j'W_j\\&=\bigsqcup_{j=1}^J\bigsqcup_{n_1,\ldots,n_{m_j}}(S_j'\cap H^{(j)}_{(n_1,\ldots,n_{m_j})})V^{(j)}_{n_1,\ldots,n_{m_j}}.
\end{align*}
Therefore, the choices $(H^{(j)}_{(n_1,\ldots,n_{m_j})})'=(S_j'\cap H^{(j)}_{(n_1,\ldots,n_{m_j})})$ witness property (ii) of almost finiteness for the new castle.
\end{proof}

We now see how our notion of almost finiteness links back to Matui's original definition. For convenience, we first present Suzuki's generalisation of Matui's original notion.
\begin{definition}[{\cite[Definition~3.2]{Suzuki}}]\label{fundamental}
Let $K$ be a compact Hausdorff \'etale groupoid. We say that a clopen subset $F$ of $K^{(0)}$ is a \textit{fundamental domain} of $K$ if there exist a natural number $n\in\N$, natural numbers $N_i\in\N$ for each $i\in\{1,\ldots,n\}$, and a pairwise disjoint collection of clopen subsets $\left\{F_j^{(i)}\mid i\in\{1,\ldots,n\}, j\in\{1,\ldots,N_i\}\right\}$ 
of $K^{(0)}$ such that
\begin{enumerate}[label=(\roman*)]
\item $F=\bigsqcup_{i=1}^n F_1^{(i)}$;
\item $K^{(0)}=\bigsqcup_{i=1}^n\bigsqcup_{j=1}^{N_i}F_j^{(i)}$;
\item for each $i\in\{1,\ldots,n\}$ and $j,k\in\{1,\ldots,N_i\}$ there is a unique compact open $K$-set $V_{j,k}^{(i)}$ satisfying $r(V_{j,k}^{(i)})=F_j^{(i)}$ and $s(V_{j,k}^{(i)})=F_k^{(i)}$; and
\item $K=\bigsqcup_{i=1}^n\bigsqcup_{1\leq j,k\leq N_i} V_{j,k}^{(i)}$. 
\end{enumerate}
\end{definition}
\begin{definition}[{\cite[Definition~3.4]{Suzuki}}]
Let $G$ be a locally compact Hausdorff \'etale groupoid. A (compact) subgroupoid $K$ of $G$ is called \textit{elementary} if $G^{(0)}\subset K$ and if $K$ admits a fundamental domain. 
\end{definition}
\begin{definition}[{\cite[Definition~3.6]{Suzuki}}]\label{gpoidAF}
Let $G$ be a locally compact Hausdorff \'etale groupoid. We say that $G$ is \textit{almost finite} if it satisfies
\begin{enumerate}[label=(\roman*)]
\item the union of all compact open $G$-sets covers $G$; and
\item for any compact subset $C\subset G$ and $\epsilon>0$, there exists a $(C,\epsilon)$-invariant elementary subgroupoid of $G$. 
\end{enumerate}
\end{definition}
\begin{thm}[{c.f. \cite[Lemma~5.2]{Suzuki}}]\label{Suzuki Analogue}
Let $G$ be a locally compact Hausdorff \'etale groupoid, $X$ a compact metric space, and $\alpha:G\curvearrowright X$ a continuous action. Let $C\subset G$ be a compact subset, and fix $\epsilon>0$. 

Suppose that $\alpha$ admits a $(C,\epsilon)$-invariant clopen tower decomposition. Then the transformation groupoid $G\ltimes X$ admits a $(C\times X,\epsilon)$-invariant elementary subgroupoid. The converse holds in the case that $G$ is ample.
\end{thm}
\begin{proof}
($\Rightarrow$): Take a clopen tower decomposition $\{(W_i,S_i)\}_{i=1}^n$ of $\alpha$ as in the theorem, so that $S_i=\bigsqcup_{j=1}^{N_i} S_{i,j}$ for each $i\in\{1,\ldots,n\}$, where the range map of $G$ is injective on each $S_{i,j}$. Define an elementary subgroupoid
\[K=\bigsqcup\limits_{i=1}^n\bigsqcup\limits_{1\leq j,k\leq N_i} V_{j,k}^{(i)}\]
as follows. We will let $V^{(i)}_{j,k}$ be a collection of arrows from the $k$-th to the $j$-th level of the $i$-th tower $(W_i,S_i)$, which we will obtain by moving from the $k$-th level $S_{i,k}W_i$ down to the base along $S_{i,k}^{-1}$, and then from the base to the $j$-th level along $S_{i,j}$. More precisely, for each $j\in\{1,\ldots, N_i\}$, consider the collection 
\[V^{(i)}_{j,\textnormal{base}}=(S_{i,j}\times W_i)\cap(G\ltimes X)=\{(g,x)\mid g\in S_{i,j},\ x\in W_i,\ s(g)=r_X(x)\}.\]
 Notice that for each $x\in W_i$ there is exactly one associated element $(g,x)\in V_{j,\textnormal{base}}^{(i)}$. Define
\[V^{(i)}_{j,k}=V^{(i)}_{j,\textnormal{base}}(V^{(i)}_{k,\textnormal{base}})^{-1}.\]
This defines an elementary subgroupoid with $F_j^{(i)}=S_{i,j}\cdot W_i$.

Now, fix a compact subset $C\subset G$ and $\epsilon>0$ and suppose that the clopen tower decomposition $\{(W_1,S_1),\ldots,(W_n,S_n)\}$ is $(C,\epsilon)$-invariant.
Construct the associated elementary subgroupoid
\[ K=\bigsqcup\limits_{i=1}^n\bigsqcup\limits_{1\leq j,k\leq N_i} V_{j,k}^{(i)}.\]
Fix $x\in X$. By construction, there exists a unique $i\in\{1,\ldots,n\}$ and $j\in\{1,\ldots,N_i\}$ such that $x\in F_j^{(i)}$. Then the only elements $k\in K$ for which $s(k)=x$ lie in the sets $V^{(i)}_{l,j}$ for $l\in\{1,\ldots,N_i\}$. From each such set, there exists a unique $k_l\in K$ with $s(k_l)=x$, so $|Kx|=N_i$. 
On the other hand, we assumed that we can partition $S_i$ into subsets $S_i=\bigsqcup_{j=1}^{N_i} S_{i,j}$ such that each $S_{i,j}$ has the same source, and such that the source map is injective on each $S_{i,j}$. For any $u\in s(S_i)$, it follows that $|S_iu|=N_i$, which shows that $|Kx|=|S_iu|$.

For the choice of $x\in F_j^{(i)}=S_{i,j}W_i$ from above, write $x=hy$ for $h\in S_{i,j}$ and $y\in W_i$. We can compute
\begin{align*}
Kx&=\{(gh^{-1}, x)\mid g\in S_{i,l}, s(g)=s(h), l=1,\ldots,N_i\}\\
&=\{(gh^{-1},hy)\mid g\in S_{i,l}, s(g)=s(h), l=1,\ldots,N_i\},
\end{align*}
which gives
\[(C\times X)Kx=\{(c,z)(gh^{-1},hy)\mid g\in S_{i,l}, s(g)=s(h), l=1,\ldots,N_i, c\in C, z\in X\}.\]
These elements are composable when $s(c)=r(g)$ and $gy=z$, producing the product $(cgh^{-1}, hy)$. Thus, with $h$ and $y$ fixed as above, we see that
\begin{align*}
(C\times X)Kx\setminus Kx=\{(cgh^{-1},hy)\mid &g\in S_{i,l}, s(g)=s(h), l=1,\ldots,N_i,\\ &c\in C, s(c)=r(g), cg\notin S_{i,m}\textnormal{ for }m=1,\ldots,N_i\}
\\=\{(cgh^{-1},hy)\mid &g\in S_i, s(g)=s(h),c\in C, s(c)=r(g), cg\notin S_i\}.\end{align*}
On the other hand, for $u=r_X(y)\in S_i^{(0)}$, we have 
\begin{align*}
CS_iu\setminus S_iu&=\{cgu\mid g\in S_i, s(g)=u=r_X(y), c\in C, s(c)=r(g), cg\notin S_i\}\\
&=\{cgu\mid g\in S_i, s(g)=s(h), c\in C, s(c)=r(g), cg\notin S_i\},
\end{align*}
so we see that $|CS_iu\setminus S_iu|=|(C\times X)Kx\setminus Kx|$ via the bijection $(cgh^{-1},hy)\mapsto cgu$. Thus, combining this with our previous observation that $|Kx|=|S_iu|=N_i$ for each $u\in S_i^{(0)}$, we have obtained
\[\frac{|(C\times X)Kx\setminus Kx|}{|Kx|}=\frac{|CS_iu\setminus S_iu|}{|S_iu|}<\epsilon,\]
where the last inequality uses $(C,\epsilon)$-invariance of $S_i$. This shows that the elementary subgroupoid $K$ is $(C\times X,\epsilon)$-invariant.

($\Leftarrow$): Let $G$ be ample, and let $K=\bigsqcup_{i=1}^n\bigsqcup_{1\leq j,k\leq N_i} V_{j,k}^{(i)}$ be a $(C\times X,\epsilon)$-invariant elementary subgroupoid of $G\ltimes X$. Since each $V_{j,k}^{(i)}$ is a compact subset of $G\ltimes X$, the image $\pi_G(V_{j,k}^{(i)})$ under the projection to the $G$ coordinate is also compact and open. Since $G$ is ample, the topology on $G$ has a base of clopen bisections by \cite[Lemma~2.2]{CFST}. Therefore, using compactness, we can write $\pi_G(V_{j,k}^{(i)})$ as a union of finitely many pairwise disjoint clopen bisections, each of which will be compact in $G$ because they are closed in the compact subset $\pi_G(V_{j,k}^{(i)})$. We may use these bisections to split $V_{j,k}^{(i)}$ into smaller compact open $(G\ltimes X)$-sets, to assume without loss of generality that the range and source maps are injective on $\pi_G(V_{j,k}^{(i)})$. We note that because the original $V_{j,k}^{(i)}$ was a $(G\ltimes X)$-set, the newly defined subsets $V_{j,k}^{(i)}$ also yield redefined clopen sets $F_j^{(i)}\subset X$ such that the collection $\{F_j^{(i)}\}_{i,j}$ is pairwise disjoint.

We then form a clopen tower decomposition $\{(W_i,S_i)\}_{i=1,\ldots,n}$ as follows. Set $W_i=F_1^{(i)}$ and
\[S_i=\pi_G(s^{-1}(F_1^{(i)})\cap K)=\pi_G\left(\bigsqcup_{l=1}^{N_i}V^{(i)}_{l,1}\right),\]
with associated decomposition $S_i=\bigsqcup_{j=1}^{N_i}S_{i,j}$, where $S_{i,j}=\pi_G(V^{(i)}_{j,1})$. 
The levels of the $i$-th tower are then precisely the clopen sets $F^{(i)}_j$ for $j\in\{1,\ldots,N_i\}$. Using the properties of the elementary subgroupoid, we see that the base of each tower is clopen, each $S_{i,j}$ is compact and open, $s(S_{i,j})=r_X(W_i)$, and that the collection of tower levels is pairwise disjoint and partitions $X$. Due to the way we modified $V^{(i)}_{j,k}$, each $S_{i,j}$ is a subset of a bisection, and it follows that the source and range maps are injective on  $S_{i,j}$. This shows that $\{(W_i,S_i)\}_{i=1,\ldots,n}$ is a clopen tower decomposition, as we claimed.

Next, we show that this clopen tower decomposition is $(C,\epsilon)$-invariant. For each $x\in W_i$ we compute
\begin{align*}
|(C\times X)Kx\setminus Kx|&=|\{(cg,x)\mid c\in C, (g,x)\in Kx, (cg,x)\notin Kx\}|\\
&=|\{cg\mid c\in C, g\in\pi_G(Kx), cg\notin\pi_G(Kx)\}|\\
&=|\{cg\mid c\in C, g\in S_ir_X(x), cg\notin S_ir_X(x)\}|\\
&=|CS_ir_X(x)\setminus S_ir_X(x)|.
\end{align*}
We also have $|Kx|=|\pi_G((s^{-1}(W_i)\cap K)x)|=|S_ir_X(x)|$. Putting this all together gives for each $x\in s(S_i)=W_i$ that
\[\frac{|CS_ir_X(x)\setminus S_ir_X(x)|}{|S_ir_X(x)|}=\frac{|(C\times X)Kx\setminus Kx|}{|Kx|}<\epsilon.\]

Thus, we have produced a $(C,\epsilon)$-invariant clopen tower decomposition of $\alpha$.
\end{proof}
\begin{cor}
Let $G$ be a locally compact Hausdorff \'etale groupoid with compact unit space, $X$ a totally disconnected compact metric space, and $\alpha:G\curvearrowright X$ a continuous action.

If $\alpha$ is almost finite, with the additional requirement that the open castles in Definition~\ref{defgroupoidactionaf} can be chosen to be clopen tower decompositions of $X$, then the transformation groupoid $G\ltimes X$ is almost finite. 
Conversely, if $G$ is ample and $G\ltimes X$ is almost finite, 
 then $\alpha$ is almost finite and the castles appearing in the definition can be chosen to be clopen tower decompositions. 
\end{cor}
\begin{proof}
Fix $\epsilon>0$ to be used throughout the proof. 

First assume that $\alpha$ is almost finite, and that the open castles in the definition can be chosen to be clopen tower decompositions. Fix a compact subset $A\subset G\ltimes X$. We aim to produce an $(A,\epsilon)$-invariant elementary subgroupoid of $G\ltimes X$. Denote by $\pi_G:G\ltimes X\rightarrow G$ the projection map to the $G$-coordinate. Since $\pi_G$ is continuous, $\pi_G(A)$ is a compact subset of $G$. By our assumption on $\alpha$, we can find a $(\pi_G(A),\epsilon)$-invariant clopen tower decomposition. Applying Theorem~\ref{Suzuki Analogue} produces a $(\pi_G(A)\times X,\epsilon)$-invariant elementary subgroupoid of $G\ltimes X$. Since $A\subset \pi_G(A)\ltimes X$, this elementary subgroupoid is also $(A,\epsilon)$-invariant, so we see that $G\ltimes X$ satisfies condition (ii) of Definition~\ref{gpoidAF}. Since $(G\ltimes X)^{(0)}\cong X$ is totally disconnected, $G\ltimes X$ automatically satisfies condition (i) of Definition~\ref{gpoidAF} (see for example \cite{Exel}), and so is seen to be almost finite.

Now assume that $G$ is ample and $G\ltimes X$ is almost finite. Fix a compact subset $C\subset G$. We aim to construct a $(C,\epsilon)$-invariant open castle which satisfies conditions (i) and (ii) from the final part of Definition~\ref{defgroupoidactionaf}.
By almost finiteness of $G\ltimes X$, since $C\times X$ is a compact subset of $G\ltimes X$, there exists a $(C\times X,\epsilon)$ elementary subgroupoid of $G\ltimes X$. Since $G$ is ample, we can apply Theorem~\ref{Suzuki Analogue} to obtain a $(C,\epsilon)$-invariant clopen tower decomposition of $\alpha$. Thus, we have obtained a $(C,\epsilon)$-invariant open castle which covers $X$, so that condition (ii) from Definition~\ref{defgroupoidactionaf} is automatically satisfied. It only remains to show that we can arrange for the levels of this castle to have diameter smaller than $\epsilon$, which follows from Lemma~\ref{amplediameter} since $X$ is totally disconnected.
\end{proof}

\section{Ornstein-Weiss quasitiling machinery for groupoids}\label{sec:groupoidOW}
Following \cite{Kerr}, we require a generalisation of the Ornstein-Weiss quasitiling theorem \cite[page~24,~Theorem~6]{OW} to the groupoid setting. In this section we prove a version of the Ornstein-Weiss theorem that mimics the alternative formulation appearing as \cite[Theorem~4.36]{Kerr-Li}. 
 
We will need the following concepts:
\begin{definition}[{\cite[Definition~4.29]{Kerr-Li}}]\label{finitemeascovering}
Let $(X,\mu)$ be a finite measure space. Let $\lambda,\epsilon\geq0$. We say that a collection $\{A_i\}_{i\in I}$ of measurable subsets of $X$
\begin{enumerate}[label=(\roman*)]
\item is a \textit{$\lambda$-even} covering of $X$ if there exists a positive integer $M$ such that $\sum_{i\in I}\mathbf{1}_{A_i}\leq M$ and $\sum_{i\in I}\mu(A_i)\geq\lambda M\mu(X)$, in which case $M$ is called a \textit{multiplicity} of the $\lambda$-even covering.
\item \textit{$\lambda$-covers} $X$ if $\mu(\bigcup_{i\in I}A_i)\geq\lambda\mu(X)$.
\item is \textit{$\epsilon$-disjoint} if, for each $i\in I$, there exists a subset $\widehat{A_i}\subset A_i$ such that $\mu(\widehat{A_i})\geq(1-\epsilon)\mu(A_i)$, and such that the collection $\{\widehat{A_i}\}_{i\in I}$ is pairwise disjoint.
\end{enumerate}
\end{definition}
We extend the above definition to apply to collections of compact subsets of \'etale groupoids by asking for the appropriate concept to hold at each (finite) source bundle within each subset. One can also write down a similar definition for locally compact groupoids admitting a Haar system by using the Haar measure at each unit. In the \'etale case, the Haar measure is just the counting measure, which we use in the definition below. 

\begin{definition}\label{groupoidcovering}
Let $G$ be an \'etale groupoid and $\lambda\geq0$. Let $A\subset G$ be compact. We say that a collection $\{A_i\}_{i\in I}$ of compact subsets of $A$
\begin{enumerate}[label=(\roman*)]
\item is a \textit{$\lambda$-even covering} of $A$ with multiplicity $M$ if, for every $u\in s(A)$, the collection $\{A_iu\}_{i\in I}$ is a $\lambda$-even covering of $Au$ with multiplicity $M$.
\item \textit{$\lambda$-covers} $A$ if $\{A_iu\}_{i\in I}$ $\lambda$-covers $Au$ for each $u\in s(A)$. 
\item is \textit{$\epsilon$-disjoint} if $\{A_iu\}_{i\in I}$ is $\epsilon$-disjoint for each $u\in s(A)$.
\end{enumerate}
\end{definition}
The following pair of lemmas will be applied repeatedly in the proof of our quasitiling result. 
\begin{lemma}[c.f. {\cite[Lemma~4.31]{Kerr-Li}}] \label{evendisjoint}
Let $A$ be a compact subset of any locally compact Hausdorff \'etale groupoid (whose unit space is not necessarily compact). Let $0\leq\epsilon\leq\frac{1}{2}$ and $0<\lambda\leq1$, and let $\{A_i\}_{i\in I}$ be a $\lambda$-even covering of $A$. Then, for each $u\in s(A)$, there is an $\epsilon$-disjoint subcollection $\{A_ju\}_{j\in J_u}$ of $\{A_iu\}_{i\in I}$, indexed by $J_u\subset I$, which $\epsilon\lambda$-covers $Au$. 
\end{lemma}
\begin{proof}
By definition, for each $u\in s(A)$, the collection $\{A_iu\}_{i\in I}$ is a $\lambda$-even covering of the finite measure space $Ax$ (equipped with the counting measure) in the sense of Definition~\ref{finitemeascovering}. Thus, for each $u\in s(A)$, applying \cite[Lemma~4.31]{Kerr-Li} yields the $\epsilon$-disjoint subcollection of $\{A_iu\}_{i\in I}$ that $\epsilon\lambda$-covers $Au$ that we seek.
\end{proof}
Throughout this section, we will use as standard the notion of approximate invariance given by property (ii) in Lemma~\ref{equivgpoidapproxinvce}.
\begin{definition}\label{Iinvariant}
Let $G$ be an \'etale groupoid. Let $C,A\subset G$ be nonempty and compact, and let $\epsilon>0$. We denote
\[I_C(A)\coloneqq\{a\in A\mid Ca\subset A\},\]
and we say that $A$ is \textit{$(C,\epsilon)$-invariant} if, for each $u\in s(A)$, we have 
\[
|I_C(Au)|\geq(1-\epsilon)|Au|.
\]
\end{definition}
\begin{lemma}[c.f. {\cite[Lemma~4.33]{Kerr-Li}}]\label{eventranslates}
Let $G$ be a locally compact Hausdorff \'etale groupoid with compact unit space. Let $C\subset G$ be a nonempty, compact subset containing $G^{(0)}$. Choose $\epsilon>0$, and let
\[\delta\geq\max\left(0, 1-\frac{\inf_{w\in G^{(0)}}|Cw|}{\sup_{w\in G^{(0)}}|wC|}(1-\epsilon)\right).\]
Suppose that $A\subset G$ is a nonempty, compact subset which is $(C,\epsilon)$-invariant. 
Then the collection $\{Ca\mid a\in I_C(A)\}$ is a $(1-\delta)$-even covering of $A$ with multiplicity $\sup_{w\in G^{(0)}}|wC|$. 
\end{lemma}
\begin{proof}
Since $A$ is $(C,\epsilon)$-invariant, for each $u\in s(A)$ we have $|I_C(Au)|\geq(1-\epsilon)|Au|$.
Therefore, we see that
\begin{align*}
\sum_{a\in I_C(A)}|Cau|&\geq |I_C(Au)|\inf_{a\in I_C(Au)}|Cr(a)|\\
&\geq(1-\epsilon)|Au|\inf_{w\in G^{(0)}}|Cw|\\
&=(1-\epsilon)|Au|\sup_{w\in G^{(0)}}|wC|\frac{\inf_{w\in G^{(0)}}|Cw|}{\sup_{w\in G^{(0)}}|wC|}\\
&\geq(1-\delta)\sup_{w\in G^{(0)}}|wC||Au|,
\end{align*}
so that the second condition in the definition of a $(1-\delta)$-even covering with multiplicity $\sup_{w\in G^{(0)}}|wC|$ is satisfied.

Next, to have $g\in Ca_1u_1 \cap Ca_2u_2$, we require that $u_1=s(g)=u_2=u$. So, suppose that $g\in Ca_1u\cap Ca_2u$ for distinct $a_1, a_2 \in I_C(Au)$. Then there exist $c_1, c_2\in C$ such that $c_1a_1=g=c_2a_2$, which requires that $r(c_1)=r(g)=r(c_2)$. Since $a_1\neq a_2$, we must also have that $c_1\neq c_2$. Extending this argument, we see that for each fixed $u\in s(A)$, the intersection of $n\in\N$ sets from the collection $\{Cau\mid a\in I_C(A)\}$ at a common element $g$ requires the existence of $n$ distinct elements of $r(g)C$. This shows that, for each $u\in s(A)$ and $g\in Au$, 
\[\sum\limits_{a\in I_C(Au)}\mathbf{1}_{Cau}(g)\leq |r(g)C|\leq\sup_{w\in G^{(0)}}|wC|.\qedhere\]
\end{proof}
In the group case, one may alternatively characterise approximate invariance by using a ``small boundary'' condition. We introduce an analogue of the notion for groupoids.
\begin{definition}[c.f. {\cite[Definition~4.34]{Kerr-Li}}]
Let $G$ be a locally compact Hausdorff \'etale groupoid, and let $C,A$ be compact subsets of $G$. The \textit{$C$-boundary} of $A$ is
\[\partial_C(A)\coloneqq\{g\in G\mid Cg\cap A\neq\emptyset\textnormal{ and }Cg\cap A^c\neq\emptyset\}\]
\end{definition}
Notice that $\partial_C(A)$ may be infinite. Therefore, to obtain the aforementioned ``small boundary'' condition, we need to consider the $C$-boundary of the source bundle at each unit, and ask that for every $u\in s(A)$ we have
\[|\partial_C(Au)|\leq\epsilon|Au|.\]
In particular, notice that $\partial_C(Au)\subset C^{-1}Au$ and so it is always finite since $C^{-1}A$ is compact. 

We now introduce our analogue of quasitilings. When reading the following definition, notice that cutting down to the finite source bundles $Au$ is still necessary, but is built into our terminology via Definition~\ref{groupoidcovering}.
\begin{definition}[c.f. {\cite[Definition~4.35]{Kerr-Li}}]
Let $G$ be a locally compact Hausdorff \'etale groupoid with compact unit space. Let $A$ be a compact subset of $G$, and fix $\epsilon>0$. A finite collection $\{T_1,\ldots,T_n\}$ of compact subsets of $G$ is said to \textit{$\epsilon$-quasitile} $A$ if there exist $C_1,\ldots,C_n\subset G$ such that $\bigcup_{i=1}^n T_iC_i\subset A$ and so that the collection of right translates $\{T_ic\mid i\in\{1,\ldots,n\}, c\in C_i\}$ is $\epsilon$-disjoint and $(1-\epsilon)$-covers $A$. 

The subsets $T_1,\ldots,T_n$ are referred to as the \textit{tiles}, and $C_1,\ldots,C_n$ as the \textit{centres} of the quasitiling.
\end{definition}
The following result is our analogue of the Ornstein-Weiss quasitiling theorem. In the case that $G$ is an amenable group, tiles $T_1,\ldots,T_n$ satisfying the conditions of the theorem automatically exist. We conjecture that the same is true when $G$ is a (topologically) amenable groupoid. We show that our main examples admit such tiles in Lemma~\ref{tilingTset}. 
\begin{thm}[c.f. {\cite[Theorem~4.36]{Kerr-Li}}]\label{groupoidOW}
Let $G$ be a locally compact Hausdorff \'etale groupoid with compact unit space.
Let $0<\epsilon<1/2$, and choose $m\in\N$ such that $m\geq2$. Let $n$ be a positive integer such that $(1-\epsilon/m)^n<\epsilon$. 
Suppose that $T_1\subset T_2\subset\cdots\subset T_n$ are compact subsets of $G$ which contain $G^{(0)}$ and are such that, for each $i\in\{1,\ldots,n\}$,
\begin{enumerate}[label=(\roman*)]
\item $1/m<\frac{\inf_{w\in G^{(0)}}|T_iw|}{\sup_{w\in G^{(0)}}|wT_i|}\leq1$;
\item $\frac{\inf_{w\in G^{(0)}}|T_iw|}{\sup_{w\in G^{(0)}}|wT_i|}(1-\epsilon)\geq1/m$; and
\item $|\partial_{T_{i-1}}(T_iw)|\leq(\epsilon^2/8)|T_iw|$ for each $w\in G^{(0)}$.
\end{enumerate}
Then every $(T_n,\epsilon^2/4)$-invariant compact subset of $G$ is $\epsilon$-quasitiled by $\{T_1,\ldots,T_n\}$. 
\end{thm}
We briefly describe the ideas behind the proof of Theorem~\ref{groupoidOW}. The general philosophy will be to start with an $(T_n,\epsilon^2/4)$-invariant compact subset $A$ of $G$, apply Lemma~\ref{eventranslates} to construct an even covering of $A$ whose tolerance is limited by $T_n$ and $\epsilon$, and then use Lemma~\ref{evendisjoint} to turn our even covering into an $\epsilon$-disjoint collection which still covers some fixed portion of each source bundle in $A$. The assumptions of Theorem~\ref{groupoidOW} are designed to ensure that we can iterate this procedure $n\in\N$ times. In particular, at each stage, the relative complement in $A$ of everything covered up to that point will be sufficiently invariant (under some $T_k$, where the index $k$ shrinks as the procedure continues) to allow another application of this pair of lemmas. In addition, the assumptions ensure that when this procedure terminates, the union of all the $\epsilon$-disjoint subcollections that have been constructed will $(1-\epsilon)$-cover the invariant compact set $A$ that we started with. 
\begin{proof}[Proof of Theorem~\ref{groupoidOW}]
Let $A$ be a nonempty $(T_n,\epsilon^2/4)$-invariant compact subset of $G$. We will recursively construct $C_n,\ldots,C_1\subset G$ (in that order) such that, for each $k\in\{1,\ldots,n\}$, $\bigcup_{i=k}^n T_iC_i\subset A$, and so that the collection of translates $\{T_ic\mid i\in\{k,\ldots,n\}, c\in C_i\}$ is $\epsilon$-disjoint and $\lambda_k$-covers $A$, where
\[\lambda_k=\min(1-\epsilon,\ 1-(1-\epsilon/m)^{n-k+1}).\]
By our assumption on $n$ it will then follow that $\{T_1,\ldots, T_n\}$ $\epsilon$-quasitiles $A$. Indeed, we have that $(1-\epsilon/m)^n<\epsilon$, so $\lambda_1=1-\epsilon$, as required.

For the base step (when $k=n$), we apply Lemma~\ref{eventranslates} to the $(T_n,\epsilon^2/4)$-invariant compact subset $A$ to notice that the collection $\{T_na\mid a\in I_{T_n}(A) \}$, where $I_{T_n}(A)=\{a\in A\mid T_na\subset A\}$ is as in Definition~\ref{Iinvariant},
 is a $\beta$-even covering of $A$, where 
\[\beta=\frac{\inf_{w\in G^{(0)}}|T_nw|}{\sup_{w\in G^{(0)}}|wT_n|}(1-\epsilon^2/4).\]
Observe that $\beta\geq\frac{\inf|T_nw|}{\sup|wT_n|}(1-\epsilon)$, and so, by condition (ii), $\beta\geq1/m$. This shows that the collection $\{T_na\mid a\in I_{T_n}(A)\}$ is a $(1/m)$-even covering of $A$. 
Then we can apply Lemma~\ref{evendisjoint} to find, for each $u\in G^{(0)}$, a subset $C_{n,u}\subset I_{T_n}(Au)$ such that the subcollection $\{T_nc\mid c\in C_{n,u}\}$ is $\epsilon$-disjoint and $(\epsilon/m)$-covers $Au$. This allows us to construct $C_n=\bigsqcup_{u\in G^{(0)}}C_{n,u}$, which satisfies the properties we seek.

Suppose that, for some $k\in\{1,\ldots,n-1\}$, we have constructed $C_n,C_{n-1},\ldots,C_{k+1}\subset G$ with the properties we desire. Set $A_k'=A\setminus\bigcup_{i=k+1}^n T_iC_i$. If $|A_k'u|<\epsilon|Au|$ for every $u\in s(A)$, then we can finish our construction by taking each of $C_k,\ldots,C_1$ to be the empty set, so suppose that the set $Z_k\coloneqq\{u\in s(A)\mid |A_k'u|\geq\epsilon|Au|\}$ is nonempty. Then, for each $u\in G^{(0)}\setminus Z_k$, we will take $C_ku=\emptyset$, and we define $A_k=\bigcup_{u\in Z_k}A_k'u$. We now wish to show that $A_k$ is $(T_k,\frac{\epsilon}{2})$-invariant. 

First, observe that, for each $u\in G^{(0)}$, $i\in\{k+1,\ldots,n\}$, and $c\in C_iu$, we have $|\partial_{T_k}(T_icu)|\leq|\partial_{T_{i-1}}(T_icu)|$, since $T_{k}\subset T_{i-1}$. Notice that the map $a\mapsto ac^{-1}$ is an injection which maps $\partial_{T_{i-1}}(T_ics(c))$ into $\partial_{T_{i-1}}(T_ir(c))$. Using this along with the boundary condition (iii) on the $T_i$, we see that
\[|\partial_{T_k}(T_ics(c))|\leq|\partial_{T_{i-1}}(T_ics(c))|\leq|\partial_{T_{i-1}}(T_ir(c))|\leq(\epsilon^2/8)|T_ir(c)|.\]
Since the collection $\{T_ics(c)\mid i\in\{k+1,\ldots,n\},\ c\in C_i\}$ is $\frac{1}{2}$-disjoint (since $\epsilon<\frac{1}{2}$), and since $|\bigcup_{i=k+1}^n T_iC_iu|\leq|Au|\leq\epsilon^{-1}|A_ku|$ for every $u\in Z_k$, we obtain, for each $u\in Z_k$, that
\begin{align}\label{unionboundary}
\left|\bigcup_{i=k+1}^n\bigcup_{c\in C_iu}\partial_{T_k}(T_icu)\right|&\leq\frac{\epsilon^2}{8}\sum_{i=k+1}^n\sum_{c\in C_iu}|T_ir(c)|\nonumber \\ 
&=\frac{\epsilon^2}{8}\sum_{i=k+1}^n\sum_{c\in C_iu}|T_ic|\nonumber \\ 
&\leq\frac{\epsilon^2}{4}\left|\bigcup_{i=k+1}^nT_iC_iu\right|\nonumber \\ 
&\leq\frac{\epsilon}{4}|A_ku|,
\end{align}
where we used $\frac{1}{2}$-disjointness to obtain the penultimate inequality, as follows. Use $\frac{1}{2}$-disjointness to find, for each $i\in\{k+1,\ldots,n\}$ and $c\in C_iu$, a subset $\widehat{T_{i,c}}\subset T_i$ such that $|\widehat{T_{i,c}}c|\geq\frac{1}{2}|T_ic|$, and such that the collection $\{\widehat{T_{i,c}}c\mid i\in\{k+1,\ldots,n\}, c\in C_iu\}$ is pairwise disjoint. Then we have
\begin{align*}
\left|\bigcup_{i=k+1}^nT_iC_iu\right|&=\left|\bigcup_{i=k+1}^n\bigcup_{c\in C_iu}T_ic\right|\\
&\geq\left|\bigsqcup_{i=k+1}^n\bigsqcup_{c\in C_iu}\widehat{T_{i,c}}c\right|\\
&=\sum_{i=k+1}^n\sum_{c\in C_iu}|\widehat{T_{i,c}}c|\\
&\geq\frac{1}{2}\sum_{i=k+1}^n\sum_{c\in C_iu}|T_ic|.
\end{align*} 
Notice that all the sums in the computation of (\ref{unionboundary}) are finite. In particular, one can see that $|C_iu|<\infty$ for each $i$ and each $u\in G^{(0)}$, because $T_iC_iu\subset Au$ which is finite, and so $\infty>|T_iC_iu|=\sum_{c\in C_iu}|T_ir(c)|\geq|C_iu|\inf_{w\in G^{(0)}}|T_iw|\geq |C_iu|$ because $G^{(0)}\subset T_i$. 

Since $A$ is $(T_n,\epsilon^2/4)$-invariant, and $T_k\subset T_n$, $A$ is also $(T_k,\epsilon^2/4)$-invariant. Thus, for each $u\in s(A)$, $|I_{T_k}(Au)|=|\{a\in Au\mid T_ka\subset A\}|\geq(1-\epsilon^2/4)|Au|$, and so, for each $u\in s(A_k)\subset Z_k$, we have
\begin{align*}
|I_{T_k}(A_ku)|&=|\{a\in A_ku\mid T_ka\subset A_ku\}|\\
&=\left|I_{T_k}(Au)\setminus\bigcup_{i=k+1}^n\left(T_iC_iu\cup\bigcup_{c\in C_iu}\partial_{T_k}(T_icu)\right)\right|\\
&\geq|I_{T_k}(Au)|-\left|\bigcup_{i=k+1}^nT_iC_iu\right|-\left|\bigcup_{i=k+1}^n\bigcup_{c\in C_iu}\partial_{T_k}(T_icu)\right|\\
&\geq\left(1-\frac{\epsilon^2}{4}\right)|Au|-(|Au|-|A_ku|)-\frac{\epsilon}{4}|A_ku|\\
&=|A_ku|-\frac{\epsilon}{4}|A_ku|-\frac{\epsilon^2}{4}|Au|\\
&\geq\left(1-\frac{\epsilon}{2}\right)|A_ku|,
\end{align*}
which shows that $A_k$ is $(T_k, \frac{\epsilon}{2})$-invariant. 
In the first equality, we think of the complement as enforcing additional constraints on the elements of $I_{T_k}(Au)$ to get elements of the left-hand side. Clearly, such elements cannot lie in $\bigcup_{i=k+1}^nT_iC_iu$, as this set is the complement of $A_ku$ in $Au$. Furthermore, if an element $a\in I_{T_k}(Au)$ lies in $\partial_{T_k}(T_icu)$ for some $i\in\{k+1,\ldots,n\}$ and $c\in C_i$, then there exists $t\in T_k$ such that $ta\in T_icu\subset \bigcup_{i=k+1}^nT_iC_iu=Au\setminus A_ku$, so such elements cannot lie in the left-hand set. To obtain the third line, we have used inequality~\eqref{unionboundary} and the fact that $\bigcup_{i=k+1}^nT_iC_iu=Au\setminus A_ku$. The final line uses the fact that $|Au|\leq\epsilon^{-1}|A_ku|$, since $u\in Z_k$.

Thus, we can apply Lemma~\ref{eventranslates} to see that the collection of right translates of $T_k$ which lie in $A_k$ form a $\beta_k$-even covering of $A_k$, where
\[\beta_k=\frac{\inf_{w\in G^{(0)}}|T_kw|}{\sup_{w\in G^{(0)}}|wT_k|}\left(1-\frac{\epsilon}{2}\right).\]
Observe that $\beta_k\geq\frac{\inf|T_kw|}{\sup|wT_k|}(1-\epsilon)$ and so, by condition (ii), we see that $\beta_k>1/m$. Therefore, the collection $\{T_ka\mid a\in A_k,\ T_ka\subset A_k\}$ is a $(1/m)$-even covering of $A_k$. Use Lemma~\ref{evendisjoint} to obtain, for each $u\in s(A_k)$, an $\epsilon$-disjoint subcollection $\{T_kcu\mid c\in C_{k,u}\}$ of these translates which $(\epsilon/m)$-covers $A_k$, and set $C_k=\bigsqcup_{u\in s(A_k)}C_{k,u}$. 

Note that $\bigcup_{i=k}^n\{T_ic\mid c\in C_i\}$ is $\epsilon$-disjoint, because it is the union of two $\epsilon$-disjoint collections, $\{T_kc\mid c\in C_k\}$ and $\bigcup_{i=k+1}^n\{T_ic\mid c\in C_i\}$, which are such that the members of each collection are disjoint from all the members of the other. All that remains to show is that the new collection $\lambda_k$-covers $A$. Indeed, it is enough to show that it $\lambda_k$-covers $Au$ for each $u\in Z_k$, since we already know the collection $(1-\epsilon)$-covers $Au$ for each $u\notin Z_k$. 

We assumed that $\bigcup_{i=k+1}^nT_iC_i$ was a $\lambda_{k+1}$-cover of $A$, so, at each unit $u\in s(A)$, the remainder has cardinality at most $(1-\lambda_{k+1})|Au|$. For each $u\in Z_k$, we constructed an $(\epsilon/m)$-cover of this remainder. Thus, for each $u\in Z_k$, the collection $\bigcup_{i=k}^n T_iC_iu$ is an $\alpha_k$-cover of $Au$ where
\[\alpha_k=\lambda_{k+1}+\frac{\epsilon}{m}(1-\lambda_{k+1}).\]

If $\lambda_{k+1}=1-\epsilon$, we obtain $\alpha_k=1-\epsilon+\epsilon^2/m>1-\epsilon\geq\lambda_k$.

 If $\lambda_{k+1}=1-(1-\epsilon/m)^{n-k}$, we obtain 
\begin{align*}
\alpha_k&=1-\left(1-\frac{\epsilon}{m}\right)^{n-k}+\frac{\epsilon}{m}\left(1-\frac{\epsilon}{m}\right)^{n-k}\\
&=1-\left(1-\frac{\epsilon}{m}\right)^{n-k}\left(1-\frac{\epsilon}{m}\right)\\
&=1-\left(1-\frac{\epsilon}{m}\right)^{n-k+1}\\
&\geq\lambda_k.
\end{align*}
In both cases, $\alpha_k\geq\lambda_k$, so $\bigcup_{i=k}^nT_iC_i$ is a $\lambda_k$-cover. 
\end{proof}

\section{Groupoid crossed products and conditional expectations}\label{sec:gpoidcrossedproducts}

In order to keep the paper as self-contained as possible, we now briefly review of the theory of groupoid crossed products introduced in \cite{Ren87}. There are several excellent treatments in the literature \cite{KS,Goehle}, but we take the approach of \cite{MuhWil}, which strongly emphasises the bundle structure. 

The main ideas for the reader to take away from this section are that we are trying to get our hands on formal sums which interact in a similar manner as elements of a group crossed product, but which respect the fibred structure of the groupoid. The end of the section contains some new material, in which we consider a conditional expectation on groupoid crossed products and use this to generalise a lemma of \cite{Phillips} for use in our main result (see pages~\pageref{condexp}--\pageref{endcondexp}).

In order to better focus on our examples of interest, we restrict our development to the topological setting, where the crossed product arises from the action of a groupoid on a space. 
\begin{definition}[{\cite[Definition~3.12]{Goehle}}]
Let $Y$ be a locally compact Hausdorff space. We say that a $C^*$-algebra $A$ is a $C_0(Y)$\textit{-algebra} if there exists a $*$-homomorphism $\Phi_A$ from $C_0(Y)$ into the center of the multiplier algebra $ZM(A)$ which is nondegenerate in the sense that the set
\[\Phi_A(C_0(Y))\cdot A\coloneqq\textnormal{span}\{\Phi_A(f)a\mid f\in C_0(Y), a\in A\}\]
is dense in $A$. 

In the case that $Y$ is compact, we instead say that $A$ is a \textit{$C(Y)$-algebra}.
\end{definition}
Suppose that $G$ is a groupoid acting continuously on a compact Hausdorff space $X$, with the associated continuous surjection $r_X:X\rightarrow G^{(0)}$ (observe that this implies that $G^{(0)}$ is compact). Then $A=C(X)$ becomes a $C(G^{(0)})$-algebra under the $*$-homomorphism $\Phi_A(f)=f\circ r_X$. 

For $u\in G^{(0)}$, let $I_u\subset C(X)$ denote the ideal of functions which vanish on $r_X^{-1}(u)$. Denote the quotient $A(u)\coloneqq A/I_u$, and the image of $a\in A$ under the associated quotient map by $a(u)$. Observe that $A(u)\cong C(r_X^{-1}(u))$. Denote the disjoint union $\AA\coloneqq\bigsqcup_{u\in G^{(0)}}A(u)$, and associate to $\AA$ the projection map $p:\AA\rightarrow G^{(0)}$ sending $A(u)$ to $u$. We think of $\AA$ as a bundle of $C^*$-algebras over $G^{(0)}$ and use \cite[Theorem~C.27]{Williams} to equip $\AA$ with a topology so that the map $a\mapsto(u\mapsto a(u))$ becomes an isomorphism from $A$ to the set of continuous sections $\Gamma(G^{(0)},\AA)\coloneqq\{f:G^{(0)}\rightarrow\AA\mid f\textnormal{ is continuous and }p(f(u))=u\textnormal{ for every }u\in G^{(0)}\}$.

The action $G\curvearrowright X$ induces an action $\alpha:G\curvearrowright C(X)$ via $\alpha_g(f)(x)=f(g^{-1}x)$, and the triple $(C(X),G,\alpha)$ is an example of a \textit{groupoid dynamical system} \cite{MuhWil}. Thus, for each $g\in G$, we obtain an isomorphism $\alpha_g:A(s(g))\rightarrow A(r(g))$ between fibres of the bundle $\AA$.

Now, suppose that $G$ is locally compact and Hausdorff, so that it admits a Haar system $\{\lambda^u\}_{u\in G^{(0)}}$. Consider the pull-back bundle $r^*\AA$ of $\AA$ by the range map $r:G\rightarrow G^{(0)}$. This is a bundle over $G$ with fibres $r^*\AA(g)\cong A(r(g))$. We equip the set of continuous compactly supported sections $\Gamma_c(G,r^*\AA)$ with the operations
\[f\ast g(h)=\int_Gf(t)\alpha_t(g(t^{-1}h))d\lambda^{r(h)}(t)\quad \textnormal{and}\quad f^*(h)=\alpha_h(f(h^{-1})^*).\]
The (reduced) groupoid crossed product will be the completion of $\Gamma_c(G,r^*\AA)$ in the (reduced) norm arising from a certain class of representations, which we now detail. We remark that as in \cite{MuhWil} we could instead immediately define a (non-$C^*$) norm $\|\cdot\|_I$ on $\Gamma_c(G,r^*\AA)$ directly and define the crossed product to be the enveloping $C^*$-algebra. However, we will need to make use of the representations we are about to define sooner or later, so we give a full exposition here. As in the group case, we will obtain a representation of $\Gamma_c(G,r^*\AA)$ by combining representations of $A$ and of $G$ into an integrated form. Due to the fibrations involved in the groupoid case, we must represent onto bundles of Hilbert spaces. 
\begin{definition}[{\cite[Definition~3.61]{Goehle}}]\label{analyticBorelHilbertbundle}
Suppose $\hh=\{H(y)\}_{y\in Y}$ is a collection of separable nonzero complex Hilbert spaces indexed by an analytic Borel space $Y$. The \textit{total space} is defined to be
\[Y\ast\hh\coloneqq\{(y,h)\mid y\in Y, h\in H(y)\}\]
and is equipped with the obvious projection map $\pi:Y\ast\hh\rightarrow Y$. We say that $Y\ast\hh$ is an \textit{analytic Borel Hilbert bundle} if it is equipped with a $\sigma$-algebra which makes it an analytic Borel space (see \cite[Appendix~D]{Williams}) such that
\begin{enumerate}[label=\alph*)]
\item $\pi$ is measurable; and
\item there is a sequence $\{f_n\}_{n\in\N}$ of sections (called a \textit{fundamental sequence} for $Y\ast\hh$) such that
\begin{enumerate}[label=(\roman*)]
\item the maps $\bar{f_n}:Y\ast\hh\rightarrow\C$ given by
\[\bar{f_n}(y,h)=\langle f_n(y),h\rangle_{H(y)}\]
are measurable (with respect to the Borel $\sigma$-algebra on $\C$) for each $n$;
\item for each $n$ and $m$, the map $Y\rightarrow\C$ given by
\[y\mapsto\langle f_n(y),f_m(y)\rangle_{H(y)}\]
is measurable; and
\item the collection of functions $\{\pi\}\cup\{\bar{f_n}\}_{n\in\N}$ separates points of $Y\ast\hh$. 
\end{enumerate}
\end{enumerate}
The set of measurable sections of an analytic Borel Hilbert bundle $Y\ast\hh$ is denoted by $B(Y\ast\hh)$. 
\end{definition}
\begin{notation}
Given Hilbert spaces $H_1$ and $H_2$, denote the the collection of unitary transformations $U:H_1\rightarrow H_2$ by $U(H_1,H_2)$.
\end{notation}
\begin{definition}
If $Y\ast\hh$ is an analytic Borel Hilbert bundle then its \textit{isomorphism groupoid} is defined as
\[\textnormal{Iso}(Y\ast\hh)\coloneqq\{(x,V,y)\mid V\in U(H(y),H(x))\}\]
equipped with the smallest $\sigma$-algebra such that the map $(x,V,y)\mapsto \langle Vf(y),g(x)\rangle_{H(x)}$ is measurable for all measurable sections $f,g\in B(Y\ast\hh)$. 

We define the set of composable pairs as
\[\textnormal{Iso}(Y\ast\hh)^{(2)}\coloneqq\{((x,V,y),(w,U,z))\in\textnormal{Iso}(Y\ast\hh)\times\textnormal{Iso}(Y\ast\hh)\mid y=w\}\]
and the operations are given by
\[(x,V,y)(y,U,z)=(x,VU,z)\quad\textnormal{and}\quad(x,V,y)^{-1}=(y,V^*,x).\]
\end{definition}
\begin{definition}\label{gpoidrep}
Let $G$ be a locally compact \'etale groupoid. A \textit{groupoid representation} of $G$ is a triple $(\mu, G^{(0)}\ast\hh,U)$ where $\mu$ is a (finite) Radon measure on $G^{(0)}$, $G^{(0)}\ast\hh$ is an analytic Borel Hilbert bundle, and $U:G\rightarrow\textnormal{Iso}(G^{(0)}\ast\hh)$ is a measurable groupoid homomorphism (with respect to the Borel $\sigma$-algebra on $G$) such that, for each $g\in G$, there exists a unitary $U_g: H(s(g))\rightarrow H(r(g))$ such that $U(g)=(r(g),U_g,s(g))$. 
\end{definition}
\begin{definition}[{\cite[Definition~3.80]{Goehle}}]\label{directintegralbundle}
Suppose that $Y\ast\hh$ is an analytic Borel Hilbert bundle and $\mu$ is a measure on $Y$. Consider
\[\LL^2(Y\ast\hh,\mu)\coloneqq\left\{f\in B(Y\ast\hh)\ \middle|\ \int_Y\|f(y)\|^2_{H(y)}d\mu(y)<\infty\right\}.\]
Equip $\LL^2(Y\ast\hh,\mu)$ with the operations of pointwise addition and scalar multiplication and let $L^2(Y\ast\hh,\mu)$ be the quotient of $\LL^2(Y\ast\hh,\mu)$ such that sections which agree $\mu$-almost everywhere are identified. When equipped with the operations inherited from $\LL^2(Y\ast\hh,\mu)$ and the inner product
\[\langle f,g\rangle=\int_Y\langle f(y),g(y)\rangle_{H(y)}d\mu(y),\]
$L^2(Y\ast\hh,\mu)$ becomes a Hilbert space, which we call the \textit{direct integral} of $\hh$ with respect to $\mu$. 
\end{definition}
\begin{definition}[{\cite[Definition~3.85]{Goehle}}]
Let $Y\ast\hh$ be an analytic Borel Hilbert bundle and $\mu$ a finite measure on $Y$. An operator $T$ on $L^2(Y\ast\hh,\mu)$ is called \textit{diagonal} if there exists a bounded measurable function $\phi: Y\rightarrow\C$ such that
\[T(h)(y)=\phi(y)h(y)\]
for $\mu$-almost every $y\in Y$. Given such a function $\phi$, the associated diagonal operator is denoted by $T_\phi$. 
\end{definition}
\begin{definition}[{\cite[Definition~3.98]{Goehle}}]
Let $Y$ be a second countable locally compact Hausdorff space, $A$ a $C_0(Y)$-algebra, $Y\ast\hh$ an analytic Borel Hilbert bundle, and $\mu$ a finite Borel measure on $Y$. We say that a representation $\pi:A\rightarrow B(L^2(Y\ast\hh,\mu))$ is \textit{$C_0(Y)$-linear} if
\[\pi(\phi\cdot a)=T_\phi\pi(a)\]
for every $a\in A$ and $\phi\in C_0(Y)$. 
\end{definition}
\begin{definition}
Suppose $Y\ast\hh$ is an analytic Borel Hilbert bundle with a fundamental sequence $\{f_n\}$. A family of bounded linear operators $T(y):H(y)\rightarrow H(y)$ is a \textit{Borel field of operators} if the map
\[y\mapsto\langle T(y)(f_n(y)), f_m(y)\rangle_{H(y)}\]
is measurable with respect to the Borel $\sigma$-algebra on $\C$ for all $m$ and $n$. 
\end{definition}
The next two propositions detail the formation of a $C_0(Y)$-linear representation of $A$ onto $L^2(Y\ast\hh,\mu)$ from representations of each $A(y)$ onto $H(y)$, and the decomposition of such a representation back into its fibrewise representations.
\begin{prop}[{\cite[Proposition~3.93]{Goehle}}]
Suppose $Y$ is a second countable locally compact Hausdorff space, $A$ a separable $C_0(Y)$-algebra, $Y\ast\hh$ an analytic Borel Hilbert bundle, and $\mu$ a $\sigma$-finite Borel measure on $Y$. Given a collection of representations $\{\pi_y:A(y)\rightarrow B(H(y))\mid y\in Y\}$ such that for each $a\in A$ the set $\{\pi_y(a(y))\mid y\in Y\}$ is a Borel field of operators, we can form a representation
\[\pi=\int_Y^\oplus \pi_y d\mu(y)\] 
of $A$ on $L^2(Y\ast\hh,\mu)$ called the \textnormal{direct integral}, and defined for $a\in A$ by
\[\pi(a)=\int_Y^\oplus \pi_y(a(y))d\mu(y).\]
\end{prop}
\begin{prop}[{\cite[Proposition~3.99]{Goehle}}]
Suppose $Y$ is a second countable locally compact Hausdorff space, $A$ a separable $C_0(Y)$-algebra, $Y\ast\hh$ an analytic Borel Hilbert bundle, and $\mu$ a finite Borel measure on $Y$. Given a $C_0(Y)$-linear representation $\pi:A\rightarrow L^2(Y\ast\hh,\mu)$, there exists for each $y\in Y$ a (possibly degenerate) representation $\pi_y:A(y)\rightarrow B(H(y))$ such that for each $a\in A$, the set $\{\pi_y(a(y))\}_{y\in Y}$ is an essentially bounded Borel field of operators, and
\[\pi=\int_Y^\oplus\pi_y d\mu(y).\]
\end{prop}
\begin{definition}[{\cite[Definition~3.70]{Goehle}}]\label{modular}
Suppose $G$ is a locally compact Hausdorff groupoid with Haar system $\lambda=\{\lambda^u\}$, where $\lambda^u$ is supported on $G^u=r^{-1}(u)$. Denote the images of these measures under the inverse map of the groupoid by $\lambda_u=(\lambda^u)^{-1}$, so that $\lambda_u$ is supported on $G_u=s^{-1}(u)$. Given a Radon measure $\mu$ on $G^{(0)}$, we define the \textit{induced measures} $\nu$ and $\nu^{-1}$ to be the Radon measures on $G$ defined by the equations
\[\nu(f)\coloneqq\int_{G^{(0)}}\int_G f(g)d\lambda^u(g)d\mu(u),\quad\textnormal{and}\]
\[\nu^{-1}(f)\coloneqq\int_{G^{(0)}}\int_G f(g)d\lambda_u(g)d\mu(u)\]
for all $f\in C_c(G)$. We say that the measure $\mu$ is \textit{quasi-invariant} if $\nu$ and $\nu^{-1}$ are mutually absolutely continuous. In this case we write $\Delta$ for the Radon-Nikodym derivative $d\nu/d\nu^{-1}$, and call it the \textit{modular function} of $\mu$. 
\end{definition}

\begin{definition}
Suppose $(A,G,\alpha)$ is a separable groupoid dynamical system, and that $\mu$ is a quasi-invariant measure on $G^{(0)}$. A \textit{covariant representation} $(\mu,G^{(0)}\ast\hh,\pi,U)$ of $(A,G,\alpha)$ consists of a unitary representation $(\mu,G^{(0)}\ast\hh,U)$ of $G$, and a $C_0(G^{(0)})$-linear representation $\pi:A\rightarrow B(L^2(G^{(0)}\ast\hh,\mu))$. We require that if $\{\pi_u\}$ is a decomposition of $\pi$, and $\nu$ is the measure induced by $\mu$, there exists a $\nu$-null set $N\subset G$ such that for all $g\notin N$ we have
\[U_g\pi_{s(g)}(a)=\pi_{r(g)}(\alpha_g(a))U_g\]
for every $a\in A(s(g))$. 
\end{definition}
\begin{remark}
We will usually simply denote the covariant representation by $(\pi, U)$ and take for granted that there is a quasi-invariant measure $\mu$ and a Borel Hilbert bundle $G^{(0)}\ast\hh$ in the background. 
\end{remark}
\begin{definition}
Suppose $(A,G,\alpha)$ is a separable dynamical system and that $(\mu, G^{(0)}\ast\hh,\pi, U)$ is a covariant representation. Let $\pi=\int_{G^{(0)}}^\oplus \pi_u du$ be a decomposition of $\pi$. Then there is a $I$-norm decreasing nondegenerate $*$-representation $\pi\rtimes U$ of $\Gamma_c(G,r^*\AA)$ on $L^2(G^{(0)}\ast\hh,\mu)$ called the \textit{integrated form} of $(\pi,U)$, and given by
\[\pi\rtimes U(f)(h)(u)=\int_G\pi_u(f(g))U_gh(s(g))\Delta(g)^{-\frac{1}{2}}d\lambda^u(g)\]
for $f\in\Gamma_c(G,r^*\AA)$, $h\in\LL^2(G^{(0)}\ast\hh,\mu)$ and $u\in G^{(0)}$. 
\end{definition}
With all this theory in hand, we can define the full crossed product as the completion of $\Gamma_c(G,r^*\AA)$ in the universal norm 
\[\|f\|\coloneqq\sup\{\|\pi\rtimes U(f)\|\mid(\pi,U)\textnormal{ is a covariant representation of }(A,G,\alpha)\}.\]
We will instead work with the reduced crossed product, for which we need to define left-regular representations. Begin with a separable groupoid dynamical system $(A,G,\alpha)$, and a nondegenerate $C_0(G^{(0)})$-linear representation $\pi:A\rightarrow B(L^2(G^{(0)}\ast\hh,\mu))$ with decomposition
\[\pi=\int^\oplus_{G^{(0)}}\pi_ud\mu(u)\]
so that $\pi_u:A(u)\rightarrow B(L^2((G^{(0)}\ast\hh)|_u,\mu))$.
As in \cite[Example~3.84]{Goehle}, form the pull-back bundle
\[s^*(G^{(0)}\ast\hh)=\{(g,h)\mid h\in H(s(g))\}.\]
Let $\nu$ be the induced measure, as in Definition~\ref{modular}. By \cite[Theorem~I.5]{Williams}, we can use the source map of $G$ to decompose $\nu$ into finite measures $\nu_u$ on $G_u$. Denote by $\nu^u$ the image of $\nu_u$ under inversion in $G$. For each $u\in G^{(0)}$ we can obtain a separable Hilbert space $K(u)=L^2(s^*(G^{(0)}\ast\hh)|_{G^u},\nu^{u})$, by considering the integrable sections of the bundle $s^*(G^{(0)}\ast\hh)$ which are supported on $G^u$. We use these as fibres to form a Borel Hilbert bundle $G^{(0)}\ast\kk$. Denote by $\Delta$ the modular function which arises from $\mu$. As in \cite[Example~3.105]{Goehle}, for each $g\in G$, one can obtain a unitary
\[\lambda_g:L^2(s^*(G^{(0)}\ast\hh)|_{G^{s(g)}},\nu^{s(g)})\rightarrow L^2(s^*(G^{(0)}\ast\hh)|_{G^{r(g)}},\nu^{r(g)})\]
by defining $\lambda_gf(h)=\Delta(g)^{1/2}f(g^{-1}h)$ for each $h\in G^{r(g)}$. 

Also, define $\tilde\pi:A\rightarrow B(L^2(G^{(0)}\ast\kk,\mu))$ by
\[(\tilde\pi(a)f(u))(g)=\pi_{s(g)}(\alpha_g^{-1}(a(u)))\left( f(u)(g)\right)\]
where $a\in A$, $f\in L^2(G^{(0)}\ast\kk,\mu)$, $u\in G^{(0)}$ and $g\in G^u$.

One can show that $(\tilde\pi, \lambda)$ defines a covariant representation of $(A,G,\alpha)$. Covariant representations of this special form are known as \textit{left regular representations}. We now discuss the integrated form $\tilde\pi\rtimes \lambda$ of such representations, which will be representations of $\Gamma_c(G,r^*\AA)$ on $L^2(s^*(G^{(0)}\ast\hh),\nu^{-1})$. In fact, instead of working directly with the integrated form $\tilde\pi\rtimes\lambda$, we work with the following representation, which is unitarily equivalent to the integrated form \cite[Example~3.122 and Remark~3.123]{Goehle}:
\[L_\pi(f)h(g)=\int_{G^{r(g)}} \pi_{s(g)}(\alpha_g^{-1}(f(t)))h(t^{-1}g)d\lambda^{r(g)}(t),\]
where $f\in\Gamma_c(G,r^*\AA)$. $h\in L^2(s^*(G^{(0)}\ast\hh),\nu^{-1})$, and $g\in G$.

Due to the equivalence of $L$ and $\tilde\pi\rtimes\lambda$, we refer to $L_\pi$ as the \textit{integrated left regular representation} associated to $\pi$. When the choice of $C_0(G^{(0)})$-linear representation is clear, we will denote this representation simply by $L$. 
\begin{definition}
Suppose $(A,G,\alpha)$ is a separable groupoid dynamical system. We define the \textit{reduced norm} on $\Gamma_c(G,r^*\AA)$ by
\[\|f\|_r\coloneqq\sup\{\|L(f)\|\mid L\textnormal{ is an integrated left regular representation of }(A,G,\alpha)\}.\]
The completion of $\Gamma_c(G,r^*\AA)$ in this norm is a $C^*$-algebra called the \textit{reduced groupoid crossed product} of $A$ by $G$, and is denoted $A\rtimes_{\alpha,r}G$, or just $A\rtimes_r G$ if the choice of action is clear.
\end{definition}

\label{condexp}
For the remainder of the section, we extend some results of \cite{Phillips} to the groupoid setting. We will need the following definition, which is introduced in \cite{Cuntz} (see also \cite[Remark~1.2]{Phillips} for this formulation).

\begin{definition}\label{Cuntzsubequivalence}
Let $A$ be a $C^*$-algebra and let $a,b\in A$ be positive elements. We say that $a$ is \textit{Cuntz subequivalent to $b$ over $A$}, and write $a\precsim_A b$, if there exists a sequence $(r_n)_{n\in\N}$ in $A$ with $a=\lim_{n\rightarrow\infty}r_nbr_n^*$. 
\end{definition}

We aim to show (Lemma~\ref{Phillips2}) that for any positive element $a\in C(X)\rtimes_r G$, we can find a positive element of $C(X)$ which is Cuntz subequivalent to $a$. This will be useful to prove $\ZZ$-stability in Theorem~\ref{tilingZstable}.

We make a remark on notation. In the sequel, it will be important to distinguish between a section of the pull-back bundle, $f\in\Gamma_c(G,r^*\AA)$, and the images of elements of $G$ under this section, $f_g\in C(r_X^{-1}(r(g)))$. To do this, we will denote the section $g\mapsto f_g$ for each $g\in G$ by $\sum_{g\in G}f_gU_g$. Observe that this notation is highly suggestive of the image of the section under the integrated form of some covariant representation $(\pi, U)$, and, in fact, we can multiply sections as follows:
\[\left(\sum_{g\in G}a_gU_g\right)\left(\sum_{h\in G}b_hU_h\right)=\sum_{(g,h)\in G^{(2)}}a_gU_gb_hU_h=\sum_{(g,h)\in G^{(2)}}a_g\alpha_g(b_h)U_{gh}.\]

\begin{lemma}\label{gpoidconditionalexp}
Let $f\in\Gamma_c(G,r^*\AA)$, and write $f=\sum_{g\in G}f_gU_g$. Then the map
\[E_\alpha(f)\coloneqq\sum_{u\in G^{(0)}}f_u\]
extends to a faithful conditional expectation $E:C(X)\rtimes_rG\rightarrow C(X)$. 
\end{lemma}
\begin{proof}
The proof of most of the properties is routine, except perhaps for faithfulness, for which the reader is referred to \cite[Lemma~1.2.1]{Sierakowski}. We check that the image of $E$ is contained within $C(X)$. Given $f\in\Gamma_c(G,r^*\AA)$, write $f=\sum_{g\in G}f_gU_g$. By \cite[Proposition~4.38]{Goehle}, there is an isomorphism $\iota:C_c(G\ltimes X)\rightarrow\Gamma_c(G,r^*\AA)$, with inverse $j$ given by $j(f)(g,x)=f_g(gx)$. Since $G$ is \'etale and Hausdorff, $G^{(0)}$ is clopen in $G$. 
This implies that $f|_{G^{(0)}}\in\Gamma_c(G,r^*\AA)$, so that $j(f|_{G^{(0)}})\in C_c(G\ltimes X)$. In fact, because $(f|_{G^{(0)}})_g=0_{r_X^{-1}(r(g))}$ unless $g\in G^{(0)}$, observe that $j(f|_{G^{(0)}})$ is supported on $G^{(0)}\ast X$, which is homeomorphic to $X$ via $\Phi(x)=(r_X(x),x)$. Thus, under the identification of $G^{(0)}\ast X$ with $X$, we see that $j(f|_{G^{(0)}})|_{G^{(0)}\ast X}\in C(X)$. On the other hand, for each $x\in X$,
\[E(f)(x)=f_{r_X(x)}(x)=j(f|_{G^{(0)}})(r_X(x), x)=j(f|_{G^{(0)}})(\Phi(x)),\]
so $E(f)=j(f|_{G^{(0)}})\circ\Phi$, which shows that $E(f)\in C(X)$ too. 
\end{proof}
We remark for later that an almost identical argument works to show that for any clopen bisection $B\subset G$, the sum $\sum_{g\in B}f_g$ defines an element of $C(X)$. The only modification that needs to be made is in the final equality, where we must precompose $j(f|_B)$ with the homeomorphism $(r|_{B\ast X})^{-1}$ arising from $r:G\ltimes X\rightarrow X$, to obtain 
\[\sum_{g\in B}f_g=j(f|_B)\circ (r|_{B\ast X})^{-1}\circ \Phi\in C(X).\]
\begin{lemma}[{c.f. \cite{Phillips}~Lemma~7.8}]\label{Phillips1}
Let $G$ be a locally compact Hausdorff \'etale groupoid acting freely and continuously on a compact Hausdorff space. Let $a\in\Gamma_c(G,r^*\AA)\subset C(X)\rtimes G$, and let $\epsilon>0$. Then there exists $f\in C(X)$ such that $0\leq f(x)\leq 1$ for every $x\in X$, $fa^*af\in C(X)$, and $\|fa^*af\|\geq\|E(a^*a)\|-\epsilon$. 
\end{lemma}
\begin{proof}
Write $b=a^*a$. If $\|E(b)\|\leq\epsilon$, then we can simply take $f=0$, so suppose there exists $x\in X$ such that $|E(b)(x)|>\epsilon$. 
Since $a$ is a compactly supported section, and the product of compact subsets of $G$ is always compact,
there exists a compact subset $K\subset G$, and $b_g\in C(r_X^{-1}(r(g)))$ for each $g\in K$, such that $b=\sum_{g\in K}b_gU_g$. 
Since $b$ is positive, $E(b)$ is too.
Let
\[U=\{x\in X\mid E(b)(x)>\|E(b)\|-\epsilon\},\]
which is a nonempty open subset of $X$. We wish to find a nonempty open set $W\subset U$ such that the sets in the collection $\{gW\mid g\in K\}$ are pairwise disjoint. 

Notice that for $g,h\in K$, $gW\cap hW\neq\emptyset$ if and only if there exist $w_1,w_2\in W$ such that $gw_1=hw_2$ if and only if $w_1=g^{-1}hw_2$ if and only if $W\cap K^{-1}KW\neq\emptyset$, so it is equivalent to ask that whenever $g,h\in K$ are such that $g^{-1}h\notin G^{(0)}$, we have $W\cap g^{-1}hW=\emptyset$. 

Since $G$ is \'etale, the topology on $G$ has a base of open bisections $\{B_i\}_{i\in I}$. Therefore, $\{B_i\}_{i\in I}$ is an open cover of $K^{-1}K$, which is the product of compact subsets in an \'etale groupoid, and so is compact itself by \cite[Lemma~5.2]{GWY}. 
So, there exists a finite subcover $\{B_1,\ldots,B_N\}$ of $K^{-1}K$. In fact, since $G$ is \'etale and Hausdorff, $G^{(0)}$ is a clopen bisection. Therefore, by putting $B_0=G^{(0)}$ and $B_k=B_k\setminus G^{(0)}$ for each $k\in\{1,\ldots,N\}$, then redefining $N$ and relabelling, we may assume that $B_1=G^{(0)}$ and that $B_i\cap G^{(0)}=\emptyset$ for each $i\in\{2,\ldots,N\}$. We will construct an open $W\subset U$ such that $W\cap \bigcup_{k=2}^NB_kW=\emptyset$. 

Now, we wish to find a subset $V\subset U$ such that, for each $k\in\{1,\ldots,N\}$, either $r_X(V)\subset s(B_k)$, or $r_X(V)\cap s(B_k)=\emptyset$. In other words, for each $k$, either every element of $V$ will be acted upon by $B_k$, or none of them will be. 
We will construct a nested sequence of nonempty open sets $U\supset V_1\supset V_2\supset\cdots\supset V_N=V$ so that, for each $k\in\{1,\ldots,N\}$, and each $1\leq j\leq k$, either $r_X(V_k)\subset s(B_j)$, or $r_X(V_k)\cap s(B_j)=\emptyset$. Observe that since $B_1=G^{(0)}$, the choice $V_1=U$ works.

Now, suppose we have constructed open sets $U=V_1\supset V_2\supset\cdots\supset V_{k-1}$ as above for some $2\leq k\leq N$, and proceed in cases as follows. 
\begin{enumerate}[label=Case \arabic*:]
\item If $r_X(V_{k-1})\cap s(B_k)=\emptyset$, then set $V_k=V_{k-1}$.

\item If $r_X(V_{k-1})\cap s(B_k)\neq\emptyset$, then set $V_k=V_{k-1}\cap r_X^{-1}(s(B_k))$. Observe that $V_k$ is open because $r_X^{-1}(s(B_k))$ is open, since $B_k$ is open, $s$ is an open map, and $r_X$ is continuous.
\end{enumerate}

Choose any $x\in V=V_N$. For each $k\in\{1,\ldots,N\}$, since $B_k$ is a bisection, there is at most one element $g\in B_k$ such that $s(g)=r_X(x)$, so the set $\bigcup_{k=1}^NB_kx\subset X$ is finite with cardinality at most $N$. 
Since the action is free, and since $B_k\cap G^{(0)}=\emptyset$ whenever $2\leq k\leq N$, for each $g\in \bigcup_{k=2}^NB_kr_X(x)$ we have $gx\neq x$. Since $X$ is Hausdorff, for each such $gx=y$, we can find disjoint open subsets $x\in V_y$ and $y\in Z_y$. Let $V'=V\cap\left(\bigcap_{y\in \bigcup_{k=2}^NB_kx}V_y\right)$, which is nonempty since $x\in V'$, and is open in $X$ because it is the intersection of finitely many open sets. Observe also that $V'\cap Z_y=\emptyset$ for each $y\in \bigcup_{k=2}^NB_kx$. 

We now wish to find a second finite sequence of nested open subsets $V'\supset W_1\supset W_2\supset\cdots\supset W_N\ni x$, such that, for each $k\in\{2,\ldots,N\}$, we have $W_k\cap gW_k=\emptyset$ for every $g\in\bigcup_{j=2}^kB_j$. Then $W=W_N$ will be such that $gW\cap hW=\emptyset$ for all $g\neq h\in K$, as we wanted. Start with $W_1=V'$, and proceed inductively as follows.

Suppose that we have constructed $V'=W_1\supset\cdots\supset W_{k-1}\ni x$ as above for some $2\leq k\leq N$, and construct $W_k$ by proceeding in cases as follows. First, observe that since $W_{k-1}\subset V$, either $r_X(W_{k-1})\subset s(B_k)$, or $r_X(W_{k-1})\cap s(B_k)=\emptyset$. 
\begin{enumerate}[label=Case \arabic*:]
\item Suppose $r_X(W_{k-1})\cap s(B_k)=\emptyset$. This means that $B_kW_{k-1}=\emptyset$, so we have $gW_{k-1}=\emptyset$ for each $g\in B_k$. We assumed that for each $g\in\bigcup_{j=2}^{k-1}B_j$ we have $W_{k-1}\cap gW_{k-1}=\emptyset$, and so for each $g\in\bigcup_{j=2}^k B_j$ we still have $W_{k-1}\cap gW_{k-1}=\emptyset$. Thus, we see that the choice $W_k=W_{k-1}$ is suitable. 

\item Suppose $r_X(W_{k-1})\subset s(B_k)$. Then, since $x\in W_{k-1}$, there exists $g\in B_k$ such that $s(g)=r_X(x)$. Write $y=gx$. Since $B_k$ is open in $G$, and since $G$ is \'etale, we see by Lemma~\ref{actiontopology} that $B_kW_{k-1}$ is an open subset of $X$. Thus, $Z_y$ and $B_kW_{k-1}$ are open subsets of $X$ which both contain $y$, so that $Z_y\cap B_kW_{k-1}$ is a nonempty open set in $X$. Put $W_k=B_k^{-1}(Z_y\cap B_kW_{k-1})\subset W_{k-1}$. Observe that $x\in W_k$, and $W_k$ is open by Lemma~\ref{actiontopology}, because $B_k^{-1}$ is open in $X$. All we have to check is that for each $g\in B_k$ we have $W_k\cap gW_k=\emptyset$ (note that this automatically holds for $g\in\bigcup_{j=2}^{k-1}B_j$, because $W_k\subset W_{k-1}$). Indeed, $gW_k\subset B_kW_k\subset Z_y$, which is disjoint from $V'\supset W_k$, as required. 
\end{enumerate}

So, construct an open set $W\subset U$ such that the collection $\{gW\mid g\in K\}$ is pairwise disjoint. Note in particular that, by construction of $K$, $r(K)\subset K$, so that $r(g)W$ and $gW$ are disjoint for each $g\in K$. Choose $x_0\in W$, and let $f\in C(X)$ satisfy $0\leq f(x)\leq 1$ for every $x\in X$, $\textnormal{supp}(f)\subset W$, and $f(x_0)=1$. Then
\[fbf=\sum_{g\in K}fb_gU_gf=\sum_{g\in K}fb_g\alpha_g(f)U_g.\]
Since $\textnormal{supp}(f)\subset W$ and $\textnormal{supp}(\alpha_g(f))\subset gW$, observe that $fb_g\alpha_g(f)=b_gf\alpha_g(f)$ is supported inside $W\cap gW$. Since $gW\subset r_X^{-1}(r(g))$, we have $W\cap gW=r(g)(W\cap gW)=r(g)W\cap gW$, and we thus see that $fb_g\alpha_g(f)$ is supported within $r(g)W\cap gW$.
So, for $g\in K\setminus G^{(0)}$, our construction of $W$ provides $fb_g\alpha_g(f)=0$. 
Thus, we have
\[fbf=\sum_{u\in K\cap G^{(0)}} fb_u\alpha_u(f)U_u=\sum_{u\in K\cap G^{(0)}}fb_uf|_{r_X^{-1}(u)},\]
where we have identified the central expression with a complex-valued function on $X$.  
Since $\textnormal{supp}(b_u)\subset r_X^{-1}(u)$, we see that $fb_uf|_{r_X^{-1}(u)}=fb_uf$, so that
\[fbf=f\left(\sum_{u\in G^{(0)}}b_u\right)f=fE(b)f\in C(X).\]
In addition, we have
\[\|fbf\|\geq |fbf(x_0)|=|f(x_0)b_{r_X(x_0)}(x_0)f(x_0)|=|b_{r_X(x_0)}(x_0)|=|E(b)(x_0)|>\|E(b)\|-\epsilon,\]
where the final inequality follows from the fact that $x_0\in U$. 
\end{proof}

\begin{lemma}[c.f. {\cite[Lemma~7.9]{Phillips}}]\label{Phillips2}
Let $G\curvearrowright X$ be a free continuous action of a locally compact Hausdorff \'etale groupoid on a compact Hausdorff space. Let $B\subset C(X)\rtimes_r G$ be a unital subalgebra such that
\begin{enumerate}[label=(\roman*)]
\item $C(X)\subset B$; and
\item $B\cap\Gamma_c(G,r^*\AA)$ is dense in $B$.
\end{enumerate}
Let $a\in B_+\setminus\{0\}$. Then there exists $b\in C(X)_+\setminus\{0\}$ such that $b\precsim_B a$, where $\precsim_B$ denotes the relation of Cuntz subequivalence over $B$.
\end{lemma}
\begin{proof}
Without loss of generality, we may assume $\|a\|\leq1$. Since the conditional expectation $E: C(X)\rtimes_r G\rightarrow C(X)$ is faithful, $E(a)\in C(X)$ is a nonzero positive element. Let $\epsilon=\frac{1}{8}\|E(a)\|$ and note that $\epsilon\leq1$, since $\|E(a)\|\leq\|a\|$. Since $B\cap\Gamma_c(G,r^*\AA)$ is dense in $B$, choose $c\in B\cap\Gamma_c(G,r^*\AA)$ such that $\|c-a^{1/2}\|<\epsilon$. Without loss of generality, we may choose $c$ such that $\|c\|\leq1$. Notice that
\begin{align*}
\epsilon^2>\|c-a^{1/2}\|^2&=\|(c-a^{1/2})(c-a^{1/2})^*\|\\
&=\|cc^*+a-ca^{1/2}-c^*a^{1/2}\|\\
&\geq\|cc^*-a\|-\|2a-ca^{1/2}-c^*a^{1/2}\|\\
&\geq\|cc^*-a\|-\|a^{1/2}\|\|2a^{1/2}-c-c^*\|,
\end{align*}
so that
\begin{align*}
\|cc^*-a\|&<\epsilon^2+\|a^{1/2}\|\|a^{1/2}-c\| + \|a^{1/2}\|\|(a^{1/2}-c)^*\|<3\epsilon.
\end{align*}
Similarly, $\|c^*c-a\|<3\epsilon$. Since $E$ is norm decreasing, we see that $\|E(c^*c-a)\|<3\epsilon$, which implies that $\|E(a)\|<\|E(c^*c)\|+3\epsilon$.
Applying Lemma~\ref{Phillips1} with $c$ in place of $a$, and with the $\epsilon$ defined above, yields $f\in C(X)$ with $0\leq f\leq 1$, which satisfies
\[\|fc^*cf\|\geq\|E(c^*c)\|-\epsilon>\|E(a)\|-4\epsilon=8\epsilon-4\epsilon=4\epsilon.\]
Therefore, $(fc^*cf-3\epsilon)_+$ is a nonzero positive element of $C(X)$. To conclude the proof, we apply \cite[Lemma~1.4(6)]{Phillips} at the first step, \cite[Lemma~1.7]{Phillips} and the fact that $cf^2c^*\leq cc^*$ because $0\leq f\leq1$ at the second step, and \cite[Lemma~1.4(10)]{Phillips} and $\|cc^*-a\|<3\epsilon$ at the final step, to see that
\[(fc^*cf-3\epsilon)_+\sim_B(cf^2c^*-3\epsilon)_+\precsim_B(cc^*-3\epsilon)_+\precsim_B a.\qedhere\]\label{endcondexp}
\end{proof}

\section{Tilings and their dynamics}
\label{tilings} 

In this section, we define tilings of $\R^d$ and introduce the properties on tilings that are used throughout this paper. These assumptions lead to a compact space of tilings with a free and minimal $d$-dimensional flow, giving rise to a dynamical system of tilings. Following Kellendonk \cite{Kel}, we consider a closed subset of this dynamical system that lends itself to defining a rather tractable groupoid and its corresponding $C^*$-algebra. A more complete exposition of the material presented here can be found in \cite{Kel, KP, Whit}.

We begin by describing the basic building blocks of the tilings we consider.

\begin{definition}
A \emph{prototile} $p$ is a labelled subset of $\R^d$ that is homeomorphic to the closed unit ball. A prototile set is a finite set of prototiles, typically denoted by $\PP:=\{p_1, \ldots, p_n\}$.
\end{definition}

Note that prototiles are labelled subsets of $\R^d$. So we may have two identical subsets with different labels, making them different prototiles. We will abuse notation and let the symbol $p$ denote both a subset of $\R^d$ and a labelled subset of $\R^d$.

\begin{definition}
A \textit{tiling} $T$ with prototile set $\PP$ is a set $\{t_1,t_2,\ldots\}$ of subsets of $\R^d$, called \textit{tiles}, such that
\begin{enumerate}
\item for each $i\in\{1,2,\ldots\}$, there exists $x_i\in\R^d$ and $p_i\in\PP$ such that $t_i=p_i+x_i$;
\item $\bigcup_{i=1}^\infty t_i =\R^d$; and
\item $\textnormal{int}(t_i) \cap \textnormal{int}(t_j) = \varnothing$ for $i \neq j$.
\end{enumerate}
\end{definition}

A \emph{patch} $P$ in a tiling $T$ is a finite connected subset of tiles. Given a tiling $T$, $x \in \R^d$ and $r>0$, we define the patch $T\sqcap B_r(x)$ by
\[
T\sqcap B_r(x)\coloneqq\{t\in T\mid t \cap B_r(x)\neq\emptyset\}.
\]
These types of patches will be useful for putting a metric on tilings, among other things. Given a tiling $T$ and $x \in \R^d$, we define $T+x:=\{t+x \mid t \in T\}$ and use this to define the set of tilings $T+\R^d :=\{T+x \mid x \in \R^d\}$.

\begin{definition}
Suppose $T$ is a tiling. Given $T_1,T_2 \in T+\R^d$ and $0<\varepsilon<1$, we say that $T_1$ and $T_2$ are $\varepsilon$-close if there exist $x_1, x_2\in\R^d$ such that $\|x_1\|,\|x_2\|\leq\varepsilon$, and such that
\[
(T_1+x_1)\sqcap B_{\varepsilon^{-1}}(0)=(T_2+x_2)\sqcap B_{\varepsilon^{-1}}(0).
\]
Define $d(T_1,T_2)$ to be the infimum of the set of $\varepsilon$ which satisfy this hypothesis. If no such $\varepsilon$ exists, set $d(T_1,T_2)=1$.
\end{definition}

The idea of the tiling metric is to extend the usual product metric in symbolic dynamics to the continuous situation of tilings. So two tilings are close if they agree on a large ball about the origin up to a small translation.

\begin{definition}
Suppose $T$ is a tiling. The \emph{continuous hull} $\Omega_T$ of $T$ is the completion of $T+\R^d$ in the tiling metric.
\end{definition}

We now introduce several properties of a tiling that we use throughout the paper.

\begin{definition}\label{tiling properties}
Suppose $T$ is a tiling.
\begin{enumerate}
\item\label{FLC} $T$ is said to have \emph{finite local complexity} (FLC) if, for every $R>0$, the set $\{T \sqcap B_R(x) \mid x \in \R^d\}/\R^d$ is finite.
\item\label{nonperiodic} $T$ is said to be \emph{nonperiodic} if $T+x=T$ implies $x=0$. We say $T$ is {\em aperiodic} if every tiling in the continuous hull $\Omega_T$ is nonperiodic.
\item\label{repetitive} $T$ is said to be \emph{repetitive} if, for every patch $P \subset T$, there exists $R>0$ such that, for every $x\in\R^d$, a translate of $P$ appears in $T\sqcap B_R(x)$.
\end{enumerate}
\end{definition}

If $T$ has FLC, then $\Omega_T$ is a compact Hausdorff space. In \cite{Kel}, it is shown that $\Omega_T$ is a space whose points are tilings and that $\R^d$ acts continuously on $\Omega_T$. So we think of $(\Omega_T,\R^d)$ as a dynamical system. If $T$ is aperiodic, then the action of $\R^d$ on $\Omega_T$ is free and if $T$ is repetitive, the action is minimal. In this case, the continuous hull does not depend on the original tiling $T$ used to construct it, in the sense that the continuous hull of any other tiling in $\Omega_T$ will also be $\Omega_T$. Therefore, we begin to drop the ``$T$" from the notation $\Omega_T$, and simply denote the continuous hull by $\Omega$.

One of Kellendonk's inspired simplifications of tiling spaces is a quantisation of the continuous hull \cite[Section 2.1]{Kel}. Suppose $\Omega$ is the continuous hull of a tiling with prototile set $\PP$. For each prototile $p \in \PP$, distinguish a point in its interior and denote it by $x(p)$. We call $x(p)$ a \emph{puncture}. Now suppose $T$ is a tiling in $\Omega$. Since each tile $t\in T$ satisfies $t=p+y$ for some $y\in\R^d$ and $p\in P$, we extend the punctures to tiles by $x(t)=x(p)+y$. Thus, every tile in every tiling in $\Omega$ can be punctured. We use these punctures to restrict the possible alignments of tilings by forcing the origin to lie on a puncture.

\begin{definition}[{\cite[Section 2.1]{Kel}}]
Suppose $\Omega$ is the continuous hull of an FLC tiling. The \emph{discrete hull} $\Opunc$ of a tiling consists of all tilings $T \in \Omega$ for which the origin is a puncture of some tile $t\in T$. 
\end{definition}
We now begin to assume that each prototile $p \in \PP$ lies in $\R^d$ with its puncture on the origin. Suppose $T$ is a tiling in $\Opunc$. We denote the (unique) tile with puncture on the origin by $T(0)$. We remark that $\Opunc$ is a Cantor set, as shown in \cite{KP}. A neighbourhood base for the Cantor set topology is given by sets of the form
\[
U(P,t)\coloneqq \{ T \in \Opunc \mid P-x(t)\subset T\}.
\]

\section{Tiling groupoids and Kellendonk's $C^*$-algebra of a tiling}\label{efun}

We now describe Kellendonk's construction of an \'{e}tale groupoid associated to a tiling \cite{Kel}.

Suppose $\Opunc$ is the discrete hull of a repetitive and aperiodic tiling with FLC. An \'{e}tale principal groupoid is defined using translational equivalence on $\Opunc$:
\[
\Rpunc\coloneqq\{(T,T-x(t))\mid T\in\Opunc \textnormal{ and } t\in T\}.
\]
Then $\Rpunc$ is a groupoid with inverse and partially defined product
\[
(T,T')^{-1}=(T',T) \quad \text{ and } \quad (T, T')\cdot(T',T'')\coloneqq (T, T'').
\]
The unit space of $\Rpunc$  is $\Opunc$, and the source and range maps $s,r:\Rpunc\rightarrow \Opunc$ are defined by $s(T,T^\prime)\coloneqq T^\prime$ and $r(T,T^\prime) \coloneqq T$.

There is a metric on $\Rpunc$ defined by
\[
d_R\big((T_1, T_1-x(t_1)), (T_2, T_2-x(t_2))\big)=d(T_1, T_2) + \|x(t_1)-x(t_2)\|.
\]
An equivalent topology is given as follows. Let $P$ be a patch of tiles and $t,t' \in P$, so that both $P-x(t)$ and $P-x(t')$ are patches with puncture on the origin. Consider the set of tilings
\begin{equation}
\label{V-sets}
V(P,t,t')\coloneqq \{ \big(T-(x(t')-x(t)),T)\big) \mid P-x(t)\subset T\}
\end{equation}
and let
\begin{equation}
\label{V-base}
\VV \coloneqq \{ V(P,t,t') \mid P \text{ is a patch from a tiling in } \Opunc \text{ and } t,t' \in P\}.
\end{equation}
The collection $\VV$ is a neighbourhood base of compact open sets; see for example \cite[Lemma 2.5 and 2.6]{Whit}. Lemma 2.6 in \cite{Whit} shows that the range and source maps are local homeomorphisms satisfying $s(V(P,t,t'))=U(P,t)$ and $r(V(P,t,t'))=U(P,t')$. This implies that $\Rpunc$ is an \'{e}tale groupoid.

In \cite{{Kel}, {KP}, {Whit}}, a groupoid $C^*$-algebra $\Apunc$ is associated to $\Rpunc$. Rather than go into the specifics of this construction, we present a set of elements that span a dense subalgebra of $\Apunc$, and a faithful and nondegenerate representation of $\Apunc$ on a Hilbert space. The details of this specific construction can be found in \cite[Section 2.2]{Kel}.

To construct a dense spanning class in $\Apunc$ we use the neighbourhood base $\VV$ from \eqref{V-base} to construct partial isometries. We set $e(P, t, t^\prime)=\mathbf{1}_{V(P,t,t')}$ to be the indicator function of the clopen set $V(P,t,t')$. Explicitly, for a pair $(T',T)\in\Rpunc$ we have
\begin{equation}\label{e-function}
e(P,t,t^\prime)(T',T)=\begin{cases}
1&\hbox{if } P-x(t)\subset T\text{ and }  T'=T-(x(t')-x(t))\\
0&\hbox{otherwise.}
\end{cases}
\end{equation}

The collection
\begin{equation}\label{ez-collection}
\mathcal{E} \coloneqq \{ e(P,t,t^\prime)\mid P \text{ is a patch from } \Opunc \text{ and } t,t' \in P\}
\end{equation}
generates a dense subalgebra of $C_c(\Rpunc)$, and hence of $\Apunc$. The proof of this appears in \cite[Proposition 3.3]{Whit}. Moreover, these functions have the following properties.

\begin{prop}
Suppose $\Opunc$ is the discrete hull of a repetitive and aperiodic tiling with FLC, and $\EE$ is the collection in \eqref{ez-collection}. Then 
\begin{enumerate}
\item $e(P,t,t')^*=e(P,t',t)$; and
\item $e(P,t,t')\cdot e(P',t',t'')=e(P\cup P',t,t'')$ if $P$ and $P'$ agree where they intersect.
\end{enumerate}
\end{prop}

To define a representation of $C_c(\Rpunc)$ as bounded operators on a Hilbert space, we follow the development in \cite[Section~2.2]{Kel}. Suppose $T$ is any tiling in $\Opunc$ and let us identify the tile $t \in T$ with the tiling $T-x(t) \in \Opunc$. Then the induced representation from the unit space $\pi_T:C_c(\Rpunc) \to B(\ell^2(T))$ acts on elements of the spanning family $\EE$ as follows. For $q \in T$ with $P-x(t)\subset T-x(q)$, there exists a tile $q'$ with puncture $x(q')=x(q)+(x(t')-x(t))$ such that $P-x(t')\subset T-x(q')$, and we have
\begin{equation}\label{ind rep}
\big(\pi_T(e(P,t,t'))\xi\big)(q)=\xi(q') \quad \text{ for } \xi \in \ell^2(T). 
\end{equation}
Since $T$ is aperiodic and repetitive, the induced representation extends to a faithful nondegenerate representation of $\Apunc$ on $\ell^2(T)$, as shown in \cite[Section 4]{MamWhit}. Restricting \eqref{ind rep} to delta functions $\delta_{t''} \in \ell^2(T)$ gives
\begin{equation}\label{ind rep delta}
\pi_T(e(P,t,t'))\delta_{t''}=\begin{cases}\delta_{t'''}&\hbox{if }P-x(t)\subset T-x(t'')\textnormal{ and }x(t''')=x(t'')+(x(t')-x(t))\\0&\hbox{otherwise}.\end{cases}
\end{equation}

\section{Almost finiteness of tiling groupoids}\label{sec:tilinggroupoidaf}
We now present a viewpoint of tiling groupoids as transformation groupoids associated to a special sort of groupoid action, and prove that they are almost finite in the sense of \cite[Definition~6.2]{Matui} and satisfy the diameter condition appearing in Theorem~\ref{Suzuki Analogue}. We remark that we could simply apply Lemma~\ref{amplediameter} to obtain the diameter condition, but our direct proof will be instructive, and will allow us to choose our clopen towers to have a particularly nice form.

Given a free action of $\R^d$ on a compact metrisable space $\Omega$, we will restrict the action to a special subset of $\Omega$ of the following type.
\begin{definition}[{\cite[Definition~2.1]{GMPS}}]\label{flatCantortransversal}
Let $d\in\N$. Let $\varphi$ be a free action of $\R^d$ on a compact, metrisable space $\Omega$. We call a closed subset $X\subset\Omega$ a \textit{flat Cantor transversal} if the following are satisfied. 
\begin{enumerate}[label=(\roman*)]
\item $X$ is homeomorphic to a Cantor set.
\item For any $x\in\Omega$, there exists $p\in\R^d$ such that $\varphi^p(x)\in X$. 
\item There exists a positive real number $M>0$ such that
\[C=\{\varphi^p(x)\mid x\in X, p\in B_M(0)\}\]
is open in $\Omega$, and
\[X\times B_M(0)\ni (x,p)\mapsto \varphi^p(x)\in C\]
is a homeomorphism.
\item For any $x\in X$ and $r>0$ there exists an open neighbourhood $U\subset X$ of $x$ such that $\{p\in B_r(0)\mid \varphi^p(x)\in X\}=\{p\in B_r(0)\mid \varphi^p(y)\in X\}$ for all $y\in U$. 
\end{enumerate}
\end{definition}
The following appears as \cite[Remark~6.4]{Matui} (see also the proof of \cite[Lemma~6.3]{Matui}). 
\begin{lemma}[{\cite[Remark~6.4]{Matui}}]
Suppose that $\varphi:\R^d\curvearrowright\Omega$ is a free action on a compact metrisable space, and let $X\subset\Omega$ be a flat Cantor transversal. Construct an \'etale principal groupoid $G$ as in \cite{GMPS} as follows:
\[G=\{(x,\varphi^p(x))\mid p\in\R^d\textnormal{ and } x, \varphi^p(x)\in X\}.\]
Then $G$ is almost finite.
\end{lemma}

Notice that when we consider the free action of $\R^d$ on the continuous hull of an aperiodic and repetitive tiling with FLC, and set $X=\Opunc$, the groupoid $G$ constructed in the lemma above is isomorphic to $\Rpunc$. So, in order to prove that $\Rpunc$ is almost finite, it suffices to prove the following lemma.
\begin{lemma}
Let $\Omega$ be the continuous hull of an aperiodic and repetitive tiling with FLC. Then $\Opunc\subset\Omega$ is a flat Cantor transversal. 
\end{lemma}
\begin{proof}
As shown in \cite{KP}, $\Opunc$ is homeomorphic to a Cantor set. Given $T\in\Omega$, observe that $\varphi^{-x(T(0))}(T)=T-x(T(0))\in\Opunc$, so (ii) is satisfied. 

Since $T$ has FLC, the set of punctures in any tiling in $\Omega$ forms a Delone set in $\R^d$, and in particular is uniformly discrete. Furthermore, all the patches in any tiling in $\Opunc$ already appear in the initial tiling that we chose, so we can choose the same constants which witness uniform discreteness and relative density in all of these Delone sets. Take $M=r/2$, where $r$ is the chosen constant witnessing uniform density, so that the distance between any two distinct punctures is smaller than $r$. Form $C$ as in Definition~\ref{flatCantortransversal}, and take a tiling $T\in C$. By definition, there exists a unique $p\in\R^d$ such that $\|p\|<M$ and $\varphi^{-p}(T)\in\Opunc$. Existence is clear by the definition of $C$, and uniqueness follows because if $x,y$ have $\|x\|,\|y\|<M$ and $T-x, T-y\in\Opunc$, then $\|x-y\|\leq\|x\|+\|y\|<2M=r$, which contradicts our choice of $r$ since $x$ and $y$ are puncture locations. Take $0<\epsilon<(M-\|p\|)/2$ to be small enough that $\epsilon^{-1}-\epsilon>M$. Then, when $d(T,T')<\epsilon$, we can find vectors $x,y\in\R^d$ with $\|x\|,\|y\|<\epsilon$ such that $(T-x)\sqcap B_{\epsilon^{-1}}(0)=(T'-y)\sqcap B_{\epsilon^{-1}}(0)$. Since $T-x$ has a puncture at $p-x$, with 
\[\|p-x\|\leq\|p\|+\|x\|<\|p\|+(M-\|p\|)/2<M<\epsilon^{-1},\]
we see that $T'-y$ also has a puncture at $p-x$, so that $T'$ has a puncture at $p-x+y$. Observe that
\[\|p-x+y\|\leq\|p\|+\|x\|+\|y\|\leq \|p\|+2\epsilon<\|p\|+2(M-\|p\|)/2=M,\]
and therefore $T'\in C$. This proves that $C$ is open. 

Since the action of $\R^d$ on $\Omega$ is continuous, the map defined in (iii) is clearly continuous, and is clearly surjective by definition of $C$. As in the proof above, our choice of $M$ guarantees that for each $T\in C$ there is a unique $T'\in\Opunc$ such that there exists $p\in B_M(0)$ with $\varphi^p(T')=T$, and therefore the map is injective. To conclude the proof that the map is a homeomorphism, we prove that it is open. It suffices to show that the images of subsets of the form $U(P,t)\times B_r(x)$ are open in $\Omega$. So, take $T=\varphi^p(T')$ for some $T'\in U(P,t)\subset\Opunc$ and $p\in B_r(x)\subset B_M(0)$. Choose $\epsilon<(r-\|p-x\|)/2$ small enough that $P\subset B_{\epsilon^{-1}-\epsilon-\|p\|}(x(t))$. 
Observe that $B_\epsilon(T)$ is a subset of the image of $U(P,t)\times B_r(x)$, which proves that the image is open. 

Finally, for property (iv), choose $T\in\Opunc$ and $r>0$. Let $P\subset T$ be a patch which contains $B_r(0)\subset T$, and set $U=U(P,t)$. Then whenever $T'\in U$, we see that $T$ and $T'$ agree on $B_r(0)$, and thus the set of punctures in $T$ which lie within $B_r(0)$ is the same subset of $\R^d$ as the set of punctures in $T'$ which lie within $B_r(0)$. This shows that
\[\{p\in B_r(0)\mid \varphi^p(T)\in\Opunc\}=\{p\in B_r(0)\mid \varphi^p(T')\in\Opunc\},\]
as required.
\end{proof}

\begin{lemma}\label{tilingalmostfinite}
Let $T$ be an aperiodic and repetitive tiling with FLC. Then the tiling groupoid $\Rpunc$ is almost finite. Furthermore, the subsets $F_j^{(i)}$  in the definition of the fundamental domain of the elementary subgroupoid can be chosen to be arbitrarily small in diameter.
\end{lemma}
\begin{proof}
We have already shown that $\Rpunc$ is almost finite in the sense of Definition~\ref{gpoidAF}, so all that remains is to show that we can arrange for the small diameter condition to be satisfied. We make use of the proof of Theorem~\ref{Suzuki Analogue} to phrase the almost finiteness of this groupoid in terms of clopen tower decompositions of $\Opunc$ (see Definition~\ref{defgroupoidactionaf}). We will use a similar argument as in the proof of Lemma~\ref{amplediameter} to show that we can iteratively modify any given tower decomposition to obtain the diameter condition. 

Let $C\subset\Rpunc$ be compact and open, and let $\epsilon>0$. Obtain a $(C,\epsilon)$-invariant clopen tower decomposition of $\Opunc$. 
Choose any tower $(W,S)$ in this decomposition. We aim to modify the tower by splitting it into finitely many towers which are still $(C,\epsilon)$-invariant, and such that the collection of the levels of the new towers is a partition of the union of the levels of $(W,S)$, but so that the levels of the new towers all have diameter smaller than $\epsilon$. Applying this procedure to all of the towers simultaneously will produce a clopen tower decomposition with the properties we seek.

Recall that $S$ decomposes into finitely many compact open $G$-sets $S=\bigsqcup_{k=1}^K S_k$, and first consider the $G$-set $S_1$, which moves the base of the tower to the first level. Consider the collection of patches $\{T\sqcap B_{R_1}(0)\mid T\in\Opunc\}$ for a sufficiently large $R_1$,
to be defined. 
By FLC, there are finitely many such patches $\{P_1^1,\ldots,P_{N_1}^1\}$. Notice that each patch contains the origin, and that the collection $\{U(P_n^1, P_n^1(0))\}_{n=1}^{N_1}$ is pairwise disjoint and covers $\Opunc$, where $P_n^1(0)$ denotes the (unique) tile in $P_n^1$ which contains the origin.   

Since the $G$-set $S_1$ is compact and open in $\Rpunc$, and the collection of sets $V(P,t,t')$ contained in $\Rpunc$ is a basis for the topology on $\Rpunc$, we can choose $R_1$ large enough so that there exists a subcollection $\{P_{n_1}^1,\ldots,P_{n_{M_1}}^1\}$ of $\{P_1^1,\ldots,P_{N_1}^1\}$, and a collection of tiles $\{t_{m,l}^1\}_{l=1,\ldots,L_m}\subset P_{n_m}^1$ for each $m\in\{1,\ldots,M_1\}$ such that 
\[
S_1=\bigsqcup_{m=1}^{M_1}\bigsqcup_{l=1}^{L_m} V(P_{n_m}^1, P_{n_m}^1(0), t_{m,l}^1).
\]
Since $W=s(S_1)$, this implies that the base of the tower is $W=\bigsqcup_{m=1}^{M_1} U(P_{n_m}^1, P_{n_m}^1(0))$, and that the first level is $S_1\cdot W=\bigcup_{m=1}^{M_1}\bigcup_{l=1}^{L_m}U(P_{n_m}^1,t_{m,l}^1)$. 

Since $S_1$ is a $G$-set, and since, for any fixed $m\in\{1,\ldots,M_1\}$ and $l\neq l'$, we have 
\[V(P_{n_m}^1, P_{n_m}^1(0), t_{m,l}^1)\cap V(P_{n_m}^1, P_{n_m}^1(0), t_{m,l'}^1)=\emptyset\]
and 
\[s(V(P_{n_m}^1, P^1_{n_m}(0), t^1_{m,l}))=U(P_{n_m}^1,P_{n_m}^1(0))=s(V(P^1_{n_m}, P^1_{n_m}(0), t^1_{m,l'})),\]
the injectivity of the source map of $S_1$ implies that all the tiles $t^1_{m,l}$ for $l\in\{1,\ldots,L_m\}$ must be equal, to $t^1_m\in P^1_{n_m}$, say. Then, in fact, we have $S_1 = \bigsqcup_{m=1}^{M_1}V(P^1_{n_m},P^1_{n_m}(0),t^1_m)$. 

Since, for $m\neq m'$, we have
\[V(P_{n_m}^1,P_{n_m}^1(0),t_m^1)\cap V(P_{n_{m'}}^1,P_{n_{m'}}^1(0),t_{m'}^1)=\emptyset,\]
the injectivity of the range map on $S_1$ tells us that the collection 
\[\{r(V(P^1_{n_m}, P^1_{n_m}(0), t^1_{m}))\}_{m=1,\ldots,M_1}=\{U(P^1_{n_m},t^1_{m})\}_{m=1,\ldots,M_1}\]
must be pairwise disjoint.

By choosing $R_1$ to be large enough, we may assume that we have $B_{\epsilon^{-1}}(x(t^1_{m}))\subset P^1_{n_m}$ for each $m\in\{1,\ldots,M_1\}$, so that the diameter of $U(P^1_{n_m},t^1_{m})$ is smaller than $\epsilon$. Indeed, since the vectors of translation associated to elements of $S_1$ are already prescribed above, and since enlarging $R_1$ produces a new finite collection of patches, each of which contains some patch $P_{n_m}$ from above, increasing $R_1$ corresponds to  partitioning each $V(P_{n_m}^1,P_{n_m}^1(0),t_m^1)$ into subsets $\{V(P^1_{n_m,q}, P^1_{n_m}(0), t_m^1)\}_{q=1,\ldots,Q_m}$, where $P^1_{n_m}\subset P^1_{n_m,q}$. Since $t_m^1$ does not depend on $q$, we simply ensure that  $B_{\epsilon^{-1}}(x(t_m^1))\subset P^1_{n_m,q}$.

The clopen tower $(W,S)$ now splits into the collection of ``one-level'' clopen towers $\{(W^1_m,H^1_m) \mid m=1,\ldots,M_1\}$
by taking 
\[
W^1_m\coloneqq W\cap U(P^1_{n_m},P^1_{n_m}(0))=U(P^1_{n_m},P^1_{n_m}(0))
\]
to be the base of the $m$-th tower, and $H^1_m\coloneqq S_1\cap s^{-1}(W^1_m)=V(P^1_{n_m},P^1_{n_m}(0),t^1_m)$ to be the shape of the $m$-th tower. 
The single level of this tower is $U(P^1_{n_m},t^1_m)$, and,
by our choice of $R_1$,
it has diameter smaller than $\epsilon$. 
In addition, as we saw above, 
the collection of the tower levels $\{H^1_m\cdot W^1_m\}_{m=1,\ldots,M_1}=\{U(P_m^1,t_m^1)\}_{m=1,\ldots,M_1}$ is pairwise disjoint, and so partitions the first level $S_1\cdot W=\bigsqcup_{m=1}^{M_1}U(P_{n_m}^1,t_m^1)$ of the tower $(W,S)$. So, from the ``one-level'' tower $(W,S_1)$, we have obtained a collection of towers whose levels have diameter less than $\epsilon$, and still partition $S_1\cdot W$. 

Now, we iterate this procedure. At the $k$-th step we replace $W$ by the base of each of the towers $(W_m^{k-1},H_m^{k-1})$, for $m=1,\ldots,M_{k-1}$, which were constructed in the $(k-1)$-th step, in turn, and replace $S_1$ by $S_k$. We find $R_k\geq R_{k-1}$ large enough that there exists a subcollection $\{P^k_{n_1},\ldots, P^k_{n_{M_k}}\}$ of the finite set of patches $\{P^k_1,\ldots,P^k_{N_k}\}$ of the form $T\sqcap B_{R_k}(0)$ such that $S_k=\bigsqcup_{m=1}^{M_k} V(P^k_{n_m}, P^k_{n_m}(0), t_m^k)$ (the fact that there is just one tile associated to each patch uses the fact that $S_k$ is a $G$-set, and an argument as above). 
We further choose $R_k$ to ensure that $B_{\epsilon^{-1}}(x(t^k_m))\subset P^k_{n_m}$ for each $m\in\{1,\ldots,M_k\}$. Then, to form the towers for the $k$-th step, we take bases of the form $W_m^k=U(P^k_{n_m}, P^k_{n_m}(0))$. Observe that, for each $m\in\{1,\ldots,M_k\}$, we have $W_m^k\subset W_{\tilde m}^{k-1}$ for some $\tilde m\in\{1,\ldots,M_{k-1}\}$. 
We define the shape $H_m^k=(H_{\tilde m}^{k-1}\cup S_k)\cap s^{-1}(W_m^k)$. That is, we ``pick up'' all of the arrows that our construction had previously associated to this subset, and also include the new arrows contained in the subset $S_k\subset S$ of the shape of the tower $(W,S)$ whose sources lie in this subset.

After finitely many iterations, the process terminates. At the end of the process, we have obtained an $R=R_K$ large enough to work for all of the sets $S_k$ in the construction above simultaneously. In particular, the levels of each tower $(W_m^K, H_m^K)$ for $m\in\{1,\ldots,M_K\}$ will have diameter smaller than $\epsilon$, and will partition $\bigsqcup_{k=1}^K(S_k\cdot W)$.
We claim that the shapes of the towers we constructed are still $(C,\epsilon)$-invariant. Let $m\in\{1,\ldots,M_K\}$ and let $T\in s(H_m^K)$ be arbitrary. By construction, the collection of arrows from $S$ with source $T$ is equal to the set of arrows from $H_m^K$ with source $T$. This, combined with the fact that $S$ was $(C,\epsilon)$-invariant, shows that $H_m^K$ is $(C,\epsilon)$-invariant too.
\end{proof}
\begin{remark}\label{tilingtowersbasissets}
By following the procedure in the proof above, we may assume, without loss of generality, that each tower $(W,S)$ in a clopen tower decomposition of the action $\Rpunc\curvearrowright\Opunc$ has $W=U(P,t)$ and $S=\bigsqcup_{t'\in Q}V(P,t,t')$ for some patch $P$, some $t\in P$, and some subset $Q\subset P$. 
Given any $\epsilon>0$, we may enlarge $P$ to further assume that $B_{\epsilon^{-1}}(x(t'))\subset P$ for each $t'\in Q$. 
\end{remark}

\section{$\mathcal{Z}$-stability of tiling algebras}\label{sec:tilingZstable}

This section contains our $\ZZ$-stability result (Theorem~\ref{tilingZstable}). Since an application of Theorem~\ref{groupoidOW} will be crucial to our proof, we first present subsets of the tiling groupoid which are compatible with this theorem. 

We caution the reader that for the purposes of this section, we will need to work with tilings of $\R^d$ alongside quasitilings of the tiling groupoid $\Rpunc$. In order to make this distinction clearer, we will make an attempt to be consistent with the notation for tilings as elements of the unit space of the groupoid $\Rpunc$. Generally, tilings of $\R^d$ will be denoted by the letters $u$, $x$ or $y$; vectors in $\R^d$ will be denoted by the letter $z$, and quasitilings of $\Rpunc$ will be made up of tiles $T_1,\ldots,T_n\subset\Rpunc$ and centres $C_1,\ldots,C_n\subset\Rpunc$ (with additional subscripts as necessary).
\begin{lemma}\label{tilingTset}
The groupoid $\Rpunc$ associated to an aperiodic and repetitive tiling of $\R^d$ with FLC admits sequences of subsets $T_l$ as in Theorem~\ref{groupoidOW} of arbitrary length. 
\end{lemma}
\begin{proof}
We directly construct a sequence of subsets $\{T_l\}_{l\in\N}$ as in Theorem~\ref{groupoidOW}. Let $\Opunc$ be the punctured hull of a tiling as in the statement. By FLC, there is a prototile of minimal volume $V_\textnormal{min}$, and one of maximal volume $V_\textnormal{max}$. Find $m\in\N$ such that $m>\max(V_\textnormal{max}/V_\textnormal{min}, 2)$. For $n\in\N$, define $T_n\subset\Rpunc$ as follows. For each $u\in\Opunc$ consider the set of punctures in $u$ which lie in $B_n(0)$. For each such puncture $x(t)$ (where $t$ is a tile in the tiling $u$ of $\R^d$), include the arrow $(u-x(t), u)$ in $T_n$. In other words, $T_n$ is the collection of allowable translates by vectors of magnitude smaller than $n$
\[T_n=\{(u+z,u)\in\Rpunc\mid u\in\Opunc, \|z\|<n\}.\]
Notice that $T_n$ is closed in the metric topology on $\Rpunc$. Indeed, if $\{(u_j+z_j,u_j)\}_{j\in\N}$ converges to $(u+z,u)$ then both $d(u,u_j)\rightarrow0$ and $\|z-z_j\|\rightarrow0$. Eventually, this convergence forces $u_j\sqcap B_n(0)=u\sqcap B_n(0)$. Furthermore, since $\|z_j\|<n$ for each $j$, using the uniform discreteness of the puncture set together with the last sentence and the convergence of $z_j$ to $z$ implies that eventually $z_j=z$ for all $j$. Thus we see that $\|z\|<n$ and so $(u+z,u)\in T_n$, showing that $T_n$ is closed. In addition, we have that
\[T_n\subset\bigcup_{u\in\Opunc}\bigcup_{t\in u\sqcap B_n(0)}V(u\sqcap B_n(0),t,u(0)).\]
By FLC, there are only finitely many patches $u\sqcap B_n(0)$ involved in the union on the right, each of which contain finitely many tiles. Thus the right-hand side is a finite union of compact sets, and so is itself compact. Thus we see that $T_n$ is a closed subset of a compact set in a metric space, so it is compact. 
Note that $T_n$ is closed under taking inverses, and therefore $|uT_n|=|T_nu|$ for each $u\in\Opunc$. It follows that
\[\frac{\inf_{u\in\Opunc}|T_nu|}{\sup_{u\in\Opunc}|uT_n|}\leq1.\]
We look to bound $\inf_{u\in\Opunc}|T_nu|$ and $\sup_{u\in\Opunc}|uT_n|=\sup_{u\in\Opunc}|T_nu|$. We will use volume estimates to do this. 

First, for $u\in\Opunc$ to maximise $|T_nu|$ we wish there to be as many tiles of $u$ intersecting $B_n(0)$ as possible, so that we can assume the punctures of all these tiles lie in this ball. By FLC there exists a prototile of largest diameter $D_\textnormal{max}$. Consider the ball $B_{n+D_\textnormal{max}}(0)$. If a tile $t$ intersects the complement of this ball then, since its diameter is at most $D_\textnormal{max}$, it cannot intersect $B_n(0)$. Therefore, since $t$ intersects $B_n(0)$ only if it is contained within $B_{n+D_\textnormal{max}}(0)$, we ask how many tiles can fit within the latter ball. Clearly a possible bound is given by
\[\sup_{u\in\Opunc}|uT_n|\leq\frac{\textnormal{Vol}(B_{n+D_\textnormal{max}}(0))}{V_\textnormal{min}}=\frac{V_d(n+D_\textnormal{max})^d}{V_\textnormal{min}}\]
where $V_d$ denotes the volume of the unit ball in $\R^d$. 

We argue similarly to establish our second bound. From this point onwards we will assume that $n>D_\textnormal{max}$. If a tile intersects $B_{n-D_\textnormal{max}}(0)$, then its puncture must lie in $B_n(0)$, so we seek to minimize the number of tiles intersecting this ball. To do so, we may assume any such tile has volume $V_\textnormal{max}$ and is completely contained within the ball. Then we see that
\[\inf_{u\in\Opunc}|T_nu|\geq\frac{V_d(n-D_\textnormal{max})^d}{V_\textnormal{max}}.\]

Combining our bounds yields
\[\frac{\inf_{u\in\Opunc}|T_nu|}{\sup_{u\in\Opunc}|uT_n|}\geq\frac{V_dV_\textnormal{min}(n-D_\textnormal{max})^d}{V_dV_\textnormal{max}(n+D_\textnormal{max})^d}\underset{n\rightarrow\infty}{\rightarrow}\frac{V_\textnormal{min}}{V_\textnormal{max}}>\frac{1}{m}.\]
Therefore, for large enough $n$ and small enough $0<\epsilon<1/2$ we have that 
\[\frac{\inf_{u\in\Opunc}|T_nu|}{\sup_{u\in\Opunc}|uT_n|}(1-\epsilon)\geq\frac{1}{m}\]
too. 

Now, fix $k\in\N$. Since $T_n$ is the set of translates in $\Rpunc$ with magnitude smaller than $n$, we see that any translate in $\partial_{T_n}(T_{n+k})$ must have magnitude larger than $(n+k)-n=k$ and smaller than $(n+k)+n=2n+k$. Thus, to bound $|\partial_{T_n}(T_{n+k}u)|$ above, we wish to maximise the number of punctures which appear in the annulus with center $0$, inner radius $k$, and outer radius $2n+k$. A tile which intersects the complement of the annulus with center $0$, inner radius $k-D_\textnormal{max}$, and outer radius $2n+k+D_\textnormal{max}$ cannot have its puncture in the region of interest, so we estimate the number of tiles contained in this latter annulus as follows
\[|\partial_{T_n}(T_{n+k}u)|\leq\frac{V_d((2n+k+D_\textnormal{max})^d-(k-D_\textnormal{max})^d)}{V_\textnormal{min}}\]
where the top of the fraction is just the volume of the second annulus. On the other hand, we know that
\[|T_{n+k}u|\geq\frac{V_d(n+k-D_\textnormal{max})^d}{V_\textnormal{max}}.\]
Let
\[C_{n,k}\coloneqq\frac{V_\textnormal{max}((2n+k+D_\textnormal{max})^d-(k-D_\textnormal{max})^d)}{V_\textnormal{min}(n+k-D_\textnormal{max})^d}\]
so that
\begin{align*}
|\partial_{T_n}(T_{n+k}u)|&\leq\frac{V_d((2n+k+D_\textnormal{max})^d-(k-D_\textnormal{max})^d)}{V_\textnormal{min}}\\
&=C_{n,k}\frac{V_d(n+k-D_\textnormal{max})^d}{V_\textnormal{max}}\\
&\leq C_{n,k}|T_{n+k}u|.
\end{align*}
Notice that for each fixed $n$, we have that $C_{n,k}\rightarrow0$ as $k\rightarrow\infty$. Therefore, for each $n$, we can find $k(n)\geq1$ such that for any $k\geq k(n)$ we have $C_{n,k}\leq\epsilon^2/8$. Then the inequality above becomes
\[|\partial_{T_n}(T_{n+k(n)}u)|\leq C_{n,k(n)}|T_{n+k(n)}u|\leq\epsilon^2/8|T_{n+k(n)}u|.\]

We now relabel the sequence $(T_n)$ by removing the first terms and shifting index, to assume that every set in the sequence satisfies properties (i) and (ii) from the statement of Theorem~\ref{groupoidOW}. We construct a subsequence by choosing indices $n_i$ recursively as follows. Let $n_1=1$ and put $n_{i+1}=n_i+k(n_i)$. By construction, the sequence $(T_{n_i})_{i\in\N}$ will satisfy all of the requirements of Theorem~\ref{groupoidOW}. 
\end{proof}

To obtain $\ZZ$-stability, we use the following criterion of Hirshberg-Orovitz, which was obtained using the breakthrough result of Matui and Sato, where they show that any nuclear tracially AF-algebra is $\ZZ$-stable, see \cite[Theorem~1.1 and Theorem~5.4]{MS2}. Hirshberg's and Orovitz's criterion requires us to embed arbitrarily large matrix algebras into the crossed product so that the image approximately commutes with a given finite set and is large in trace. That this criterion is equivalent to $\ZZ$-stability for simple separable unital nuclear $C^*$-algebras is recorded in \cite[Proposition~2.2 and Theorem~4.1]{HO}.
\begin{thm}[{\cite[Definition~2.1]{HO}}]\label{Zstable}
Let $A$ be a simple separable unital nuclear $C^*$-algebra which is not isomorphic to $\C$. Write $\precsim$ for the relation of Cuntz subequivalence over $A$, and given $a,b\in A$ write $\lbrack a,b\rbrack$ for the commutator $ab-ba$. Suppose that for every $n\in\N$, finite set $F\subset A$, $\epsilon>0$ and nonzero positive element $a\in A$ there exists an order zero completely positive and contractive map $\phi:M_n\rightarrow A$ such that
\begin{enumerate}[label=(\roman*)]
\item $1-\phi(1)\precsim a$,
\item $\|\lbrack w,\phi(B)\rbrack\|<\epsilon$ for all $w\in F$ and $B\in M_n$ with $\|B\|=1$.
\end{enumerate}
Then $A$ is $\ZZ$-stable.
\end{thm}
The following lemma will allow us to witness Cuntz subequivalence when checking condition (i) of Theorem~\ref{Zstable}.
\begin{lemma}[{c.f. \cite[Lemma~12.3]{Kerr}}]\label{David12.3}
Let $G\curvearrowright X$ be an action of a groupoid on a compact metrizable space. Let $A$ be a closed subset of $X$ and $B$ an open subset of $X$ such that $A\prec B$. Let $f,g:X\rightarrow\lbrack0,1\rbrack$ be continuous functions such that $f=0$ on $X\setminus A$, and $g=1$ on $B$. Then there exists a $v\in C(X)\rtimes_r G$ such that $v^*gv=f$.
\end{lemma}
\begin{proof}
Since $A\prec B$, and $A$ is a closed subset of $A$, find a finite collection $U_1,\ldots,U_M$ of open subsets of $X$ which cover $A$, and subsets $S_1,\ldots,S_M$ of $G$ as in the definition. Without loss of generality, by removing any element $t\in S_m$ such that $s(t)\notin r_X(U_m)$, we may assume that $r_X(U_m)=s(S_m)$. Furthermore, by preserving one element from each source bundle and removing all other elements from $S_m$ if necessary, we may assume that $s: S_m\rightarrow r_X(U_m)$ is injective for each $m\in\{1,\ldots,M\}$. 

Using a partition of unity argument, produce, for each $m\in\{1,\ldots,M\}$, a continuous function $h_m:X\rightarrow\lbrack0,1\rbrack$  such that $h_m(x)=0$ for any $x\in X\setminus U_m$, and which satisfy $0\leq\sum_{m=1}^Mh_m(x)\leq 1$ for any $x\in X$, and $\sum_{m=1}^Mh_m(a)=1$ for any $a\in A$. Put $v=\sum_{m=1}^M\sum_{t\in S_m} U_t(fh_m)^{1/2}$. 

Denote by $\alpha$ the induced action of $G$ on $C(X)$, given by 
\[\alpha_g(f)(x)=\begin{cases}f(g^{-1}x)&\hbox{if } x\in r_X^{-1}(r(g))\\0&\hbox{otherwise}.\end{cases}\]
We compute that
\begin{align*}
v^*gv&=\left(\sum_{m=1}^M\sum_{t\in S_m}(fh_m)^{1/2}U_{t^{-1}}\right)g\left(\sum_{\tilde m=1}^M\sum_{\tilde t\in S_{\tilde m}}U_{\tilde t}(fh_{\tilde m})^{1/2}\right)\\
&=\sum_{m,\tilde m=1}^M\sum_{t\in S_m}\sum_{\tilde t\in S_{\tilde m}}U_{t^{-1}}g\alpha_t(f)^{1/2}\alpha_t(h_m)^{1/2}\alpha_{\tilde t}(h_{\tilde m})^{1/2}\alpha_{\tilde t}(f)^{1/2}U_{\tilde t}.\\
\end{align*}
Note that $\alpha_t(h_m)$ is supported on $tU_m$, and $\alpha_{\tilde t}(h_{\tilde m})$ is supported on $\tilde tU_{\tilde m}$, which are disjoint unless $m=\tilde m$ and $t=\tilde t$. Thus, the only nonzero contributions to the above sum require $t=\tilde t$ and $m=\tilde m$, so we obtain
\begin{align*}
v^*gv&=\sum_{m=1}^M\sum_{t\in S_m}U_{t^{-1}}g\alpha_t(h_m)\alpha_t(f)U_t\\
&=\sum_{m=1}^M\sum_{t\in S_m}\alpha_{t^{-1}}(g)h_mf.
\end{align*}
The term corresponding to a particular $m\in\{1,\ldots,M\}$ and $t\in S_m$ in the above sum is supported in $U_m\cap r_X^{-1}(s(t))$. For $x\in U_m\cap r_X^{-1}(s(t))$, we have $\alpha_{t^{-1}}(g)(x)=g(tx)$, where $tx\in tU_m\subset B$, so we obtain $\alpha_{t^{-1}}(g)=1$ wherever this term is supported, because $g(b)=1$ for any $b\in B$ by assumption. This gives
\[v^*gv=\sum_{m=1}^M\sum_{t\in S_m}(fh_m)|_{U_m\cap r_X^{-1}(s(t))}\]
Since $s(S_m)=r_X(U_m)$, since $s:S_m\rightarrow r_X(U_m)$ was injective for each $m$, 
and using the facts that $h_m$ is supported inside $U_m$ and $f$ is supported inside $A$, together with our assumption that $\sum_{m=1}^Mh_m(a)=1$ for each $a\in A$, 
we see that
\[v^*gv=\sum_{m=1}^M(fh_m)|_{U_m}=\sum_{m=1}^Mfh_m=\sum_{m=1}^Mf(h_m)|_A=f\sum_{m=1}^M(h_m)|_A=f.\qedhere\]
\end{proof}

We now turn to our $\ZZ$-stability result. Our method of proof is by now well-established in the group case \cite{CJKMSTD, Kerr}. Nevertheless, we wish to provide an outline of the proof. Our outline will focus on the first half of the proof (pages \pageref{Zstabilitystart}-\pageref{endsetup}), which contains a large amount of technical setup. For now, we aim to ignore the technicalities for a sharper focus on the philosophy behind the proof. We also highlight parts of the argument which are specifically necessary in the groupoid case.

We begin with a tiling $C^*$-algebra, $A_\textnormal{punc}=C^*(\Rpunc)$. Observe that $\Rpunc$ is isomorphic to the transformation groupoid $\Rpunc\ltimes\Opunc$ associated to the canonical groupoid action of $\Rpunc$ on its unit space $\Opunc$. Since $\Rpunc\cong\Rpunc\ltimes\Opunc$ is almost finite by Lemma~\ref{tilingalmostfinite}, it follows from Theorem~\ref{Suzuki Analogue} that this action is almost finite. 
We use \cite[Proposition~4.38]{Goehle} to think of $C^*(\Rpunc)\cong C^*(\Rpunc\ltimes\Opunc)\cong C(\Opunc)\rtimes_r\Rpunc$ as a groupoid crossed product, and we seek to embed arbitrarily large matrix algebras into this algebra to witness $\ZZ$-stability using Theorem~\ref{Zstable}. 

We consider compact subsets $\Opunc\subset T_1\subset\cdots\subset T_L\subset\Rpunc$ as in Lemma~\ref{tilingTset}, and we use almost finiteness of the action $\Rpunc\curvearrowright\Opunc$ to obtain a clopen tower decomposition $\{(W_k,S_k)\}$ whose shapes are sufficiently invariant to be quasitiled by the tiles $\{T_1,\ldots,T_L\}$, so that the quasitiling of $S_k$ has centres $\{C_{k,1},\ldots,C_{k,L}\}$. Compared to the group case, this step requires additional care because it will be important that the quasitiling encapsulates the structure of the tower. In the group case, each tower level is obtained by allowing a single group element to act on the base of the tower, but this is not generally true for groupoids due to the fibred nature of their actions. Therefore, we must take some care to ensure that, for each tower level $S_{k,p}W_k$, each of the sets $C_{k,l}$ and $T_lC_{k,l}$ either contains all of $S_{k,p}$, or does not intersect it at all. This consideration is the main point of difficulty in generalising to the groupoid setting, and must be accounted for whenever we modify these subsets.

We introduce a finite open cover $\{U_1,\ldots,U_M\}$ of $\Opunc$ by sets of sufficiently small diameter, and use Lemma~\ref{amplediameter} to assume that our tower levels have diameter smaller than the Lebesgue number of the cover. We partition each of the sets $C_{k,l}$ into $C_{k,l,1},\ldots,C_{k,l,M}\subset C_{k,l}$ by requiring that the tower levels associated to elements of $C_{k,l,m}$ are subordinate to $U_m$. This allows the later use of a uniform continuity argument to obtain approximate centrality of our matrix embedding under a finite subset of $C(\Opunc)$. 
We then choose $n$ equally sized pairwise disjoint subsets $C_{k,l,m}^{(1)},\ldots,C_{k,l,m}^{(n)}\subset C_{k,l,m}$ of cardinality $\lfloor|C_{k,l,m}|/n\rfloor$, and define a system of bijections between these sets. The matrix unit $e_{ij}\in M_n$ will be implemented as movement from the tower levels associated to elements of $T_lC_{k,l,m}^{(j)}$ to those associated to the corresponding elements of $T_lC_{k,l,m}^{(i)}$ under the appropriate bijection. 

We also need to ensure that our matrix embedding is approximately central under the action of a predefined compact subset of the groupoid. To do so, we use $\beta$-disjointness of the quasitiling $\{T_lc\mid c\in C_{k,l}\}$ to slightly shrink the tiles to eliminate as much ``stacking'' of tower levels as possible. We assume that the tiles are sufficiently invariant under multiplication by the compact set that we can find a large ``core'' in each consisting of the elements that remain within the tile after a large number $Q$ of such multiplications. We use this ``core'' to partition the tile, grouping the elements according to how many multiplications it takes to remove them from the tile. We then assign each tower level in our matrix embedding an indicator function which is scaled by a factor $q/Q$ for some $q\in\{1,\ldots,Q\}$ based on which set in the partition the associated groupoid element lies. For example, elements in the ``core'', which remain in the tile after $Q$ multiplications, are assigned the value $Q/Q=1$, while those which are removed from the tile after just one multiplication are assigned the value $1/Q$. 
As a result of this construction, conjugation by an element of the compact set is guaranteed to perturb the indicator function by no more than $1/Q$.

The final step is to ensure that our matrix embedding is large in trace. This is expressed by comparison to a positive element, which we may assume is a function in $C(\Opunc)$ by Lemma~\ref{Phillips2}. By using minimality of the action and increasing the invariance of the shapes of the tower if necessary, we are able to assume that a large number of the tower levels are contained within the support of this function; whence, with no loss of generality, we may assume the function takes the constant value 1 on these tower levels. More precisely, since the matrix embedding was constructed to be supported on a union of tower levels, we may ask that the number of tower levels on which the matrix embedding is \textit{not} supported is smaller than the number of tower levels on which the positive element takes the value 1. By setting up an injection between the first set of tower levels and the second, we are able to put the complement of the matrix embedding below the positive element, witnessing the largeness in trace.
 
\begin{thm}\label{tilingZstable}
The $C^*$-algebra associated to an aperiodic and repetitive tiling with FLC is $\ZZ$-stable.
\end{thm}
\begin{proof}
\label{Zstabilitystart}
To simplify notation, denote $G=\Rpunc$ and $X=\Opunc$. 
As per usual, we let $\alpha$ denote the induced action of $G$ on $C(X)$, so that $\alpha_g:C(X)\rightarrow C(r_X^{-1}(r(g)))$ is given by $\alpha_g(f)(x)=f(g^{-1}x)$ for $x\in r_X^{-1}(r(g))$, and we define $\alpha_g(f)(x)=0$ if $r_X(x)\neq r(g)$. 

Let $n\in\N$, $\epsilon>0$, and let $a\in C(X)\rtimes_r G$ be a nonzero positive element. Let $F\subset C(X)\rtimes_r G$ be a finite set. We will show that there exists a map $\phi:M_n\rightarrow C(X)\rtimes_r G$ as in Theorem~\ref{Zstable}, from which we will be able to conclude that $C(X)\rtimes_r G$ is $\ZZ$-stable.  

Notice that, since $\Gamma_c(G,r^*\AA)$ is dense in $A\rtimes_r G$, it will be enough to consider finite subsets $F\subset\Gamma_c(G,r^*\AA)$. Each such subset is generated by a finite subset $\Upsilon$ of the unit ball of $C(X)$
 and a compact subset $K\subset G$ which, without loss of generality, we may assume contains $G^{(0)}$ and is closed under taking inverses. To see that $\Upsilon$ can be chosen to be finite, fix any $f=\sum_{g\in C}f_gU_g\in F$ and observe that, since $C$ is compact and $G$ is ample, $C$ admits a finite cover of pairwise disjoint 
compact open bisections $\{B_1,\ldots,B_N\}$. Then $f$ is generated by $\{U_g\mid g\in C\}$ together with the finite collection of functions $\{\sum_{g\in B_i}f_g\}_{i=1,\ldots,N}$, each of which is in $C(X)$ by the argument following the proof of Lemma~\ref{gpoidconditionalexp}. Since $C(X)\rtimes_r\Rpunc$ is generated by sets of this form, it will be enough to check condition (ii) in Theorem~\ref{Zstable} for each element $w\in\Upsilon$ and for $w=\sum_{g\in K}U_g$. 

Since $K$ is a compact subset of $\Rpunc$, there exists a real number $R_K>0$ such that if $(u-z,u)\in K$ (for some $u\in\Opunc$ and $z\in\R^d$) then $|z|<R_K$. By enlarging $K$ we may further assume, without loss of generality, that $R_K^{-1}<r/2$, where $r$ is the constant of uniform discreteness for the puncture set in the tiling. This ensures that whenever $u,u'\in\Opunc$ have $d(u,u')<R_K^{-1}$ in the tiling metric, the subsets $Ku$ and $Ku'\subset\Rpunc$ consist of translations by the same set of vectors in $\R^d$. We enlarge $K$ one last time to assume that $K=\{(u-z,u)\in\Rpunc\mid |z|<R_K\}$. Reasoning as for the sets $T_n$ in Example~\ref{tilingTset}, this choice of $K$ is seen to be compact.

By Lemma~\ref{Phillips2}, we may assume that $a\in C(X)$. Then, since $a$ is a nonzero positive element, there exists $x_0\in X$ and $\theta>0$ such that $a$ is strictly positive on the closed ball of radius $2\theta$ centred at $x_0$. Thus, we may assume that $a$ is a $\lbrack0,1\rbrack$-valued function which takes the value $1$ on all points within distance $2\theta$ from $x_0$, and the value $0$ at all points at  distance at least $3\theta$ from $x_0$. Denote by $O$ the open ball of radius $\theta$ centred at $x_0$. 
Because the action is minimal, there exists a subset $D\subset G$ such that $D^{-1}\cdot O=X$ and so that there exists $m\in\N$ such that, for each $u\in G^{(0)}$, $|D\cap r^{-1}(u)|<m$ and $|D\cap s^{-1}(u)|<m$. To see this, since $G$ is \'etale it has a base of open bisections $\{B_i\}_{i\in I}$.  For each $i\in I$, the set $B_iO$ is open in $X$ by Lemma~\ref{actiontopology}, because $B_i$ is open in $G$ and $O$ is open in $X$. Because $B_i$ is a bisection, for each point $x\in O$, there exists at most one $g\in B_i$ such that $s(g)=r_X^{-1}(x)$, and therefore at most one image $gx\in B_iO$. Since the action is minimal, the collection $\{B_iO\}_{i\in I}$ is an open cover of $X$ and, since $X$ is compact, there exists a finite subcover $\{B_1O,\ldots,B_mO\}$. Let $D=\bigcup_{j=1}^m B_j^{-1}$, which has the properties we seek.

Let $0<\kappa<1$, to be determined on page~\pageref{kappadetermined}. Choose an integer 
\[
Q>(n^2\sup_{u\in G^{(0)}}|uK|)/\epsilon.
\]
Denote by $K^Q$ the product of $K$ with itself $Q$ times in $G$:
\[K^Q\coloneqq\{k_Q\cdots k_1\mid (k_{i+1},k_i)\in K^{(2)}\ \textnormal{ for every }  i=1,\ldots,Q-1\}.\]
Take $\beta>0$ to be small enough so that if $T$ is a nonempty compact subset of $G$ which is sufficiently invariant under left-translation by $K^Q$, then, for every $T'\subset T$ with $|T'u|\geq(1-n\beta)|Tu|$ at every $u\in G^{(0)}$, we have, for every $u\in G^{(0)}$, that
\begin{equation}\label{largesetinvariance}
\left|\bigcap_{g\in K^Q}\left(gT'\sqcup\bigsqcup_{v\in G^{(0)}\setminus r(g)}vT'\right)u\right|\geq(1-\kappa)|Tu|.
\end{equation}

To ease thinking, note that since $G^{(0)}\subset K$ and $K$ is closed under taking inverses, the left-hand set is the subset consisting of $t\in T'u$ such that $K^Qt\subset T'u$. Choose $L\in\N$ large enough so that $(1-\beta/2)^L<\beta$. By Example~\ref{tilingTset}, 
 there exist compact subsets $G^{(0)}\subset T_1\subset\cdots\subset T_L$ of $G$ such that $|\partial_{T_{l-1}}(T_lu)|\leq(\beta^2/8)|T_lu|$ for each $l=2,\ldots,L$ and $u\in G^{(0)}$. 
We may also enlarge each $T_l$ in turn to assume that it is sufficiently invariant under left translation by $K^Q$ so that we can use the previous paragraph to ensure that 
for every subset $T'\subset T_l$ satisfying $|T'u|\geq(1-n\beta)|T_lu|$ for each $u\in G^{(0)}$, we have, for each $u\in G^{(0)}$, that \eqref{largesetinvariance} holds for this choice of $T'$ and with $T$ replaced with $T_l$. 

Since $T_L$ was constructed as in Example~\ref{tilingTset}, there exists $R_L>0$ such that $T_L=\{(u-z,u)\in\Rpunc\mid |z|<R_L\}$. 
By the uniform continuity of the functions in $\Upsilon\cup\Upsilon^2$, and of the action map $T_L\times X\rightarrow X$, there exists $0<\eta<(QR_K+R_L)^{-1}$ so that if $x,y\in X$, and $t,t'\in T_L$ are such that $s(t)=x$ and $s(t')=y$, and satisfy $d(x,y)<\eta$ and $d(t,t')<\eta$, then $|f(tx)-f(t'y)|<\epsilon^2/(4n^4)$. Let $\UU=\{U_1,\ldots,U_M\}$ be an open cover of $X$ whose members all have diameter less than $\eta$. Let $\eta>\eta'>0$ be a Lebesgue number for $\UU$ which is no larger than $\theta$. 

Let $E$ be a compact subset of $G$ containing $T_L$, and let $\delta>0$ be such that $\delta<\beta^2/4$. Since $G_u$ is infinite for each $u\in G^{(0)}$, we may enlarge $E$ and shrink $\delta$ as necessary to ensure that, for each nonempty $(E,\delta)$-invariant compact set $S\subset G$ and each $u\in G^{(0)}$, we have that
\begin{equation}\label{Mnmax}
\beta|Su|\geq Mn\sum_{l=1}^L\max_{v\in G^{(0)}}|T_lv|.
\end{equation}
By combining Lemma~\ref{tilingalmostfinite} and Theorem~\ref{Suzuki Analogue}, we see that the action $G\curvearrowright X$ admits clopen tower decompositions of arbitrary invariance. Therefore, for some $N\in\N$, we can find nonempty clopen sets $W_1,\ldots,W_N\subset X$ and nonempty $(E,\delta)$-invariant compact open sets $S_1,\ldots,S_N\subset G$ such that the family $\{(W_k,S_k)\}_{k=1}^N$ is a clopen tower decomposition of $X$
with levels of diameter less than $\eta'$. Write $S_k=\bigsqcup_{j=1}^{n_k}S_{k,j}$ for the decomposition of each shape. By Remark~\ref{tilingtowersbasissets}, we may assume, without loss of generality, that, for each $k\in\{1,\ldots,N\}$, there exists a patch $P_k$ and a tile $t_k\in P_k$ such that $W_k=U(P_k,t_k)$, and so that, for each $j\in\{1,\ldots,n_k\}$, there exists $t_{k,j}\in P_k$ such that $S_{k,j}=V(P_k,t_k,t_{k,j})$. Note in particular that, for each $l\in\{1,\ldots,L\}$, since $T_l=\{(u-z,u)\in\Rpunc\mid |z|<R_l\}$ for some $0<R_l\leq R_L$, our choice of $\eta'$ ensures that if $x$ and $y$ are in the same tower level, then $T_lx$ and $T_ly$ consist of arrows which implement exactly the same set of vectors of translation.

Let $k\in\{1,\ldots,N\}$. Since $S_k$ is $(T_L,\beta^2/4)$-invariant, and since the sets $T_1,\ldots,T_L$ satisfy the assumptions of Theorem~\ref{groupoidOW}, we can find $C_{k,1},\ldots,C_{k,L}\subset S_k$ such that $\bigcup_{l=1}^LT_lC_{k,l}\subset S_k$, and so that the collection $\{T_lc\mid l\in\{1,\ldots,L\}, c\in C_{k,l}\}$ is $\beta$-disjoint and $(1-\beta)$-covers $S_k$. 
We claim that we may assume that, for each $l\in\{1,\ldots,L\}$ and each $p\in\{1,\ldots,n_k\}$, we have either $S_{k,p}\subset T_lC_{k,l}$ or $S_{k,p}\cap T_lC_{k,l}=\emptyset$. 
That is, $T_lC_{k,l}$ either contains all the groupoid elements corresponding to a given level of the tower, or none of them. We prove this assertion over a number of steps.

To elaborate, recall that $C_{k,L}$ was constructed from the set $I=\{c\in S_k\mid T_Lc\subset S_k\}$. Suppose $c\in S_{k,j}=V(P_k,t_k,t_{k,j})$ is an element of $I$. Then, for each fixed $t\in T_Lr(c)$, $tc\in S_k$, and so, in particular, $tc\in S_{k,i}=V(P_k,t_k,t_{k,i})$ for some $i$. Recall that $T_L$ was the collection of allowable vectors of translation of magnitude smaller than $R_L$. Thus, if $tc\in V(P_k,t_k,t_{k,i})\subset S_k$ for some $c\in S_{k,j}$, then the vector $x(t_{k,i})-x(t_{k,j})$ was associated to $t$, and was allowable in $r(c)\in U(P_k,t_{k,j})$. 
Let $c'$ be any other element of $S_{k,j}$. By the diameter condition on the tower levels, $T_Lr(c')$ consists of arrows which implement the same set of vectors as the arrows in $T_Lr(c)$. By the argument at the start of the paragraph, we see that any such arrow $t'\in T_Lr(c')$ implements the vector $x(t_{k,i})-x(t_{k,j})$ for some $i$. Since $c'\in S_{k,j}=V(P_k,t_k,t_{k,j})$, we see that $t'c'\in V(P_k,t_k,t_{k,i})=S_{k,i}\subset S_k$. This shows that $c'\in I$, so that $S_{k,j}\subset I$. Hence, we have shown that, for each $j$, either $S_{k,j}\cap I=\emptyset$ or $S_{k,j}\subset I$.

To obtain $C_{k,L}$, we applied Lemma~\ref{evendisjoint} to get a subset $C_{k,L}\subset I$ such that the collection $\{T_Lc\mid c\in C_{k,L}\}$ was $\beta$-disjoint. In doing so, we claim that we can choose to either remove or keep each tower level as a whole. 
Indeed, if we fix any $x\in W_k$ and apply Lemma~\ref{evendisjoint} to the collection $\{T_Lcx\mid c\in I\}$, then we obtain $C_{k,L}x\subset Ix$ such that $\{T_Lcx\mid c\in C_{k,L}x\}$ is $\beta$-disjoint. Now, choose any other $y\in W_k=U(P_k,t_k)$. By the previous paragraph, $I$ contains all or none of the elements of each $S_{k,p}$, so, by the structure of the tower $(W_k,S_k)$, $Ix$ and $Iy$ consist of arrows which implement the same vectors of translation. This allows us to choose $C_{k,L}y$ to consist of arrows which implement the same vectors as those in $C_{k,L}x$. Since $I\subset S_k$, if $c,c'\in I$ implement the same vector of translation and have $s(c)=x$ and $s(c')=y$, then $r(c)$ and $r(c')$ are in the same tower level, and thus (as tilings) they agree on $B_{\eta^{-1}}(0)$. Therefore, $T_Lr(c)$ and $T_Lr(c')$ consist of elements which implement the same set of vectors. This shows that $\{T_Lcx\mid c\in C_{k,L}x\}$ and $\{T_Lc'y\mid c'\in C_{k,L}y\}$ implement the same set of vectors. Combining this with the fact that the collection $\{T_Lcx\mid c\in C_{k,L}x\}$ was $\beta$-disjoint, we see that the collection $\{T_Lc'y\mid c'\in C_{k,L}y\}$ is also $\beta$-disjoint. Repeat this construction for each $y\in W_k$, and set $C_{k,L}=\bigsqcup_{y\in W_k}C_{k,L}y$, so that the collection $\{T_Lc\mid c\in C_{k,L}\}$ is $\beta$-disjoint. Observe that by construction, for each $j$, we have either $S_{k,j}\cap C_{k,l}=\emptyset$ or $S_{k,j}\subset C_{k,l}$. 

By construction, since $C_{k,L}\subset I$, we have $T_LC_{k,L}\subset S_k$. Now, suppose that $g\in S_{k,p}=V(P_k,t_k,t_{k,p})$ is such that $g\in T_LC_{k,L}$. Then there exists $c\in C_{k,L}$ such that $c\in S_{k,j}=V(P_k,t_k,t_{k,j})$ for some $j$, and $t\in T_L$ implementing the vector $x(t_{k,p})-x(t_{k,j})$, such that $tc=g$. Choose any other $g'\in S_{k,p}$, and consider the groupoid element $c'$ with source $s(g')$ and range $s(g')+x(t_{k,j})-x(t_k)$. Observe that $c'$ is in $S_{k,j}=V(P_k, t_k, t_{k,j})$, which is a subset of $C_{k,L}$ by the previous paragraph, since $c\in S_{k,j}\cap C_{k,L}$. Since $r(c')$ is in the same tower level as $r(c)=s(t)$, the arrows in $T_Lr(c')$ implement exactly the same vectors as the arrows in $T_Ls(t)$, so, in particular, there is an element $t'\in T_Lr(c')$ which implements the same vector $x(t_{k,p})-x(t_{k,j})$ that $t$ does. Then $t'c'$ implements the vector $x(t_{k,p})-x(t_k)$, to $s(t'c')=s(g')$, so we see that $t'c'=g'$, and hence that $g'\in T_LC_{k,L}$. This shows that, for each $p$, either $S_{k,p}\cap T_LC_{k,L}=\emptyset$, or $S_{k,p}\subset T_LC_{k,L}$. 

We now simply repeat the procedure of the previous three paragraphs for each $l\in\{L-1,\ldots,1\}$ in turn. As in the proof of Lemma~\ref{groupoidOW}, at each step we consider $I=\{c\in S_k\mid T_lc\subset S_k\setminus\bigsqcup_{j=l+1}^LT_jC_{k,j}\}$. Notice that when forming $S_k\setminus\bigsqcup_{j=l+1}^LT_jC_{k,j}$, we have removed entire tower levels from consideration, so at each step we are working with a subtower of $(W_k,S_k)$. 

Since the levels of the tower $(W_k,S_k)$ have diameter less than $\eta'$, which is a Lebesgue number for $\UU$, for each $l\in\{1,\ldots,L\}$ there is a partition
\[C_{k,l}=C_{k,l,1}\sqcup C_{k,l,2}\sqcup\cdots\sqcup C_{k,l,M}\]
such that $cW_k\subset U_m$ whenever $m\in\{1,\ldots,M\}$ and $c\in C_{k,l,m}$. 
In fact, by the diameter condition on the tower levels, for each $j$ such that $S_{k,j}\subset C_{k,l}$, there exists $m$ such that $S_{k,j}W_k\subset U_m$, whence, for each $c\in S_{k,j}$, we may put $c\in C_{k,l,m}$. In this way, we may assume that, for each $j$ and $m$, either $S_{k,j}\cap C_{k,l,m}=\emptyset$ or $S_{k,j}\subset C_{k,l,m}$. In addition, since $T_lC_{k,l,m}\subset S_k$, and making use of the structure of $T_l$ and $S_{k,j}$, it is then automatic that, for each $j$ and $m$, either $S_{k,j}\cap T_lC_{k,l,m}=\emptyset$ or $S_{k,j}\subset T_lC_{k,l,m}$ by following a similar argument as that for $T_LC_{k,L}$ above.

For each $l$ and $m$, choose pairwise disjoint subsets $C_{k,l,m}^{(1)},\ldots,C_{k,l,m}^{(n)}$ of $C_{k,l,m}$ such that, for each $u\in G^{(0)}$ and $i\in\{1,\ldots,n\}$, we have $|C_{k,l,m}^{(i)}u|=\left\lfloor|C_{k,l,m}u|/n\right\rfloor$, where $\lfloor\cdot\rfloor$ denotes the floor function.
Further enforce that, if $S_{k,j}\subset C_{k,l,m}$ and $c\in S_{k,j}$ is included in $C_{k,l,m}^{(i)}$, then $S_{k,j}\subset C_{k,l,m}^{(i)}$. Since the source map is injective on $S_{k,j}$, this amounts to adding a single element to $C_{k,l,m}^{(i)}u$ for each $u\in s(S_{k,j})=W_k$. In this way, each $C_{k,l,m}^{(i)}$ corresponds to a choice of a $\left(1/n\right)$-th of the tower levels associated to $C_{k,l,m}$ (ignoring any remainder), so that, for each $j$, either $S_{k,j}\cap C_{k,l,m}^{(i)}=\emptyset$, or $S_{k,j}\subset C_{k,l,m}^{(i)}$. As in the last paragraph, it is then automatic that either $S_{k,j}\cap T_lC_{k,l,m}^{(i)}=\emptyset$, or $S_{k,j}\subset T_lC_{k,l,m}^{(i)}$ for each $j$. 

For each $i\in\{2,\ldots,n\}$, choose a bijection
\[\Lambda_{k,i}:\bigsqcup_{l,m}C_{k,l,m}^{(1)}\rightarrow\bigsqcup_{l,m}C_{k,l,m}^{(i)}\]
which sends 
$C_{k,l,m}^{(1)}u$ to $C_{k,l,m}^{(i)}u$
for every $l\in\{1,\ldots,L\}$, $m\in\{1,\ldots,M\}$, and every $u\in G^{(0)}$. Further enforce that, whenever $S_{k,j}\subset C_{k,l,m}^{(1)}$, we have, for each $i$, that $\Lambda_{k,i}(S_{k,j})=S_{k,p}$ for some $p$ such that $S_{k,p}\subset C_{k,l,m}^{(i)}$, so that $\Lambda_{k,i}$ sends the set of groupoid elements associated to one particular tower level to the set of groupoid elements associated to some other tower level. 
Also, define $\Lambda_{k,1}$ to be the identity map on $\bigsqcup_{l,m}C_{k,l,m}^{(1)}$, and denote
\[\Lambda_{k,i,j}\coloneqq\Lambda_{k,i}\circ\Lambda_{k,j}^{-1}: \bigsqcup_{l,m}C_{k,l,m}^{(j)}\rightarrow\bigsqcup_{l,m}C_{k,l,m}^{(i)}.\] 

Since the collection $\{T_lc\mid l\in\{1,\ldots,L\},c\in C_{k,l}\}$ is $\beta$-disjoint, for every $l\in\{1,\ldots,L\}$ and $c\in C_{k,l}$, we can find a $T_{k,l,c}\subset T_l$ satisfying $|T_{k,l,c}u|\geq(1-\beta)|T_lu|$ for each $u\in G^{(0)}$, which is such that the collection $\{T_{k,l,c}c\mid l\in\{1,\ldots,L\}, c\in C_{k,l}\}$ is pairwise disjoint. 
Recall that, for each $c\in C_{k,l}$, we have $T_lc\subset S_k\setminus\bigsqcup_{j=l+1}^LT_jC_{k,j}$. Therefore, given $c\in C_{k,l}$ and $c'\in C_{k',l'}$, observe that $T_lc$ and $T_{l'}c'$ can only intersect when $k=k'$ (otherwise the levels of the $k$-th and $k'$-th tower would intersect), and when $l=l'$. Therefore, by following a similar argument as in the construction of $C_{k,L}$, we may perform this construction in such a way that, whenever $S_{k,j}\subset C_{k,l}$ and $c,c'\in S_{k,j}$, we choose $T_{k,l,c}r(c)$ and $T_{k,l,c'}r(c')$ to consist of arrows which implement the same set of vectors.
This shows that, for each $j$, either $S_{k,j}\cap\bigsqcup_{c\in C_{k,l}}T_{k,l,c}c=\emptyset$ or $S_{k,j}\subset\bigsqcup_{c\in C_{k,l}}T_{k,l,c}c$. 
We are also free to choose $T_{k,l,c}u=T_lu$ for every $u\in G^{(0)}\setminus\{r(c)\}$. 

In addition, by construction, for every $c\in C_{k,l,m}^{(j)}$ and every $i\in\{1,\ldots,n\}$, we have $r(c)\in U_m$, and $r(\Lambda_{k,i,j}(c))\in U_m$.
In particular, this means that $d(r(c),r(\Lambda_{k,i,j}(c)))<\eta$ in the tiling metric, so
the tilings $r(c)$ and $r(\Lambda_{k,i,j}(c))$ agree on $B_{\eta^{-1}}(0)$. Therefore, for each $i$, $T_lr(c)$ and $T_lr(\Lambda_{k,i,j}(c))$ consist of groupoid elements which implement exactly the same vectors to $r(c)$ and $r(\Lambda_{k,i,j}(c))$, respectively. 
Consider the set $\widetilde{T_{k,l,c}}$, which we obtain from $T_{k,l,c}$ by removing the arrows which correspond to the same vectors as the arrows which are removed from any set of the form $T_lr(\Lambda_{k,i,j}(c))$ to construct $T_{k,l,\Lambda_{k,i,j}(c)}$. In other words, if an arrow corresponding to translating $r(\Lambda_{k,i,j}(c))$ by some vector is removed from $T_lr(\Lambda_{k,i,j}(c))$ when constructing $T_{k,l,\Lambda_{k,i,j}(c)}$ for any $i$, then the arrow corresponding to translating $r(c)$ by the same vector does not appear in $\widetilde{T_{k,l,c}}$. Since we had $|T_{k,l,\Lambda_{k,i,j}(c)}r(\Lambda_{k,i,j}(c))|\geq(1-\beta)|T_lr(\Lambda_{k,i,j}(c))|$ for each $i\in\{1,\ldots,n\}$, we see that $|\widetilde{T_{k,l,c}}r(c)|\geq(1-n\beta)|T_lr(c)|$. In addition, the sets of vectors of translation associated to the elements of any two of the sets $\widetilde{T_{k,l,\Lambda_{k,i,j}(c)}}r(\Lambda_{k,i,j}(c))$, for $i\in\{1,\ldots,n\}$, are the same. Also, when $u\neq r(c)$, we have $\widetilde{T_{k,l,c}}u=T_lu$. Observe that, whenever $c,c'\in S_{k,j}$, the sets $\widetilde{T_{k,l,c}}r(c)$ and $\widetilde{T_{k,l,c'}}r(c')$ will consist of arrows which implement the same vectors of translation, so that, for each $j$, either $S_{k,j}\cap \bigsqcup_{c\in C_{k,l}}\widetilde{T_{k,l,c}}c=\emptyset$ or $S_{k,j}\subset\bigsqcup_{c\in C_{k,l}}\widetilde{T_{k,l,c}}c$.

Now, for each $j\in\{1,\ldots,n\}$ and each $c\in C_{k,l,m}^{(j)}$, consider the set $T_{k,l,c}'\coloneqq\bigcap_{i=1}^n \widetilde{T_{k,l,\Lambda_{k,i,j}(c)}}$.
Observe that, for each $u\in G^{(0)}$,
\[|T_{k,l,c}'u|\geq(1-n\beta)|T_lu|.\]
In addition, we still have, for each $j$, that whenever $c,c'\in S_{k,j}$, the sets $T_{k,l,c}'r(c)$ and $T_{k,l,c'}r(c')$ implement the same vectors of translation, so that, for each $j$, either $S_{k,j}\cap\bigsqcup_{c\in C_{k,j}}T_{k,l,c}'c=\emptyset$ or $S_{k,j}\subset\bigsqcup_{c\in C_{k,j}}T_{k,l,c}'c$.  
Since $T_{k,l,c}'r(c)$ and $T_{k,l,c}'r(\Lambda_{k,i,j}(c))$ consist of arrows which implement the same vectors for each $i$, given $t\in T_{k,l,c}'r(c)$, we will denote by $t^{(i)}$ the element of $T_{k,l,c}'r(\Lambda_{k,i,j}(c))$ which implements the same vector as $t$.

Set
\[B_{k,l,c,Q}u=\bigcap_{g\in K^Q}\left(gT'_{k,l,c}\sqcup\bigsqcup_{v\in G^{(0)}\setminus r(g)}vT'_{k,l,c}\right),\]
noting that, by our assumption on the invariance of the sets $T_l$, and using \eqref{largesetinvariance}, we have, for each $u\in G^{(0)}$, that $|B_{k,l,c,Q}u|\geq(1-\kappa)|T_lu|$. 
We can alternatively describe $B_{k,l,c,Q}$ as the collection of elements $t\in T'_{k,l,c}$ such that $K^Qt\subset T'_{k,l,c}$. 

Suppose $g\in S_{k,p}$ is such that $g\in\bigcup_{\tilde{c}\in C_{k,l}}B_{k,l,\tilde{c},Q}\tilde{c}$, so that $g=tc$, with $t\in B_{k,l,c,Q}$ for $c\in S_{k,j}\subset C_{k,l}$, say. Given $c'\in S_{k,j}$, define $t'\in T_{k,l,c'}'$ to be the element of $Gr(c')$ which implements the same vector of translation that $t$ does. 
Observe that, since $r(t)$ and $r(t')$ are in the same tower level, they agree on $B_{\eta^{-1}}(0)$, and so $K^Qr(t)$ and $K^Qr(t')$ consist of arrows implementing the same vectors of translation. 
Since $t\in B_{k,l,c,Q}$, 
it follows that $t'\in B_{k,l,c',Q}$, and therefore $t'c'\in \bigsqcup_{\tilde{c}\in C_{k,l}}B_{k,l,\tilde{c},Q}\tilde{c}$. Since $t'c'$ was an arbitrary element of $S_{k,p}$, this shows that either $S_{k,p}\cap\bigsqcup_{\tilde{c}\in C_{k,l}}B_{k,l,\tilde{c},Q}\tilde{c}=\emptyset$, or $S_{k,p}\subset\bigsqcup_{\tilde{c}\in C_{k,l}}B_{k,l,\tilde{c},Q}\tilde{c}$.

For each $q\in\{0,\ldots,Q-1\}$, set
\[B_{k,l,c,q}=K^{Q-q}B_{k,l,c,Q}\setminus K^{Q-q-1}B_{k,l,c,Q}.\]
Since $G^{(0)}\subset K$, we have, for each $q\in\{0,\ldots,Q\}$, that $B_{k,l,c,q}\subset (G^{(0)})^q(K^{Q-q}B_{k,l,c,Q}\setminus K^{Q-q-1}B_{k,l,c,Q})\subset K^QB_{k,l,c,Q}$, so these sets partition $K^QB_{k,l,c,Q}$.
In addition, we have $B_{k,l,c,Q}\subset (G^{(0)})^QB_{k,l,c,Q}\subset K^QB_{k,l,c,Q}$, so that, for each $u\in G^{(0)}$, $|K^QB_{k,l,c,Q}u|\geq|B_{k,l,c,Q}u|\geq(1-\kappa)|T_lu|$.

Since $K^QB_{k,l,c,Q}$ is a subset of $T_{k,l,c}'$, we have $K^QB_{k,l,c,Q}c\subset S_k$. Also, since $K^qx$ and $K^qy$ consist of groupoid elements which implement the same vectors of translation whenever $q\in\{1,\ldots,Q\}$ and $d(x,y)<\eta$, and since we have either $S_{k,p}\subset \bigsqcup_{c\in C_{k,l}}B_{k,l,c,Q}c$ or $S_{k,p}\cap\bigsqcup_{c\in C_{k,l}}B_{k,l,c,Q}c=\emptyset$, we see that, for each $k,l,q$ and $j$, we have either $S_{k,j}\subset \bigsqcup_{c\in C_{k,l}}B_{k,l,c,q}c$ or $S_{k,j}\cap\bigsqcup_{c\in C_{k,l}}B_{k,l,c,q}c=\emptyset$. In other words, we may ensure that all the elements of $S_k$ which correspond to any one clopen tower level are all associated to the same $q$.  

We claim that, for each $q\in\{1,\ldots,Q\}$ and $i\in\{1,\ldots,n\}$, whenever $t\in B_{k,l,c,q}$ we have $t^{(i)}\in B_{k,l,c,q}$ as well. Suppose that $s(t)=r(\Lambda_{k,r,j}(c))$, and observe that $s(t^{(i)})=r(\Lambda_{k,i,j}(c))$. By construction, $s(t)$ and $s(t^{(i)})$ are in the same $U_m$, and so the distance between them is smaller than $\eta$, so these tilings agree on $B_{QR_K+R_L}$. By construction, this means that the groupoid elements in $K^QT_L$ implement the same set of translations at both of these tilings. We assumed that $t\in B_{k,l,c,q}$, so we can find $b\in B_{k,l,c,Q}$ and $g\in K^{Q-q}$ such that $gb\notin K^{Q-q-1}B_{k,l,c,Q}$, and such that $t=gb$. By construction, since $gb\in K^QT_L$, the vector which implements $b$ is also allowable at $s(t^{(i)})$, and is implemented by $\tilde{b}\in T_L$, 
say, and then the vector which implements $g$ at $r(b)$ is also allowable at $r(\tilde b)$, and is implemented by $\tilde g\in\Rpunc$, say. Since $r(b)$ and $r(\tilde{b})$ agree on $B_{QR_K}(0)$, and since $g\in K^{Q-q}$, it follows from our choice of $K$ that $\tilde{g}\in K^{Q-q}$.
Next, we show that $\tilde b\in B_{k,l,c,Q}$. To do so, we prove that $K^Q\tilde b\subset T'_{k,l,c}$. Indeed, we assumed that $b\in B_{k,l,c,Q}$, so we know that $K^Qb\subset T'_{k,l,c}$. Since $r(b)$ and $r(\tilde b)$ agree on $B_{QR_K}(0)$, we know that $K^Qr(b)$ and $K^Qr(\tilde b)$ consist of the same vectors of translation. By construction of $T'_{k,l,c}$, since $s(b)$ and $s(\tilde b)$ agree on $B_{QR_K+R_L}(0)$, the sets $T'_{k,l,c}s(b)$ and $T'_{k,l,c}s(\tilde b)$ also implement the same set of vectors. Thus, since $b$ and $\tilde{b}$ implemented the same vector, we see that, whenever $k\in K^Qr(b)$ and $\tilde k\in K^Qr(\tilde b)$ implement the same translation, we have $kb\in T'_{k,l,c}$ if and only if $\tilde k\tilde b\in T'_{k,l,c}$. Since $K^Qb\subset T'_{k,l,c}$ by assumption, this shows that $K^Q\tilde b\subset T'_{k,l,c}$, so $\tilde b\in B_{k,l,c,Q}$, as required. Finally, we must show that $\tilde{g}\tilde{b}\notin K^{Q-q-1}B_{k,l,c,Q}$. This follows because $\tilde{g}\tilde{b}$ implements the same vector as $gb$, which was not an element of $K^{Q-q-1}B_{k,l,c,Q}$, and the elements of the sets $K^{Q-q-1}B_{k,l,c,Q}s(b)$ and $K^{Q-q-1}B_{k,l,c,Q}s(\tilde{b})$ are implemented by precisely the same set of vectors.

Given $g\in K$, it is clear that
\begin{equation}\label{BQ}
gB_{k,l,c,Q}\subset KB_{k,l,c,Q}\subset B_{k,l,c,Q-1}\cup B_{k,l,c,Q}.
\end{equation}
For $q\in\{1,\ldots,Q-1\}$, we have
\begin{equation}\label{Bq}
gB_{k,l,c,q}\subset B_{k,l,c,q-1}\cup B_{k,l,c,q}\cup B_{k,l,c,q+1},
\end{equation}
because it is clear that $gB_{k,l,c,q}\subset K^{Q-q+1}B_{k,l,c,Q}=\bigsqcup_{j=q-1}^QB_{k,l,c,j}$, 
while, given $h\in B_{k,l,c,q}$, if we had $gh\in K^{Q-q-2}B_{k,l,c,Q}=\bigsqcup_{j=q+2}^QB_{k,l,c,j}$
then the closure of $K$ under inverses would provide $h\in g^{-1}K^ {Q-q-2}B_{k,l,c,Q}\subset KK^{Q-q-2}B_{k,l,c,Q}=K^{Q-q-1}B_{k,l,c,Q}$, 
which contradicts the membership of $h$ in $B_{k,l,c,q}$. 

We obtain a $*$-homomorphism $\psi:M_n\rightarrow C(X)\rtimes_r G$ by defining it on the standard matrix units $\{e_{ij}\}_{i,j=1}^n$ of $M_n$ by
\[\psi(e_{ij})=\sum_{k=1}^N\sum_{l=1}^L\sum_{m=1}^M\sum_{c\in C_{k,l,m}^{(j)}}\sum_{q=1}^Q\sum_{t\in B_{k,l,c,q}r(c)}
U_{t^{(i)}\Lambda_{k,i,j}(c)c^{-1}t^{-1}}1_{tcW_k}\]
where $t^{(i)}\in T_{k,l,\Lambda_{k,i,j}(c)}r(\Lambda_{k,i,j}(c))$ is the element constructed earlier, and extending linearly.

For each $k\in\{1,\ldots,N\}$, let 
$h_k:X\rightarrow\lbrack0,1\rbrack$ denote the indicator function for the clopen set $W_k$, and note that, since $W_k$ is clopen, $h_k\in C(X)$. 
For each $k\in\{1,\ldots,N\}$, $l\in\{1,\ldots,L\}$, $m\in\{1,\ldots,M\}$, $i,j\in\{1,\ldots,n\}$ and $c\in C_{k,l,m}^{(j)}$, we set
\begin{align*}
h_{k,l,c,i,j}&=\sum_{q=1}^Q\sum_{t\in B_{k,l,c,q}r(c)}
\frac{q}{Q}U_{t^{(i)}\Lambda_{k,i,j}(c)c^{-1}t^{-1}}\alpha_{tc}(h_k)\\
&=\sum_{q=1}^Q\sum_{t\in B_{k,l,c,q}r(c)}\frac{q}{Q}U_{t^{(i)}\Lambda_{k,i,j}(c)c^{-1}t^{-1}}1_{tcW_k}.
\end{align*}
Define a linear map $\phi:M_n\rightarrow C(X)\rtimes_r G$ by setting
\[\phi(e_{ij})=\sum_{k=1}^N\sum_{l=1}^L\sum_{m=1}^M\sum_{c\in C_{k,l,m}^{(j)}}h_{k,l,c,i,j}\]
and extending linearly. \label{endsetup} Set
\[h=\sum_{k=1}^N\sum_{l=1}^L\sum_{m=1}^M\sum_{i=1}^n\sum_{c\in C_{k,l,m}^{(i)}}h_{k,l,c,i,i}\]
so that $h:X\rightarrow\lbrack0,1\rbrack$ is a continuous function.
We check that $h$ commutes with the image of $\psi$, and that
\[\phi(b)=h\psi(b)\]
for every $b\in M_n$, which will show that $\phi$ is an order-zero c.p.c map. 

We have
\begin{align*}
h\psi(e_{ij})&=\left(\sum_{k=1}^N\sum_{l=1}^L\sum_{m=1}^M\sum_{r=1}^n\sum_{c\in C_{k,l,m}^{(r)}}\sum_{q=1}^Q\sum_{t\in B_{k,l,c,q}r(c)}\frac{q}{Q}U_{t^{(r)}\Lambda_{k,r,r}(c)c^{-1}t^{-1}}1_{tcW_k}\right)\\&\hspace{10mm}\left(\sum_{\tilde k=1}^N\sum_{\tilde l=1}^L\sum_{\tilde m=1}^M\sum_{\tilde c\in C_{\tilde k,\tilde l,\tilde m}^{(j)}}\sum_{\tilde q=1}^Q\sum_{\tilde t\in B_{\tilde k,\tilde l,\tilde c,\tilde q}r(\tilde c)}U_{\tilde t^{(i)}\Lambda_{\tilde k,i,j}(\tilde c)\tilde c^{-1}\tilde t^{-1}}1_{\tilde t\tilde cW_{\tilde k}}\right)\\
&=\sum_{k,\tilde k=1}^N\sum_{l,\tilde l=1}^L\sum_{m, \tilde m=1}^M\sum_{r=1}^n\sum_{c\in C_{k,l,m}^{(r)}}\sum_{\tilde c\in C_{\tilde k, \tilde l,\tilde m}^{(j)}}\\&\hspace{10mm}\sum_{\tilde q=1}^Q\sum_{\tilde t\in B_{\tilde k,\tilde l,\tilde c,\tilde q}r(\tilde c)}\sum_{q=1}^Q\sum_{t\in B_{k,l,c,q}r(c)}\frac{q}{Q}U_{t^{(r)}t^{-1}}1_{tcW_k}U_{\tilde t^{(i)}\Lambda_{\tilde k,i,j}(\tilde c)\tilde c^{-1}\tilde t^{-1}}1_{\tilde t\tilde cW_{\tilde k}}\\
&=\sum_{k,\tilde k=1}^N\sum_{l,\tilde l=1}^L\sum_{m, \tilde m=1}^M\sum_{r=1}^n\sum_{c\in C_{k,l,m}^{(r)}}\sum_{\tilde c\in C_{\tilde k, \tilde l,\tilde m}^{(j)}}\\&\hspace{10mm}\sum_{\tilde q=1}^Q\sum_{\tilde t\in B_{\tilde k,\tilde l,\tilde c,\tilde q}r(\tilde c)}\sum_{q=1}^Q\sum_{t\in B_{k,l,c,q}r(c)}\frac{q}{Q}U_{t^{(r)}t^{-1}}1_{tcW_k}1_{\tilde t^{(i)}\Lambda_{\tilde k,i,j}(\tilde c)W_{\tilde k}}U_{\tilde t^{(i)}\Lambda_{\tilde k,i,j}(\tilde c)\tilde c^{-1}\tilde t^{-1}}
\end{align*}
Notice that, since $t$ is associated to $c\in C_{k,l,m}^{(r)}$, we have $t^{(r)}=t$. Indeed, let $t=(r(c)-z,r(c))$, so that $t^{(r)}=(r(\Lambda_{k,r,r}(c))-z, r(\Lambda_{k,r,r}(c)))=(r(c)-z,r(c))=t$.
Now, we have $c\in C_{k,l,m}^{(r)}\subset C_{k,l}$ and $t\in T_{k,l,c}'\subset T_{k,l,c}\subset T_l$, and similarly $\Lambda_{\tilde k,i,j}(\tilde c)\in C_{\tilde k,\tilde l}$ and $\tilde t^{(i)}\in T_{\tilde l}$, so that $tc\in S_k$ and $\tilde t^{(i)}\Lambda_{\tilde k,i,j}(\tilde c)\in S_{\tilde k}$, so for the supports of the indicator functions in the above sum to intersect, and hence the corresponding term to be nonzero, we must have $k=\tilde k$ (because levels of different towers are disjoint).  
Since $tc$ and $\tilde t^{(i)}\Lambda_{k,i,j}(\tilde c)$ are both elements of $S_k$, for the indicators accociated to them to intersect, we must have $tc=\tilde t^{(i)}\Lambda_{k,i,j}(\tilde c)$. But $tc\in T_{k,l,c}c$, and $\tilde t^{(i)}\Lambda_{k,i,j}(\tilde c)\in T_{k,\tilde l,\Lambda_{k,i,j}(\tilde c)}\Lambda_{k,i,j}(\tilde c)$, and these sets are disjoint unless $l=\tilde l$ and $c=\Lambda_{k,i,j}(\tilde c)$ (which implies also that $m=\tilde m$ and $r=i$).
Combining this with the fact that $tc=\tilde t^{(i)}\Lambda_{k,i,j}(\tilde c)=\tilde t^{(i)}c$ shows that $t=\tilde t^{(i)}$, and hence that $q=\tilde q$. 
Therefore, we only need to sum over $k,l,m$, $\tilde c\in C^{(j)}_{k,l,m}$, 
and $\tilde q$ and $\tilde t\in B_{k,l,\tilde c,\tilde q}r(\tilde c)$
 (because knowing $\tilde c$ and $\tilde t$ allows us to determine $c$ and $t$). Putting this all together yields
\begin{align*}
h\psi(e_{ij})&=\sum_{k=1}^N\sum_{l=1}^L\sum_{m=1}^M\sum_{\tilde{c}\in C_{k,l,m}^{(j)}}\sum_{\tilde q=1}^Q\sum_{\tilde t\in B_{k,l,\tilde c,\tilde q}r(\tilde c)}\frac{\tilde q}{Q}U_{\tilde t^{(i)}(\tilde t^{(i)})^{-1}}1_{\tilde t^{(i)}\Lambda_{k,i,j}(\tilde c)W_{k}}U_{\tilde t^{(i)}\Lambda_{k,i,j}(\tilde c)\tilde c^{-1}\tilde t^{-1}}\\
&=\sum_{k=1}^N\sum_{l=1}^L\sum_{m=1}^M\sum_{\tilde c\in C_{k,l,m}^{(j)}}\sum_{\tilde q=1}^Q\sum_{\tilde t\in B_{k,l,\tilde c,\tilde q}r(\tilde c)}\frac{\tilde{q}}{Q}U_{r(\tilde t^{(i)})}1_{\tilde t^{(i)}\Lambda_{k,i,j}(\tilde c)W_{k}}U_{\tilde t^{(i)}\Lambda_{k,i,j}(\tilde c)\tilde{c}^{-1}\tilde t^{-1}}\\
&=\sum_{k=1}^N\sum_{l=1}^L\sum_{m=1}^M\sum_{\tilde c\in C_{k,l,m}^{(j)}}\sum_{\tilde q=1}^Q\sum_{\tilde t\in B_{k,l,\tilde c,\tilde q}r(\tilde c)}\frac{\tilde q}{Q}U_{\tilde t^{(i)}\Lambda_{k,i,j}(\tilde c)\tilde{c}^{-1}\tilde t^{-1}}1_{\tilde t\tilde cW_k}\\
&=\phi(e_{ij}).
\end{align*}
Similarly, we compute that
\begin{align*}
\psi(e_{ij})h&=\left(\sum_{\tilde k=1}^N\sum_{\tilde l=1}^L\sum_{\tilde m=1}^M\sum_{\tilde c\in C_{\tilde k,\tilde l,\tilde m}^{(j)}}\sum_{\tilde q=1}^Q\sum_{\tilde t\in B_{\tilde k,\tilde l,\tilde c,\tilde q}r(\tilde c)}U_{\tilde t^{(i)}\Lambda_{\tilde k,i,j}(\tilde c)\tilde c^{-1}\tilde t^{-1}}1_{\tilde t\tilde cW_{\tilde k}}\right)\\&\hspace{10mm}\left(\sum_{k=1}^N\sum_{l=1}^L\sum_{m=1}^M\sum_{r=1}^n\sum_{c\in C_{k,l,m}^{(r)}}\sum_{q=1}^Q\sum_{t\in B_{k,l,c,q}r(c)}\frac{q}{Q}U_{t^{(r)}\Lambda_{k,r,r}(c)c^{-1}t^{-1}}1_{tcW_k}\right)&\\
&=\sum_{k,\tilde k=1}^N\sum_{l,\tilde l=1}^L\sum_{m, \tilde m=1}^M\sum_{r=1}^n\sum_{c\in C_{k,l,m}^{(r)}}\sum_{\tilde c\in C_{\tilde k, \tilde l,\tilde m}^{(j)}}\\&\hspace{10mm}\sum_{\tilde q=1}^Q\sum_{\tilde t\in B_{\tilde k,\tilde l,\tilde c,\tilde q}r(\tilde c)}\sum_{q=1}^Q\sum_{t\in B_{k,l,c,q}r(c)}\frac{q}{Q}U_{\tilde t^{(i)}\Lambda_{\tilde k,i,j}(\tilde c)\tilde c^{-1}\tilde t^{-1}}1_{\tilde t\tilde cW_{\tilde k}}U_{r(t)}1_{tcW_k}.\\
\end{align*}
This time, the only nonzero contribution occurs when $k=\tilde k$ and when $\tilde t\tilde c=tc$. Using similar arguments as above, this forces $c=\tilde c$ and $t=\tilde t$, and hence $q=\tilde q$, $l=\tilde l$, $m=\tilde m$, and $r=j$. Thus, the sum above becomes
\begin{align*}
\psi(e_{ij})h&=\sum_{k=1}^N\sum_{l=1}^L\sum_{m=1}^M\sum_{c\in C_{k,l,m}^{(j)}}\sum_{q=1}^Q\sum_{t\in B_{k,l,c,q}r(c)}\frac{q}{Q}U_{t^{(i)}\Lambda_{k,i,j}(c)c^{-1}t^{-1}}1_{tcW_k}\\
&=\phi(e_{ij}).
\end{align*} 

Next, we verify condition (ii) in Theorem~\ref{Zstable} for the element $w=\sum_{g\in K}U_g$. Let $1\leq i,j\leq n$. Using the fact that $K$ is closed under taking inverses in the sum over $K$ in the second term, we have
\begin{align*}
wh_{k,l,c,i,j}-h_{k,l,c,i,j}w&=\sum_{q=1}^Q\sum_{t\in B_{k,l,c,q}r(c)}\sum_{g\in Kr(t^{(i)})}\frac{q}{Q}U_{gt^{(i)}\Lambda_{k,i,j}(c)c^{-1}t^{-1}}1_{tcW_k}\\
&\hspace{10mm}-\sum_{\tilde q=1}^Q\sum_{\tilde t\in B_{k,l,c,\tilde q}r(c)}\sum_{\tilde g\in Kr(\tilde t)}\frac{\tilde q}{Q}U_{\tilde t^{(i)}\Lambda_{k,i,j}(c)c^{-1}\tilde t^{-1}}1_{\tilde tcW_k}U_{\tilde g^{-1}}\\
=\sum_{q,\tilde q=1}^Q\sum_{\substack{t\in B_{k,l,c,q}r(c)\\\tilde t\in B_{k,l,c,\tilde q}r(c)}}\sum_{\substack{g\in Kr(t^{(i)})\\\tilde g\in Kr(\tilde t)}}&\frac{q}{Q} U_{gt^{(i)}\Lambda_{k,i,j}(c)c^{-1}t^{-1}}1_{tcW_k}-\frac{\tilde q}{Q}U_{\tilde t^{(i)}\Lambda_{k,i,j}(c)c^{-1}(\tilde g\tilde t)^{-1}}1_{\tilde g\tilde tcW_k}.
\end{align*}
In view of \eqref{BQ} and \eqref{Bq}, we may pair up the elements in the first and second terms as follows. Because $gt^{(i)}\in B_{k,l,c,q+a}$ for $a\in\{-1,0,1\}$, we want to associate the element $\frac{q}{Q} U_{gt^{(i)}\Lambda_{k,i,j}(c)c^{-1}t^{-1}}1_{tcW_k}$ to an element corresponding to $\tilde q=q+a$. To find an appropriate $\tilde g$ and $\tilde t$, observe that by construction there is an element $h$ of $K$ with source equal to $r(t)$ which implements the same vector of translation as $g$. Set $\tilde t=ht$, noting that by our construction, $\tilde t\in B_{k,l,c,q+a}$, and set $\tilde g=h^{-1}$ so that $\tilde g\tilde t=t$. Observe that by construction, we obtain $\tilde t^{(i)}=gt^{(i)}$, so the term of the sum which corresponds to these choices of $g,\tilde g$, $q,\tilde q$ and $t,\tilde t$ becomes
\[\frac{-a}{Q}U_{gt^{(i)}\Lambda_{k,i.j}(c)c^{-1}t^{-1}}1_{tcW_k}=\frac{-a}{Q}1_{gt^{(i)}\Lambda_{k,i,j}(c)W_k}U_{gt^{(i)}\Lambda_{k,i.j}(c)c^{-1}t^{-1}},\]
and the norm of this element is no larger than $1/Q<\epsilon/(n^2\sup_{u\in G^{(0)}}|uK|)$. 
Now, all of the associated indicator functions are supported on tower levels $gt^{(i)}\Lambda_{k,i,j}(c)W_k$, and the only time two of these levels can intersect is if they are associated to two different $g\in K$ with the same range. Thus, we obtain
\[\|wh_{k,l,c,i,j}-h_{k,l,c,i,j}w\|\leq \frac{\sup_{u\in G^{(0)}}|uK|}{Q}<\frac{\epsilon}{n^2}.\]

All of the associated indicator functions in sight here are supported in $K^QB_{k,l,c,Q}cW_k$. Since these sets are pairwise disjoint for distinct $k$, $l$, and $c$, we obtain

\[\|w\phi(e_{ij})-\phi(e_{ij})w\|=\sup_{k,l,c}\|wh_{k,l,c,i,j}-h_{k,l,c,i,j}w\|\leq\frac{\epsilon}{n^2}.\]
Hence, for any norm-one $b=(b_{ij})\in M_n$, we have
\begin{align*}
\|\lbrack w,\phi(b)\rbrack\|&=\|w\phi(b)-\phi(b)w\|\\
&\leq\sum_{i,j=1}^n\|w\phi(b_{ij})-\phi(b_{ij})w\|\leq n^2\cdot\frac{\epsilon}{n^2}=\epsilon. 
\end{align*}

Next, we check condition (ii) in Theorem~\ref{Zstable} for the functions in $\Upsilon$. Let $i,j\in\{1,\ldots,n\}$ and $f\in\Upsilon\cup\Upsilon^2$. Let $k\in\{1,\ldots,N\}$ and $l\in\{1,\ldots,L\}$. Let $c\in C_{k,l,m}^{(j)}$.
Since $k,i,j$ are fixed in the calculation below, we write $\Lambda=\Lambda_{k,i,j}$ for short.
\begin{align*}
h_{k,l,c,i,j}^*fh_{k,l,c,i,j}&=\left(\sum_{q=1}^Q\sum_{t\in B_{k,l,c,q}r(c)}\frac{q}{Q}U_{t^{(i)}\Lambda_{k,i,j}(c)c^{-1}t^{-1}}\alpha_{tc}(1_{W_k})\right)^*f\\ &\hspace{10mm}\left(\sum_{\tilde{q}=1}^Q\sum_{\tilde{t}\in B_{k,l,c,\tilde{q}}r(c)}\frac{\tilde q}{Q}U_{\tilde{t}^{(i)}\Lambda_{k,i,j}(c)c^{-1}\tilde{t}^{-1}}\alpha_{\tilde{t}c}(1_{W_k})\right)\\
&=\sum_{q, \tilde q=1}^Q\sum_{t\in B_{k,l,c,q}r(c)}\sum_{\tilde t\in B_{k,l,c,\tilde{q}}r(c)}\frac{q\tilde q}{Q^2}1_{tcW_k}U_{tc\Lambda(c)^{-1}(t^{(i)})^{-1}}fU_{\tilde t^{(i)}\Lambda(c)c^{-1}\tilde t^{-1}}1_{\tilde tcW_k}.
\end{align*}
To simplify the notation, let $g_t=tc$ and $h_t=t^{(i)}\Lambda(c)$. Then we have
\begin{align*}
h_{k,l,c,i,j}^*fh_{k,l,c,i,j}&=\sum_{q,\tilde q=1}^Q\sum_{t\in B_{k,l,c,q}r(c)}\sum_{\tilde t\in B_{k,l,c,\tilde{q}}r(c)}\frac{q\tilde q}{Q^2}1_{g_tW_k}\alpha_{g_th_t^{-1}}(f)U_{g_th_t^{-1}h_{\tilde t}g_{\tilde{t}}}1_{g_{\tilde t}W_k}\\
&=\sum_{q,\tilde q=1}^Q\sum_{t\in B_{k,l,c,q}r(c)}\sum_{\tilde t\in B_{k,l,c,\tilde{q}}r(c)}\frac{q\tilde q}{Q^2}\alpha_{g_th_{t}^{-1}}(f)1_{g_tW_k}1_{g_th_t^{-1}h_{\tilde t}W_k}U_{g_th_t^{-1}h_{\tilde t}g_{\tilde t}^{-1}}.\end{align*}
Extracting the important part, this term is associated to the product of indicator functions
\[1_{tcW_k}1_{tc\Lambda(c)^{-1}(t^{(i)})^{-1}\tilde t^{(i)}\Lambda(c)W_k}.\]
The groupoid element $(t^{(i)})^{-1}\tilde t^{(i)}$ is only defined when $r(t^{(i)})=r(\tilde t^{(i)})$. We already know that $s(t^{(i)})=s(\tilde t^{(i)})=r(\Lambda_{k,i,j}(c))$, so, by principality of $\Rpunc$, this element is only defined when $t^{(i)}=\tilde t^{(i)}$. Since $t^{(i)}$ and $\tilde t^{(i)}$ are associated to the same vectors in $r(\Lambda_{k,i,j}(c))$ as $t$ and $\tilde t$ are in $c$, this shows that the only contributions to the sum occur when $t=\tilde t$, and hence when $q=\tilde q$. Thus, we obtain
\begin{equation}\label{h*fh}
h_{k,l,c,i,j}^*fh_{k,l,c,i,j}=\sum_{q=1}^Q\sum_{t\in B_{k,l,c,q}r(c)}\frac{q^2}{Q^2}\alpha_{tc\Lambda_{k,i,j}(c)^{-1}(t^{(i)})^{-1}}(f)1_{tcW_k}.
\end{equation}

Similarly, we obtain
\begin{align*}
fh_{k,l,c,i,j}^*&h_{k,l,c,i,j}=f\left(\sum_{q=1}^Q\sum_{t\in B_{k,l,c,q}r(c)}\frac{q}{Q}U_{t^{(i)}\Lambda_{k,i,j}(c)c^{-1}t^{-1}}\alpha_{tc}(1_{W_k})\right)^*\\ &\hspace{24mm}\left(\sum_{\tilde{q}=1}^Q\sum_{\tilde{t}\in B_{k,l,c,\tilde{q}}r(c)}\frac{\tilde q}{Q}U_{\tilde{t}^{(i)}\Lambda_{k,i,j}(c)c^{-1}\tilde{t}^{-1}}\alpha_{\tilde{t}c}(1_{W_k})\right)\\
&=\sum_{q, \tilde q=1}^Q\sum_{t\in B_{k,l,c,q}r(c)}\sum_{\tilde t\in B_{k,l,c,\tilde{q}}r(c)}\frac{q\tilde q}{Q^2}f1_{tcW_k}U_{tc\Lambda_{k,i,j}(c)^{-1}(t^{(i)})^{-1}\tilde t^{(i)}\Lambda_{k,i,j}(c)c^{-1}\tilde t^{-1}}1_{\tilde tcW_k}.
\end{align*}
As before, this term is only defined when $((t^{(i)})^{-1},\tilde t^{(i)})\in G^{(2)}$, which forces $t=\tilde t$ and $q=\tilde q$, whence we obtain
\begin{equation}\label{fh*h}
fh_{k,l,c,i,j}^*h_{k,l,c,i,j}=\sum_{q=1}^Q\sum_{t\in B_{k,l,c,q}r(c)}\frac{q^2}{Q^2}f1_{tcW_k}.
\end{equation}

Now, let $x\in W_k$. By definition of $C_{k,l,m}$ and $\Lambda_{k,i,j}$, both $\Lambda_{k,i,j}(c)x$ and $cx$ belong to $U_m$, which had diameter less than $\eta$. In addition, we have $d(t,t^{(i)})<\eta$ for all $t$ and $i$. Therefore, by the definition of $\eta$, we obtain that
\[|f(t^{(i)}\Lambda_{k,i,j}(c)x)-f(tcx)|<\frac{\epsilon^2}{4n^4}\]
and hence
\begin{align*}
\sup_{y\in tcW_k}&\left|\left(\alpha_{tc\Lambda_{k,i,j}(c)^{-1}(t^{(i)})^{-1}}(f)-f\right)(y)\right| \\
&=\sup_{x\in W_k}\left|\left(\alpha_{tc\Lambda_{k,i,j}(c)^{-1}(t^{(i)})^{-1}}(f)-f\right)(tcx)\right|\\
&=\sup_{x\in W_k}\left|\alpha_{tc\Lambda_{k,i,j}(c)^{-1}(t^{(i)})^{-1}}(f)(tcx)-f(tcx)\right|\\
&=\sup_{x\in W_k}|f(t^{(i)}\Lambda_{k,i,j}(c)x)-f(tcx)|\\
&\leq\frac{\epsilon^2}{4n^4}.
\end{align*}

Using this along with \eqref{h*fh} and \eqref{fh*h}, 
and the fact that the tower levels $tcW_k$ are disjoint for distinct $t$, we obtain
\begin{align}
\|h_{k,l,c,i,j}^*&fh_{k,l,c,i,j}-fh_{k,l,c,i,j}^*h_{k,l,c,i,j}\| \\
&=\underset{q\in\{1,\ldots,Q\}}\max\underset{t\in B_{k,l,c,q}}\sup\frac{q^2}{Q^2}\|(\alpha_{tc\Lambda_{k,i,j}(c)^{-1}(t^{(i)})^{-1}}(f)-f)1_{tcW_k}\| <\frac{\epsilon^2}{3n^4} \label{hfineq}
\end{align}

Now, set $w=\phi(e_{ij})$ and fix $f\in\Upsilon$. Since the indicator functions associated to $h_{k,l,c,i,j}$ have pairwise disjoint supports for distinct $k\in\{1,\ldots,N\}$, $l\in\{1,\ldots,L\}$, $m\in\{1,\ldots,M\}$ and $c\in C_{k,l,m}^{(j)}$, it follows from \eqref{hfineq} that
\[\|w^*gw-gw^*w\|<\frac{\epsilon^2}{3n^4}\]
where $g$ is equal to $f$ or $f^2$ (because in \eqref{hfineq} we were working with $f\in\Upsilon\cup\Upsilon^2$). From this, it follows that
\begin{align*}
\|w^*f^2w-fw^*fw\|&\leq\|w^*f^2w-f^2w^*w\|+\|f^2w^*w-fw^*fw\|\\
&<\frac{\epsilon^2}{3n^4}+\|f(fw^*w-w^*fw)\|\\
&\leq\frac{\epsilon^2}{3n^4}+\|f\|\|fw^*w-w^*fw\|\\
&<\frac{\epsilon^2}{3n^4}+\frac{\epsilon^2}{3n^4}\\
&=\frac{2\epsilon^2}{3n^4}.
\end{align*}
Hence
\begin{align*}
\|fw-wf\|^2&=\|(fw-wf)^*(fw-wf)\|\\
&=\|w^*f^2w-fw^*fw+fw^*wf-w^*fwf\|\\
&\leq\|w^*f^2w-fw^*fw\|+\|(fw^*w-w^*fw)f)\|\\
&<\frac{2\epsilon^2}{3n^4}+\frac{\epsilon^2}{3n^4}\\
&=\frac{\epsilon^2}{n^4}
\end{align*}
so that
\[\|fw-wf\|<\frac{\epsilon}{n^2}.\]
Therefore, for every norm-one $b=(b_{ij})\in M_n$, we have
\[\|\lbrack f,\phi(b)\rbrack\|\leq\sum_{i,j=1}^n\|\lbrack f,\phi(b_{ij})\rbrack\|<n^2\cdot\frac{\epsilon}{n^2}=\epsilon.\]

To complete the proof, we now show that $1-\phi(I)\precsim a$. By enlarging $E$ to include $D$, and shrinking $\delta$ if necessary, we can ensure that the sets $S_1,\ldots,S_N$ are sufficiently left-invariant under $D$ so that, for each $k\in\{1,\ldots,N\}$ and 
$u\in G^{(0)}$, the set $R_ku\coloneqq\{s\in S_ku\mid Ds\subset S_k\}$ has cardinality at least $|S_ku|/2$. Note that, since $D^{-1}\cdot O=X$, we have $s(D)=X=G^{(0)}$. Thus, for any $t\in G$, we have $t\in D^{-1}Dt$, because $Dt\neq\emptyset$.  

Fix $k\in\{1,\ldots,N\}$. Let $R_k=\bigsqcup_{u\in G^{(0)}} R_ku\subset S_k$. Let $R_k'$ be a maximal subset of $R_k$ with the property that the collection $\{Ds\mid s\in R_k'\}$ is pairwise disjoint. Observe that if $s,t\in R_k$ have $Ds\cap Dt\neq\emptyset$, then $s\in D^{-1}Dt$. Therefore, if $R_k'$ is chosen in such a way that there exists $t\in R_k$ such that $R_k'\cap D^{-1}Dt=\emptyset$, then, by including any element of $D^{-1}Dt\cap R_k$ in $R_k'$ (for instance, the choice $t\in D^{-1}Dt\cap R_k$ is always valid), we obtain a larger set which satisfies the disjointness condition on $R_k'$, contradicting the maximality of $R_k'$. So, for each $u\in G^{(0)}$,  $R_k'u$ contains at least one element from each subset of the form $D^{-1}Dt$ for $t\in R_ku$. Now, the smallest that we can make $R_k'u$ is to assume that we choose exactly one element from each such set, and to assume that all of these sets are subsets of $R_ku$. 
In this case, we get
\[|R_k'u|\geq\frac{|R_ku|}{\max_{t\in R_ku}|D^{-1}Dt|}\geq\frac{|S_ku|}{2\max_{u\in G^{(0)}}|D^{-1}Du|}.\]
Put $P=\max_{u\in G^{(0)}}|D^{-1}Du|$, and note that $P<\infty$ since we assumed that $|Du|$ and $|D^{-1}u|=|uD|$ were uniformly bounded for $u\in G^{(0)}$.

By constructing $D$ using the basis of clopen bisections $\{V(T\sqcap B_r(0),t,t')\}$, where we only consider sufficiently large values of $r$, we may ensure that the sets $Dx$ and $Dy$ consist of groupoid elements which implement the same vectors whenever $x$ and $y$ are sufficiently close.
This will allows us to choose $R_k'$ to contain either all of the groupoid elements (of some $S_{k,j}$) which make up a particular level of the tower $(W_k,S_k)$, or none of them. Combining this with the inequality above, we may assume that $(W_k,R_k')$ is a subtower of $(W_k,S_k)$ which contains at least one $(2P)$-th of the levels of the tower $(W_k,S_k)$. 

Since $D^{-1}O=X$, for each $s\in R_k'$ there exists $t\in D$ such that $tsW_k$ intersects $O$. Indeed, $s\in S_k$, so $sW_k
\neq\emptyset$, and thus we can find $t\in D$ and $w\in O$ such that $t^{-1}w=r(s)\in sW_k$, which implies that $w\in tsW_k$ and hence that $tsW_k\cap O\neq\emptyset$. By construction of $R_k$, we have $ts\in S_k$, so $tsW_k$ is a subset of a level of our tower, which was assumed to have diameter less than $\eta'\leq\theta$. Thus, since $O$ was the ball of radius $\theta$ centred at $x_0$, and since $tsW_k$ intersects $O$ and has diameter smaller than $\theta$, the tower level containing $tsW_k$ is contained in the ball of radius $2\theta$ centred at $x_0$, and thus $a$ takes value 1 on this entire tower level. Furthermore, since the sets $Ds$ for $s\in R_k'$ are disjoint, all the the elements $ts\in S_k$ constructed here are distinct. 
Thus, we have constructed an injective association $s\mapsto\tilde s$ of elements $s\in R_k'$ to elements $\tilde s=ts\in S_k$, such that $a$ takes the constant value 1 on the tower level which contains $r(\tilde s)$. This shows that, for each $u\in G^{(0)}$, the set $S_k^\sharp u$ of elements $t\in S_ku$ such that $a$ takes the constant value 1 on the tower level containing $r(t)$ has cardinality $|S_k^\sharp u|\geq|R_k'u|\geq|S_ku|/(2P)$. 
Note that, for any $j$, if $t\in S_{k,j}$ is an element of $S_k^\sharp$, then all elements $s\in S_{k,j}$ are in $S_k^\sharp$ as well, since $t\in S_k^\sharp$ implies that $a$ takes the constant value $1$ on the tower level $S_{k,j}W_k$. In other words, $(W_k,S_k^\sharp)$ is a subtower of $(W_k,S_k)$, and contains at least one $(2P)$-th of the levels of this tower. 

Set 
\[S_k''=\bigsqcup_{l=1}^L\bigsqcup_{m=1}^M\bigsqcup_{i=1}^n\bigsqcup_{c\in C_{k,l,m}^{(i)}}B_{k,l,c,Q}c,\]
noting that $(W_k,S_k'')$ is a subtower of $(W_k,S_k)$, and fix $u\in G^{(0)}$. We have, for each $v\in G^{(0)}$, that $|B_{k,l,c,Q}v|\geq(1-\kappa)|T_lv|$, 
 so we obtain
\begin{align*}
|S_k''u|&\geq\sum_{l=1}^L\sum_{m=1}^M\sum_{i=1}^n\sum_{c\in C_{k,l,m}^{(i)}u}|B_{k,l,c,Q}r(c)|\\
&\geq(1-\kappa)\sum_{l=1}^L\sum_{m=1}^M\sum_{i=1}^n\sum_{c\in C_{k,l,m}^{(i)}u}|T_{l}r(c)|\\
&=(1-\kappa)\sum_{l=1}^L\sum_{m=1}^M\sum_{i=1}^n\sum_{c\in C_{k,l,m}^{(i)}u}|T_{l}c|\\
&\geq(1-\kappa)\sum_{l=1}^L\sum_{m=1}^M\sum_{i=1}^n\left|\bigcup_{c\in C_{k,l,m}^{(i)}u}T_{l}c\right|\\ 
&=(1-\kappa)\sum_{l=1}^L\sum_{m=1}^M\sum_{i=1}^n|T_{l}C_{k,l,m}^{(i)}u|.
\end{align*}
Recall that $\bigcup_{i=1}^n C_{k,l,m}^{(i)}\subset C_{k,l,m}$ contains all except at most $n$ elements of $C_{k,l,m}$. We use this to continue the computation above, obtaining 
\begin{align*}
|S_k''u|&\geq(1-\kappa)\sum_{l=1}^L\sum_{m=1}^M\sum_{i=1}^n|T_{l}C_{k,l,m}^{(i)}u|\\
&\geq(1-\kappa)\sum_{l=1}^L\sum_{m=1}^M(|T_{l}C_{k,l,m}u|-n\max_{v\in G^{(0)}}|T_{l}v|)\\
&\geq(1-\kappa)\sum_{l=1}^L\sum_{m=1}^M(|T_lC_{k,l,m}u|-n\max_{v\in G^{(0)}}|T_lv|)\\
&\geq(1-\kappa)\sum_{l=1}^L\left(\left|\bigcup_{m=1}^MT_lC_{k,l,m}u\right|-Mn\max_{v\in G^{(0)}}|T_lv|\right)\\
&\geq(1-\kappa)\sum_{l=1}^L(|T_lC_{k,l}u|-Mn\max_{v\in G^{(0)}}|T_lv|)\\
&\geq(1-\kappa)\left(\left|\bigcup_{l=1}^LT_lC_{k,l}u\right|-Mn\sum_{l=1}^L\max_{v\in G^{(0)}}|T_lv|\right)\\
&\geq(1-\kappa)\left((1-\beta)|S_ku|-\beta|S_ku|\right)\\
&\geq(1-\kappa)(1-2\beta)|S_ku|,
\end{align*}
where we used the fact that $\{T_lc\mid l\in\{1,\ldots,L\}, c\in C_{k,l}\}$ is a $(1-\beta)$-cover of $S_k$, and the inequality \eqref{Mnmax} together with $(E,\delta)$-invariance of $S_k$ to obtain the second-to-last line.

So, if $\kappa$ and $\beta$ are small enough,\phantomsection\label{kappadetermined} 
we can ensure that, for each $u\in G^{(0)}$, we have
\[|(S_k\setminus S_k'')u|\leq\frac{|S_ku|}{4P}\leq\frac{|S_k^\sharp u|}{2}.\]

In this case, choose an injection $f_k: S_k\setminus S_k''\rightarrow S_k^\sharp$ which sends $(S_k\setminus S_k'')u$ to $S_k^\sharp u$ for each $u\in G^{(0)}$. Further enforce that if $t,t'\in S_{k,j}$ for some $j$, and $f_k(t)\in S_{k,p}$ for some $p$, then $f_k(t')\in S_{k,p}$ as well. In this way, $f_k$ will map the collection of elements associated to any tower level to a collection of elements associated to some other tower level. Since $\{(W_k,S_k)\}_{k=1}^N$ is a clopen tower decomposition of $X$, we have $X=\bigsqcup_{k=1}^NS_kW_k$, and hence
\[X\setminus\bigsqcup_{k=1}^NS_k''W_k=\bigsqcup_{k=1}^N(S_k\setminus S_k'')W_k.\]
Since $(W_k, S_k\setminus S_k'')$ and $(W_k, S_k^\sharp)$ are subtowers of $(W_k,S_k)$, write $S_k\setminus S_k''=\bigsqcup_{j\in J} S_{k,j}$, and $S_k^\sharp=\bigsqcup_{j^\sharp\in J^\sharp}S_{k,j^\sharp}$, where $J$ and $J^\sharp$ are finite index sets. Observe that the collection of subsets $S_{k,j}W_k$ for $k\in\{1,\ldots,N\}$ and $j\in J$ covers $X\setminus\bigsqcup_{k=1}^N S_k''W_k$. 
For each $j\in J$, and each $t\in S_{k,j}$, consider the element $f_k(t)t^{-1}\in G$, noting that $s(f_k(t))=s(t)=r(t^{-1})$ by construction of $f_k$, so that the multiplication is defined. Then $\bigsqcup_{t\in S_{k,j}}(f_k(t)t^{-1})tW_k=\bigsqcup_{t\in S_{k,j}}f_k(t)W_k$ is a level $S_{k,j^\sharp}$ of the tower $(W_k,S_k^\sharp)$, so, since $f_k$ was injective and all the levels  $S_kW_k$ of the tower $(W_k,S_k)$ were disjoint for $k\in\{1,\ldots,N\}$, the collection $\{\bigsqcup_{t\in S_{k,j}}f_k(t)W_k\mid k\in\{1,\ldots,N\}, j\in J\}$ of open subsets of $\bigsqcup_{k=1}^NS_k^\sharp W_K$ is pairwise disjoint, and consists of images (under the action map) of sets which covered $X\setminus\bigsqcup_{k=1}^N S_k''W_k$. Thus, we have shown that
\[X\setminus\bigsqcup_{k=1}^N S_k''W_k\prec\bigsqcup_{k=1}^N S_k^\sharp W_k.\]

Since the function $1-\phi(I)$ is supported on $X\setminus\bigsqcup_{k=1}^NS_k''W_k$,
and $a$ takes the constant value 1 on $\bigsqcup_{k=1}^NS_k^\sharp W_k$, it follows by Lemma~\ref{David12.3} that there exists a $v\in C(X)\rtimes_r G$ such that $v^*av=1-\phi(I)$, which shows that $1-\phi(I)\precsim a$. 
\end{proof}

\section{Quasidiagonality of tiling algebras}\label{sec:quasidiagonal}
In this section, we give a direct proof that the $C^*$-algebra $A_\textnormal{punc}$ associated to any aperiodic, repetitive tiling $T$ with finite local complexity is \textit{quasidiagonal}. Note that these algebras were already known to be quasidiagonal as a consequence of the general result obtained by Tikuisis, White and Winter \cite[Corollary~B]{TWW}. Our result  sacrifices generality to allow for a less abstract proof specific to tiling $C^*$-algebras. 

Quasidiagonality does not give us any new information regarding the classifiability of these algebras, but it does simplify the machinery required for their classification in certain cases. In particular, under the additional assumption that the algebra has a unique trace, the main results of \cite{MS} and \cite{MS2} show that tiling $C^*$-algebras have finite decomposition rank, and are therefore classified by their Elliott invariant. This allows us to obtain classifiability without reference to some of the complicated constructions culminating in \cite{TWW}. We also note that \cite[Theorem~B]{SWW} provides yet another route to classification in the monotracial setting, which does not require quasidiagonality.

In the light of the above, for the purposes of this section, we are interested in tilings whose $C^*$-algebras have unique trace. It is shown in \cite{Kel} that the $C^*$-algebras associated to aperiodic, repetitive \textit{substitution} tilings with FLC have unique trace. Recall that such $C^*$-algebras were already known to be classifiable in the case that the substitution system causes prototiles to appear in only finitely many orientations by combining \cite[Corollary~3.2]{Win1} and \cite{SW} (see also \cite{DS}). 
However, this route to classification relies on changing the prototile set so that the tling becomes a square tiling with appropriate matching rules and equipped with a $\Z^d$-action, while our classification result requires no such modification of the tiling.

In the non-substitution case, the class of \textit{quasiperiodic} tilings is detailed in \cite[Section~8.2]{Rob}, in which it is stated \cite[Corollary~8.5]{Rob} (see also \cite[Section~12]{Rob96}) that the dynamical systems associated to these tilings are uniquely ergodic, and hence their $C^*$-algebras have unique trace. 

The concept of quasidiagonality was introduced for sets operators by Halmos  in \cite[Page~902]{Halmos}, and extended to representations of $C^*$-algebras by Thayer \cite{Thayer}. We will use the following definition for abstract $C^*$-algebras discussed in \cite[Chapter~7.2]{BrOz}.

\begin{definition}
Let $H$ be a separable Hilbert space. A subset $A\subset B(H)$ is called a {\em quasidiagonal set of operators} if there exists an increasing sequence of finite rank projections, $P_1\leq P_2\leq P_3\leq\cdots$, which converge strongly to the identity operator (that is, for any $x\in H$, $\|P_n(x)-x\|\rightarrow0$ as $n\rightarrow\infty$), and are such that, for every $a\in A$,
\[
\|\lbrack a, P_n \rbrack \|\coloneqq\| aP_n-P_na\|\rightarrow 0 \textnormal{ as }n\rightarrow\infty.
\]
A $C^*$-algebra $A$ is called {\em quasidiagonal} if there exists a faithful representation $\pi:A\rightarrow B(H)$ such that $\pi(A)$ is a quasidiagonal set of operators.
\end{definition}

\begin{thm}
The $C^*$-algebra associated to any aperiodic and repetitive tiling with FLC is quasidiagonal.
\end{thm}
\begin{proof}
Let $\Opunc$ be the punctured hull of such a tiling, and fix a tiling $T \in \Opunc$. In \cite{Ren} it is shown that the induced representation from the unit space $\pi \coloneqq \pi_T$ described on the generating set $\EE$ defined in \eqref{ez-collection}
extends to a faithful nondegenerate representation of $A_\textnormal{punc}$ on $\ell^2(T)\coloneqq\ell^2(\lbrack T\rbrack)$. We will prove that $\pi(A_\textnormal{punc})$ is a quasidiagonal set of operators. Define a sequence of projections $Q_n \in B(\ell^2(T))$ by
\[
Q_n(\delta_t) \coloneqq \begin{cases}
\delta_t & \text{ if } t \in T \sqcap B_n(0)\\
0 & \text{ otherwise.}
\end{cases}
\]
It is clear that $(Q_n)_{n\in\N}$ is an increasing sequence of projections. Since $T$ has FLC, $Q_n$ has finite rank for each $n\in\N$, and $(Q_n)_{n\in\N}$ converges strongly to the identity by square-summability of the elements of $\ell^2(T)$.

We now show that, for each element $e(P,t,t') \in \EE$, we have
\begin{equation*}
\lim\limits_{n\rightarrow\infty}\|Q_n\pi(e(P, t, t^\prime)) - \pi(e(P,t,t^\prime))Q_n\|_{\textnormal{op}}=0.
\end{equation*}
This will prove that $A_\textnormal{punc}$ is quasidiagonal, since $\overline{\text{span}\{\EE\}}=A_\textnormal{punc}$. We first observe that
\begin{align}
\notag
\lim\limits_{n\rightarrow\infty}\|Q_n&\pi(e(P, t, t^\prime)) - \pi(e(P,t,t^\prime))Q_n\|_{\textnormal{op}}\\
\label{quasi form 1}
&=\lim\limits_{n\rightarrow\infty}\underset{t^{\prime\prime}\in T}{\sup}\|Q_n\pi(e(P, t, t^\prime))\delta_{t^{\prime\prime}} - \pi(e(P,t,t^\prime))Q_n\delta_{t^{\prime\prime}}\|_2.
\end{align}
Set $\epsilon>0$. Then there exists $m\in\N$ such that
\begin{align}
\notag
\underset{t^{\prime\prime}\in T}{\sup}\|Q_n&\pi(e(P, t, t^\prime))\delta_{t^{\prime\prime}} - \pi(e(P,t,t^\prime))Q_n\delta_{t^{\prime\prime}}\|_2 \\
\label{quasi form 2}
&\leq \epsilon+ \underset{t^{\prime\prime}\in T\sqcap B_m(0)}{\sup}\|Q_n\pi(e(P, t, t^\prime))\delta_{t^{\prime\prime}} - \pi(e(P,t,t^\prime))Q_n\delta_{t^{\prime\prime}}\|_2.
\end{align}
Since the supremum in \eqref{quasi form 2} is over a finite set, we can use this bound in \eqref{quasi form 1} and then interchange the limit and supremum on the right-hand side to obtain
\begin{align}\label{quasi form 3}
\lim\limits_{n\rightarrow\infty}\|Q_n&\pi(e(P, t, t^\prime)) - \pi(e(P,t,t^\prime))Q_n\|_{\textnormal{op}}\nonumber\\
&\leq\epsilon + \underset{t^{\prime\prime}\in T\sqcap B_m(0)}{\sup}\lim_{n\rightarrow\infty}\|Q_n\pi(e(P, t, t^\prime))\delta_{t^{\prime\prime}} - \pi(e(P,t,t^\prime))Q_n\delta_{t^{\prime\prime}}\|_2.
\end{align}
 Choose $n \in \N$ large enough that $T \sqcap B_n(0)$ includes all tiles that intersect the patch $B_{m+\|x(t^\prime)-x(t)\|}(0)$. Then we compute that
\begin{align*}
&\underset{t'' \in T\sqcap B(0,m)}{\sup}\|Q_n\pi(e(P, t, t^\prime))\delta_{t''} - \pi(e(P,t,t^\prime))Q_n\delta_{t''}\|_2\\
&= \begin{cases}
\underset{t'' \in T\sqcap B(0,m)}{\sup}\|Q_n\delta_{t'''} - \pi(e(P,t,t^\prime))\delta_{t''}\|_2&\text{ if } P-x(t)\subset T-x(t'')\\&\hspace{0.5cm}\text{and }x(t''')=x(t'')+(x(t')-x(t))\\ 
0 &\text{ otherwise} \\
\end{cases}\\
&=\begin{cases}
\underset{t'' \in T\sqcap B(0,m)}{\sup}\|\delta_{t'''} - \delta_{t'''} \|_2&\text{ if } P-x(t) \subset T-x(t'')\\&\hspace{0.5cm}\text{and }x(t''')=x(t'')+(x(t')-x(t))\\
0 &\text{ otherwise} \\
\end{cases}\\
&=0.
\end{align*}
Since this holds for all sufficiently large $n\in\N$, we see that
\begin{equation}\label{quasi form 4}
\underset{t^{\prime\prime}\in T\sqcap B_m(0)}{\sup}\lim_{n\rightarrow\infty}\|Q_n\pi(e(P, t, t^\prime))\delta_{t^{\prime\prime}} - \pi(e(P,t,t^\prime))Q_n\delta_{t^{\prime\prime}}\|_2=0.
\end{equation}
Using \eqref{quasi form 4} in \eqref{quasi form 3}, we obtain 
\[\lim\limits_{n\rightarrow\infty}\|Q_n\pi(e(P, t, t^\prime)) - \pi(e(P,t,t^\prime))Q_n\|_{\textnormal{op}}\leq\epsilon,\]
which gives the desired result.
\end{proof}

\begin{remark}
In \cite{Whit}, Kellendonk's construction of a tiling $C^*$-algebra was extended to include tilings with \textit{infinite rotational symmetry} whose prototiles appear in infinitely many orientations. In this case, a unitary is added to the densely spanning set $\EE$ defined in \eqref{ez-collection} to keep track of the rotational deviation of each tile from the standard orientation of the prototile. This unitary can easily be incorporated into the proof above to show that the $C^*$-algebras of tilings with infinite rotational symmetry are also quasidiagonal.
\end{remark}

\bibliographystyle{plain}
\bibliography{References/references}

\end{document}